\documentclass[reqno]{amsproc}
\usepackage{amssymb}
\usepackage{euscript}
\usepackage{mathrsfs} 
\usepackage{units}   
\usepackage{color}

\makeatletter
\@namedef{subjclassname@2010}{%
\textup{2010} Mathematics Subject Classification}
\makeatother

\newtheorem{thm}{Theorem}
\newtheorem{cor}[thm]{Corollary}
\newtheorem{lem}[thm]{Lemma}
\newtheorem{pro}[thm]{Proposition}

\newtheorem*{thm*}{Theorem}

\theoremstyle{remark}
\newtheorem{rem}[thm]{Remark}

\theoremstyle{definition}
\newtheorem{exa}[thm]{Example}

\DeclareMathOperator{\D}{d\hspace{-0.25ex}}

\DeclareMathOperator{\dzii}{{\mathsf{Chi}}}
\DeclareMathOperator{\E}{e}

\DeclareMathOperator{\paa}{{\mathsf{par}}}

\newcommand*{\ascr}{\mathscr A}
\newcommand*{\ats}{{\ascr \otimes \varSigma}}
\newcommand*{\borel}[1]{{\mathfrak B}(#1)}
\newcommand*{\bscr}{\mathscr B}
\newcommand*{\cbb}{\mathbb C}
\newcommand*{\esf}{\mathsf{E}}

\newcommand*{\dz}[1]{{\EuScript D}(#1)}
\newcommand*{\dzi}[1]{\dzii(#1)}

\newcommand*{\dzn}[1]{{\EuScript D}^\infty(#1)}

\newcommand*{\Ge}{\geqslant}
\newcommand*{\gammab}{\boldsymbol \gamma}
\newcommand*{\hh}{\mathcal H}
\newcommand*{\hscr}{\mathscr H}
\newcommand*{\hsf}{\mathsf h}

\newcommand*{\is}[2]{\langle#1,#2\rangle}
\newcommand*{\jd}[1]{\EuScript N(#1)}
\newcommand*{\kk}{\mathcal K}

\newcommand*{\lambdab}{{\boldsymbol\lambda}}

\newcommand*{\Le}{\leqslant}

\newcommand*{\nbb}{\mathbb N}

\newcommand*{\ogr}[1]{\boldsymbol B(#1)}
\newcommand*{\ob}[1]{{\EuScript R}(#1)}
\newcommand*{\obn}[1]{{\EuScript R}^{\infty}(#1)}
\newcommand*{\pa}[1]{\paa(#1)}
\newcommand*{\phii}[1]{\phi^{-1}_{\bullet}(\{#1\})}
\newcommand*{\pscr}{{\mathscr P}}

\newcommand*{\rbb}{\mathbb R}
\newcommand*{\rbop}{{\overline{\mathbb R}_+}}
\newcommand*{\slam}{S_{\boldsymbol \lambda}}
\newcommand*{\smalloplus}{\raise0pt
\hbox{$\scriptscriptstyle \oplus$}}

\newcommand*{\supp}[1]{\mathrm{supp}\,#1}
\newcommand*{\tcal}{{\mathscr T}}

\newcommand*{\zbb}{\mathbb Z}

\begin{document}
   \title[Unbounded Subnormal Composition
Operators in $L^2$-Spaces]{Unbounded Subnormal
Composition Operators in $L^2$-Spaces}
   \author[P.\ Budzy\'{n}ski]{Piotr Budzy\'{n}ski}
   \address{Katedra Zastosowa\'{n} Matematyki,
Uniwersytet Rolniczy w Krakowie, ul.\ Balicka 253c,
PL-30198 Krak\'ow}
   \email{piotr.budzynski@ur.krakow.pl}
   \author[Z.\ J.\ Jab{\l}o\'nski]{Zenon Jan
Jab{\l}o\'nski}
   \address{Instytut Matematyki,
Uniwersytet Jagiello\'nski, ul.\ \L ojasiewicza 6,
PL-30348 Kra\-k\'ow, Poland}
   \email{Zenon.Jablonski@im.uj.edu.pl}
   \author[I.\ B.\ Jung]{Il Bong Jung}
   \address{Department of Mathematics, Kyungpook National University, Daegu
702-701, Korea}
   \email{ibjung@knu.ac.kr}
   \author[J.\ Stochel]{Jan Stochel}
\address{Instytut Matematyki, Uniwersytet
Jagiello\'nski, ul.\ \L ojasiewicza 6, PL-30348
Kra\-k\'ow, Poland}
   \email{Jan.Stochel@im.uj.edu.pl}
   \thanks{The research of the first author
was supported by the NCN (National Science
Center) grant DEC-2011/01/D/ST1/05805. The
research of the third author was supported by
Basic Science Research Program through the
National Research Foundation of Korea (NRF)
funded by the Ministry of Education, Science and
Technology (2012-008590).}
   \begin{abstract}
A criterion for subnormality of unbounded composition
operators in $L^2$-spaces, written in terms of
measurable families of probability measures satisfying
the so-called consistency condition, is established.
It becomes a new characterization of subnormality in
the case of bounded composition operators.
Pseudo-moments of a measurable family of probability
measures that satisfies the consistency condition are
proved to be given by the Radon-Nikodym derivatives
which appear in Lambert's characterization of bounded
composition operators. A criterion for subnormality of
composition operators induced by matrices is provided.
The question of subnormality of composition operators
over discrete measure spaces is studied. Two new
classes of subnormal composition operators over
discrete measure spaces are introduced. A recent
criterion for subnormality of weighted shifts on
directed trees by the present authors is essentially
improved in the case of rootless directed trees and
nonzero weights by dropping the assumption of density
of $C^\infty$-vectors in the underlying
$\ell^2$-space.
   \end{abstract}
   \maketitle
   \section{PRELIMINARIES}
   \subsection{Introduction}
In 1950 Halmos introduced the notion of a bounded
subnormal operator and gave its first characterization
(cf.\ \cite{hal1}), which was successively simplified
by Bram \cite{bra}, Embry \cite{emb} and Lambert
\cite{lam}. Neither of them is true for unbounded
operators (see \cite{Con} and \cite{StSz1,StSz2,StSz3}
for foundations of the theory of bounded and unbounded
subnormal operators). The only known general
characterizations of subnormality of unbounded
operators refer to semispectral measures or elementary
spectral measures (cf.\ \cite{bis,foi,FHSz}). They
seem to be useless in the context of particular
classes of operators. The other known criteria for
subnormality (with the exception of \cite{Sz4})
require the operator in question to have an invariant
domain (cf.\ \cite{StSz2,StSz,c-s-sz,Al-V}). In this
paper we give a criterion for subnormality of densely
defined composition operators (in $L^2$-spaces) with
no additional restrictions.

Composition operators occur in many areas of
mathematics. They play a vital role in ergodic theory
and functional analysis. The theory of bounded
composition operators seems to be well-developed (see
\cite{sin,nor,wh,ha-wh,lam1,lam2,di-ca,emb-lam3,sin-man,bu-ju-la,Bu-St1,Bu-St2};
see also \cite{emb-lam2,ml,sto,da-st,2xSt} for
particular classes of such operators). As opposed to
the bounded case, the theory of unbounded composition
operators is at a rather early stage of development.
There are few papers concerning this issue. Some basic
facts about unbounded composition operators can be
found in \cite{ca-hor,jab,b-j-j-sC,Bu}. To the best of
our knowledge, there is no paper concerning the
question of subnormality of (general) unbounded
composition operators. A criterion for subnormality of
certain composition operators built over directed
trees can be deduced from \cite[Theorem
5.1.1]{b-j-j-sA} via \cite[Lemma 4.3.1]{j-j-s0}.
However, it requires the operator in question to have
dense set of $C^\infty$-vectors. The reason for this
is that its proof is based on an approximation
technique derived from \cite[Theorem 21]{c-s-sz} in
which the invariance of the domain plays an essential
role. In other words, this technique could not be
applied when looking for a general criterion for
subnormality of unbounded composition operators. On
the other hand, Lambert's characterization of bounded
subnormal composition operators, which is written in
terms of the Radon-Nikodym derivatives
$\{\hsf_{\phi^n}\}_{n=0}^\infty$ (cf.\ \eqref{hfi}),
is no longer valid in the unbounded case (see
\cite[Theorem 4.3.3]{j-j-s0} and \cite[Section
11]{b-j-j-sC}).

In the present paper we give the first ever criterion
for subnormality of unbounded composition operators,
which becomes a new characterization of subnormality
in the bound\-ed case. It states that if an injective
densely defined composition operator has a measurable
family of probability measures that satisfies the
so-called consistency condition, then it is subnormal
(cf.\ Theorem \ref{glowne}). The consistency condition
appeals to the Radon-Nikodym derivative $\hsf_{\phi}$.
To invent it, we revisit the Lambert's construction of
a quasinormal extension of a bounded subnormal
composition operator which is given in \cite{lam2}.
Surprisingly, the pseudo-moments of a measurable
family of probability measures that satisfies the
consistency condition are given by the Radon-Nikodym
derivatives $\{\hsf_{\phi^n}\}_{n=0}^\infty$ (cf.\
Theorem \ref{sms}).

The paper consists of three parts. The first contains
some background material concerning Stieltjes moment
sequences, composition operators and conditional
expectation (with respect to $\phi^{-1}(\ascr)$). The
second consists of four sections. Section \ref{acrit}
provides the main criterion for subnormality of
unbounded composition operators (cf.\ Theorem
\ref{glowne}). That this criterion becomes a
characterization in the bounded case is justified in
Section \ref{tbc}. The consistency condition is
investigated in Section \ref{tcc}. In particular, it
is proved that the consistency condition behaves well
with respect to the operation of taking powers of
composition operators (cf.\ Proposition \ref{niema}).
Section \ref{scc} deals with the strong consistency
condition, a variant of the consistency condition
which does not appeal to conditional expectation. It
is shown that in the bounded case the strong
consistency condition is equivalent to requiring that
the Radon-Nikodym derivatives
$\{\hsf_{\phi^n}\}_{n=0}^\infty$ be invariant for the
operator of conditional expectation (cf.\ Proposition
\ref{eqscc}). The third part of the paper deals with
particular classes of bounded or unbounded composition
operators. In Section \ref{tmc} we prove that
composition operators in $L^2(\mu_\gamma)$ induced by
normal $\kappa\times\kappa$ matrices are subnormal,
where $\mu_\gamma$ is a Borel measure on $\rbb^\kappa$
with a density function given by an entire function
with nonnegative Taylor coefficients at $0$ (cf.\
Theorem \ref{matrical}). The question of subnormality
of composition operators in $L^2$-spaces over discrete
measure spaces is reexamined in Section \ref{dc} (cf.\
Theorem \ref{glownedis}). A model for such operators
with injective symbols is established in Remark
\ref{model}. In Section \ref{lcs} we introduce a
``local consistency technique'' which is new even in
the bounded case (cf.\ Lemma \ref{dcc}). It enables us
to deduce subnormality of a composition operator in an
$L^2$-space over a discrete measure space from the
Stieltjes determinacy of the Radon-Nikodym derivatives
$\{\hsf_{\phi^{n+1}}\}_{n=0}^\infty$ (cf.\ Theorem
\ref{detimsub1}). In Section \ref{asefp} we use the
``local consistency technique'' to model subnormal
composition operators induced by a transformation
which has only one essential fixed point. Section
\ref{fullk} deals with the question of subnormality of
a class of composition operators over directed trees
with finite constant valence on generations. In this
case, even though the operator of conditional
expectation is far from being the identity, we can use
the strong consistency condition. This enables us to
characterize subnormality within this class by using
Lambert's condition (cf.\ Theorem \ref{ehphin}), the
phenomenon known so far for unilateral and bilateral
injective weighted shifts only. In Section \ref{wsrdc}
we show that Theorem 5.1.1 of \cite{b-j-j-sA}, which
is a criterion for subnormality of a weighted shift on
a directed tree, remains valid if the assumption that
$C^\infty$-vectors are dense is dropped, provided the
weights are nonzero and the tree is rootless and
leafless (cf.\ Theorem \ref{wsi}).

The paper is concluded with appendices concerning
composition operators induced by roots of the
identity, symmetric composition operators and
orthogonal sums of composition operators.
   \subsection{Prerequisites}
We write $\zbb$, $\rbb$ and $\cbb$ for the sets
of integers, real numbers and complex numbers,
respectively. We denote by $\nbb$, $\zbb_+$ and
$\rbb_+$ the sets of positive integers,
nonnegative integers and nonnegative real
numbers, respectively. Set $\rbop = \rbb_+ \cup
\{\infty\}$. In what follows, we adhere to the
convention that $0 \cdot \infty = \infty \cdot 0
= 0$, $\frac{1}{0} = \infty$ and $\frac{0}{0}=1$.
If $\zeta\colon X \to \rbop$ is a function on a
set $X$, then we put $\{\zeta=0\}=\{x\in X\colon
\zeta(x)=0\}$ and $\{\zeta> 0\}=\{x\in X\colon
\zeta(x) > 0\}$. Given subsets $\varDelta,
\varDelta_n$ of $X$, $n\in \nbb$, we write
$\varDelta_n \nearrow \varDelta$ as $n\to \infty$
if $\varDelta_n \subseteq \varDelta_{n+1}$ for
every $n\in \nbb$ and $\varDelta =
\bigcup_{n=1}^\infty \varDelta_n$. The
characteristic function of a subset $\varDelta$
of $X$ is denoted by $\chi_\varDelta$. The symbol
$\sigma(\pscr)$ is reserved for the
$\sigma$-algebra generated by a family $\pscr$ of
subsets of $X$. All measures considered in this
paper are assumed to be positive. Given two
measures $\mu$ and $\nu$ on the same
$\sigma$-algebra, we write $\mu \ll \nu$ if $\mu$
is absolutely continuous with respect to $\nu$;
then $\frac{\D\mu}{\D\nu}$ stands for the
Radon-Nikodym derivative of $\mu$ with respect to
$\nu$ (provided it exists). We shall abbreviate
the expressions ``almost everywhere with respect
to $\mu$'' and ``for $\mu$-almost every $x$'' to
``a.e.\ $[\mu]$'' and ``for $\mu$-a.e.\ $x$'',
respectively. As usual, $L^2(\mu)$ stands for the
Hilbert space of all square integrable (with
respect to a measure $\mu$) complex functions on
$X$. If $\mu$ is the counting measure on $X$,
then we write $\ell^2(X)$ in place of $L^2(\mu)$.
The $\sigma$-algebra of all Borel sets of a
topological space $Z$ is denoted by $\borel{Z}$.
In what follows $\delta_t$ stands for the Borel
probability measure on $\rbb_+$ concentrated at
$t\in \rbb_+$. The closed support of a finite
Borel measure $\nu$ on $\rbb_+$ is denoted by
$\supp \nu$.

Now we state an auxiliary lemma which follows
from \cite[Proposition I-6-1]{Nev} and
\cite[Theorem 1.3.10]{Ash}.
   \begin{lem} \label{2miary}
Let $\pscr$ be a semi-algebra of subsets of a set
$X$ and $\mu_1, \mu_2$ be measures on
$\sigma(\pscr)$ such that $\mu_1(\varDelta) =
\mu_2(\varDelta)$ for all $\varDelta \in \pscr$.
Suppose there exists a sequence
$\{\varDelta_n\}_{n=1}^\infty \subseteq \pscr$
such that $\varDelta_n \nearrow X$ as $n\to
\infty$ and $\mu_1(\varDelta_k) < \infty$ for
every $k \in \nbb$. Then $\mu_1 = \mu_2$.
   \end{lem}
From now on, we write $\int_0^\infty$ instead of
$\int_{\rbb_+}$. A sequence $\{a_n\}_{n=0}^\infty
\subseteq \rbb$ is said to be a {\em Stieltjes
moment sequence} if there exists a Borel measure
$\nu$ on $\rbb_+$, called a {\em representing
measure} of $\{a_n\}_{n=0}^\infty$, such that
   \begin{align*}
a_n = \int_0^\infty s^n \nu(\D s), \quad n \in
\zbb_+.
   \end{align*}
If such a $\nu$ is unique, then
$\{a_n\}_{n=0}^\infty$ is called {\em
determinate}. A Borel measure $\nu$ on $\rbb_+$
is said to be {\em determinate} if all its
moments $\int_0^\infty s^n \nu(\D s)$, $n\in
\zbb_+$, are finite and the Stieltjes moment
sequence $\{\int_0^\infty s^n \nu(\D
s)\}_{n=0}^\infty$ is determinate. Sequences or
measures which are not determinate are called
{\em indeterminate}. Recall that any finite Borel
measure on $\rbb_+$ with compact support is
determinate (cf.\ \cite{fug}). Another criterion
for determinacy can be deduced from the M. Riesz
theorem (cf.\ \cite{fug}) and \cite[Lemma
2.2.5]{j-j-s0}.
   \begin{align} \label{mriesz}
   \begin{minipage}{70ex}
A Borel measure $\nu$ on $\rbb_+$ whose all
moments are finite and $\nu(\{0\})=0$ is
determinate if and only if $\cbb[t]$ is dense in
$L^2((1+t^2) \nu(\D t))$,
   \end{minipage}
   \end{align}
where $\cbb[t]$ stands for the ring of all
complex polynomials in real variable $t$. We
refer the reader to \cite[Proposition
1.3]{ber-th} for a full characterization of
determinacy. The following useful lemma is
related to \cite[Exercise 23, Chapter 3]{Rud}. We
include its proof to keep the exposition as
self-contained as possible.
   \begin{lem} \label{granica}
If $\{a_n\}_{n=0}^\infty \subseteq (0,\infty)$ is a
Stieltjes moment sequence with a representing measure
$\nu$, then the sequence
$\big\{\frac{a_{n+1}}{a_n}\big\}_{n=0}^\infty$ is
monotonically increasing and
   \begin{align*}
\sup_{n \in \zbb_+} \frac{a_{n+1}}{a_n} = \sup(\supp
\nu).
   \end{align*}
   \end{lem}
   \begin{proof}
Applying the Cauchy-Schwarz inequality, we deduce that
the sequence
$\big\{\frac{a_{n+1}}{a_n}\big\}_{n=0}^\infty$ is
monotonically increasing. This implies that
   \begin{align*}
\sup_{n\in\zbb_+} \frac{a_{n+1}}{a_n} = \lim_{n\to
\infty} \frac{a_{n+1}}{a_n} \overset{(\dag)}=
\lim_{n\to \infty} \sqrt[n]{a_n} \overset{(\ddag)}=
\sup(\supp \nu),
   \end{align*}
where $(\dag)$ and $(\ddag)$ may be inferred from
\cite[Lemma 2.2]{sto-c} (with $\varOmega = \zbb_+$,
$A(n)=n+1$ and $\phi(n)=a_n$) and \cite[Exercise 4,
Chapter 3]{Rud}, respectively.
   \end{proof}
Let $A$ be an operator in a complex Hilbert space
$\hh$ (all operators considered in this paper are
linear). Denote by $\dz{A}$, $\jd{A}$, $\ob{A}$ and
$A^*$ the domain, the kernel, the range and the
adjoint of $A$ (in case it exists) respectively. Set
$\dzn{A} = \bigcap_{n=0}^\infty \dz{A^n}$ with
$A^0=I$, where $I=I_{\hh}$ stands for the identity
operator on $\hh$. Members of $\dzn{A}$ are called
{\em $C^\infty$-vectors} of $A$. A vector subspace
$\mathcal{E}$ of $\dz{A}$ is called a {\em core} for
$A$ if $\mathcal{E}$ is dense in $\dz{A}$ with respect
to the graph norm of $A$. If $A$ is closed and densely
defined, then $A$ has a (unique) {\em polar
decomposition} $A=U|A|$, where $U$ is a partial
isometry on $\hh$ such that the kernels of $U$ and $A$
coincide and $|A|$ is the square root of $A^*A$ (cf.\
\cite[Section 8.1]{b-s}). Given two operators $A$ and
$B$ in $\hh$, we write $A \subseteq B$ if $\dz{A}
\subseteq \dz{B}$ and $Af=Bf$ for all $f\in \dz{A}$.
In what follows $\ogr{\hh}$ stands for the
$C^*$-algebra of all bounded operators in $\hh$ whose
domains are equal to $\hh$. A densely defined operator
$N$ in $\hh$ is said to be {\em normal} if $N$ is
closed and $N^*N=NN^*$ (or equivalently if and only if
$\dz{N}=\dz{N^*}$ and $\|Nf\|=\|N^*f\|$ for all $f \in
\dz{N}$, see \cite{b-s}). We say that a densely
defined operator $S$ in $\hh$ is {\em subnormal} if
there exist a complex Hilbert space $\kk$ and a normal
operator $N$ in $\kk$ such that $\hh \subseteq \kk$
(isometric embedding), $\dz{S} \subseteq \dz{N}$ and
$Sf = Nf$ for all $f \in \dz{S}$. Since powers of a
normal operator are normal, we see that any densely
defined power of a subnormal operator is still
subnormal. The members of the next class are related
to subnormal operators. A closed densely defined
operator $A$ in $\hh$ is said to be {\em quasinormal}
if $U |A| \subseteq |A|U$, where $A=U|A|$ is the polar
decomposition of $A$. Recall that quasinormal
operators are subnormal (see \cite[Theorem 1]{bro} and
\cite[Theorem 2]{StSz2}). The reverse implication does
not hold in general. It is well-known that if $S$ is
subnormal, then $\{\|S^n f\|^2\}_{n=0}^\infty$ is a
Stieltjes moment sequence for every $f \in \dzn{S}$
(see \cite[Proposition 3.2.1]{b-j-j-sA}). The converse
does not always hold, even if $\dzn{S}$ is dense in
$\hh$ (see \cite[Section 3.2]{b-j-j-sA}).

Let $(X,\ascr, \mu)$ be a $\sigma$-finite measure
space. A map from $X$ to $X$ is called a {\em
transformation} of $X$. Let $\phi$ be an
$\ascr$-{\em measurable} transformation of $X$,
i.e., $\phi^{-1}(\varDelta) \in \ascr$ for all
$\varDelta \in \ascr$. Denote by $\mu\circ
\phi^{-1}$ the measure on $\ascr$ given by
$\mu\circ
\phi^{-1}(\varDelta)=\mu(\phi^{-1}(\varDelta))$
for $\varDelta \in \ascr$. We say that $\phi$ is
{\em nonsingular} if $\mu\circ \phi^{-1}$ is
absolutely continuous with respect to $\mu$. The
following is easily seen to be true.
   \begin{align} \label{compos}
   \begin{minipage}{68ex}
If $\phi$ is nonsingular, $Y$ is a nonempty set and
$f, g \colon X \to Y$ are functions such that $f=g$
a.e.\ $[\mu]$, then $f \circ \phi = g \circ \phi$
a.e.\ $[\mu]$.
   \end{minipage}
   \end{align}
Clearly, if $\phi$ is nonsingular, then the map
$C_\phi\colon L^2(\mu) \supseteq \dz{C_\phi}\to
L^2(\mu)$ given by
   \begin{align*}
\dz{C_\phi} = \{f \in L^2(\mu) \colon f \circ \phi \in
L^2(\mu)\} \text{ and } C_\phi f = f \circ \phi \text{
for } f \in \dz{C_\phi},
   \end{align*}
is well-defined (and linear); the converse is
true as well. Such $C_\phi$ is called a {\em
composition operator} with a {\em symbol} $\phi$
(or {\em induced} by $\phi$). Note that every
composition operator is closed (cf.\
\cite[Proposition 3.2]{b-j-j-sC}). If $\phi$ is
nonsingular, then by the Radon-Nikodym theorem
there exists a unique (up to sets of measure
$\mu$ zero) $\ascr$-measurable function $\mathsf
\hsf_\phi \colon X \to \rbop$ such that
   \begin{align} \label{hfi}
\mu\circ \phi^{-1}(\varDelta) = \int_\varDelta
\hsf_\phi \D\mu, \quad \varDelta \in \ascr.
   \end{align}
Recall that $\dz{C_\phi}=L^2(\mu)$ if and only if
$\hsf_{\phi} \in L^\infty(\mu)$; moreover, if
$\hsf_{\phi} \in L^\infty(\mu)$, then $C_{\phi}
\in \ogr{L^2(\mu)}$ and $\|C_\phi\|^2 =
\|\hsf_{\phi} \|_{L^\infty(\mu)}$ (see e.g.,
\cite[Theorem 1]{nor}). It is well-known that
(cf.\ \cite[Lemma 6.1]{ca-hor})
   \begin{align} \label{gokr}
   \begin{minipage}{76ex}
if $\phi$ is nonsingular, then $C_{\phi}$ is densely
defined if and only if $\hsf_{\phi} < \infty$ a.e.\
$[\mu]$.
   \end{minipage}
   \end{align}
Note also that (cf.\ \cite[Proposition
6.5]{b-j-j-sC})
   \begin{align} \label{hfi0}
   \begin{minipage}{50ex}
if $\phi$ is nonsingular, then $\hsf_{\phi} \circ
\phi > 0$ a.e.\ $[\mu]$.
   \end{minipage}
   \end{align}
The following fact is patterned on the integral
formula due to Embry and Lambert (cf.\ \cite[p.\
168]{Em-Lam-center}).
   \begin{pro}\label{dzielenie}
Let $(X,\ascr,\mu)$ be a $\sigma$-finite measure
space and $\phi$ be a nonsingular transformation
of $X$ such that $\hsf_{\phi} < \infty$ a.e.\
$[\mu]$. Then
   \begin{align} \label{klucz}
\int_X \frac{f\circ \phi}{\hsf_\phi \circ \phi}
\D \mu = \int_{\{\hsf_{\phi}>0\}} f \D\mu \text{
for any $\ascr$-measurable function $f\colon X
\to \rbop$.}
   \end{align}
   \end{pro}
   \begin{proof}
Apply \eqref{hfi0} and the measure transport
theorem (cf.\ \cite[Theorem 1.6.12]{Ash}) to the
restriction of $\phi$ to a set of $\mu$-full
measure on which $\hsf_\phi \circ \phi$ is
positive.
   \end{proof}
Given $n \in \nbb$, we denote by $\phi^n$ the
$n$-fold composition of $\phi$ with itself;
$\phi^0$ is the identity transformation
$\mathrm{id}_X$ of $X$. We write
$\phi^{-n}(\varDelta) = (\phi^n)^{-1}(\varDelta)$
for $\varDelta \in \ascr$ and $n \in \zbb_+$. If
$\phi$ is nonsingular and $n \in \zbb_+$, then
$\phi^n$ is nonsingular and thus $\hsf_{\phi^n}$
makes sense. It is clear that $\hsf_{\phi^0}=1$
a.e.\ $[\mu]$.

The question of when a (not necessarily densely
defined) composition operator is bounded from
below has an explicit answer.
   \begin{pro}\label{ogrodd}
Let $(X,\ascr,\mu)$ be a $\sigma$-finite measure space
and $\phi$ be a nonsingular transformation of $X$. If
$c$ is a positive real number, then the following two
conditions are equivalent{\em :}
   \begin{enumerate}
   \item[(i)] $\|C_{\phi}f\| \Ge c \|f\|$ for
every $f\in \dz{C_{\phi}}$,
   \item[(ii)] $\hsf_{\phi} \Ge c^2$ a.e.\ $[\mu]$.
   \end{enumerate}
   \end{pro}
   \begin{proof}
If (i) holds, then
   \begin{align} \label{noc}
\int_X (\hsf_{\phi} - c^2) |f|^2 \D \mu \Ge 0, \quad
f\in \dz{C_{\phi}}.
   \end{align}
Since $\mu$ is $\sigma$-finite, there exists a
sequence $\{X_n\}_{n=1}^\infty \subseteq \ascr$ such
that $\mu(X_k) < \infty$ for every $k\Ge 1$, and $X_n
\nearrow X$ as $n\to \infty$. Set $Y_n = X_n \cap
\{x\in X \colon \hsf_{\phi} \Le n\}$ for $n\Ge 1$. Fix
$n\Ge 1$. It is easily seen that $\chi_{\varDelta} \in
\dz{C_{\phi}}$ for any $\varDelta \in \ascr$ such that
$\varDelta \subseteq Y_n$. Substituting
$f=\chi_{\varDelta}$ into \eqref{noc}, we get
$\int_{Y_n} |\hsf_{\phi} - c^2| \D \mu < \infty$ and
$\int_{\varDelta} (\hsf_{\phi} - c^2) \D \mu \Ge 0$
for every $\varDelta \in \ascr$ such that $\varDelta
\subseteq Y_n$. This implies that $\hsf_{\phi} - c^2
\Ge 0$ a.e.\ $[\mu]$ on $Y_n$. Since $Y_k \nearrow Y$
as $k\to \infty$, where $Y=\{x\in X \colon
\hsf_{\phi}(x) < \infty\}$, we conclude that
$\hsf_{\phi} \Ge c^2$ a.e.\ $[\mu]$. The reverse
implication is obvious.
   \end{proof}
Now we collect some properties of conditional
expectation that are needed in this paper. Set
$\phi^{-1}(\ascr)=\{\phi^{-1}(\varDelta)\colon
\varDelta \in \ascr\}$. Suppose $\phi$ is a
nonsingular transformation of $X$ such that $\hsf_\phi
< \infty$ a.e.\ $[\mu]$. Then the measure
$\mu|_{\phi^{-1}(\ascr)}$ is $\sigma$-finite (cf.\
\cite[Proposition 3.2]{b-j-j-sC}), and thus by the
Radon-Nikodym theorem, for every $\ascr$-measurable
function $f\colon X \to \rbop$ there exists a unique
(up to sets of measure $\mu$ zero)
$\phi^{-1}(\ascr)$-measurable
function\footnote{\;Recall the well-known fact that a
function $v\colon X \to \rbop$ is
$\phi^{-1}(\ascr)$-measurable if and only if there
exists an $\ascr$-measurable function $u\colon X \to
\rbop$ such that $v = u\circ \phi$.} $\esf(f)\colon X
\to \rbop$ such that for every $\ascr$-measurable
function $g\colon X \to \rbop$,
   \begin{align} \label{CE-3}
\int_X g \circ \phi \cdot f \D\mu= \int_X g \circ \phi
\cdot \esf(f) \D\mu.
   \end{align}
We call $\esf(f)$ the {\em conditional expectation} of
$f$ with respect to $\phi^{-1}(\ascr)$ (see \cite{Rao}
and \cite{b-j-j-sC} for more information). For
simplicity we do not make the dependence of $\esf(f)$
on $\phi$ explicit. It is well-known that
   \begin{align} \label{CE-2}
\text{if $0\Le f_n \nearrow f$ and $f_n, f$ are
$\ascr$-measurable, then $\esf(f_n) \nearrow
\esf(f)$,}
   \end{align}
where $g_n \nearrow g$ means that for $\mu$-a.e.\
$x\in X$, the sequence $\{g_n(x)\}_{n=1}^\infty$
is monotonically increasing and convergent to
$g(x)$. Note that for every $\ascr$-measurable
function $u \colon X \to \rbop$ there exists a
unique (up to sets of measure $\mu$ zero)
$\ascr$-measurable function $g\colon X \to \rbop$
such that $u \circ \phi = g \circ \phi$ a.e.\
$[\mu]$ and $g=0$ a.e.\ $[\mu]$ on $X \setminus
\varOmega_{\phi}$, where $\varOmega_{\phi}:=
\{\hsf_{\phi} > 0\}$. Indeed, by the measure
transport theorem, we have
$\int_{\phi^{-1}(\varDelta)} u \circ \phi \D \mu
= \int_{\varDelta} u \, \hsf_{\phi} \D \mu =
\int_{\phi^{-1}(\varDelta)} (u \,
\chi_{\varOmega_{\phi}}) \circ \phi \D \mu$ for
all $\varDelta \in \ascr$, and thus $g=u \,
\chi_{\varOmega_{\phi}}$ has the required
properties (because $\mu|_{\phi^{-1}(\ascr)}$ is
$\sigma$-finite). A similar argument yields the
uniqueness of $g$. As a consequence, if $f\colon
X \to \rbop$ is $\ascr$-measurable function, then
$\esf(f) = g\circ \phi$ a.e.\ $[\mu]$ with some
$\ascr$-measurable function $g\colon X \to \rbop$
such that $g=0$ a.e.\ $[\mu]$ on $X \setminus
\varOmega_{\phi}$. Set $\esf(f) \circ \phi^{-1} =
g$ a.e.\ $[\mu]$. By the above discussion (see
also \cite{ca-hor}), this definition is correct
and
   \begin{align} \label{fifi}
(\esf(f) \circ \phi^{-1})\circ \phi = \esf(f) \quad
\text{a.e.\ $[\mu|_{\phi^{-1}(\ascr)}]$.}
   \end{align}
In particular, the following holds.
   \begin{align} \label{hf0-0}
   \begin{minipage}{70ex}
If $\phi$ is a nonsingular transformation of $X$ such
that $0 < \hsf_{\phi} < \infty$ a.e.\ $[\mu]$ and
$u,g\colon X \to \rbop$ are $\ascr$-measurable
functions such that $u\circ \phi = g \circ \phi$ a.e.\
$[\mu]$, then $u = g$ a.e.\ $[\mu]$.
   \end{minipage}
   \end{align}
The reader should be aware of the fact that
$\esf(\chi_{X}) = 1$ a.e.\ $[\mu]$ and
   \begin{align} \label{jedynka}
\esf(\chi_{X}) \circ \phi^{-1} = \chi_{\{\hsf_{\phi} >
0\}} \text{ a.e.\ $[\mu]$.}
   \end{align}
   \section{A CONSISTENCY TECHNIQUE IN SUBNORMALITY}
   \subsection{\label{acrit}The general case}
Let $(X, \ascr)$ and $(T,\varSigma)$ be measurable
spaces and $P\colon X \times \varSigma \to [0,1]$ be
an {\em $\ascr$-measurable family of probability
measures}, i.e.,
   \begin{enumerate}
   \item[(i)] the set-function  $P(x,\cdot)$  is a
probability measure for every $x \in X$,
   \item[(ii)] the function $P(\cdot,\sigma)$ is
$\ascr$-measurable for every $\sigma \in \varSigma$.
   \end{enumerate}
Denote by $\ats$ the $\sigma$-algebra generated by the
family
   \begin{align*}
\text{$\ascr \boxtimes \varSigma:=\{\varDelta \times
\sigma\colon \varDelta \in \ascr, \, \sigma \in
\varSigma\}$.}
   \end{align*}
Let $\mu\colon \ascr \to \rbop$ be a $\sigma$-finite
measure. Then (cf.\ \cite[Theorem 2.6.2]{Ash}) there
exists a unique measure $\rho$ on $\ats$ such that
   \begin{align} \label{rhoabs}
\rho(\varDelta \times \sigma) = \int_{\varDelta}
P(x,\sigma) \mu(\D x), \quad \varDelta\in \ascr,
\sigma \in \varSigma.
   \end{align}
Such a $\rho$ is automatically $\sigma$-finite.
Moreover, for every $\ats$-measurable function
$f\colon X \times T \to \rbop$,
   \begin{gather} \label{rhoabs0}
\text{the function $X \ni x \to \int_T f (x,t) P(x,\D
t) \in \rbop$ is $\ascr$-measurable}
   \end{gather}
and
   \begin{gather}\label{rhoabs1}
\int_{X\times T} f \D\rho=
\int_X\int_T f (x,t) P(x,\D t) \mu(\D x).
   \end{gather}
Let $\phi$ be an $\ascr$-measurable transformation of
$X$. Define the transformation $\varPhi$ of $X\times
T$ by
   \begin{align} \label{vaPhi}
\varPhi(x,t) = (\phi(x),t), \quad x \in X, t \in T.
   \end{align}
Since the $\sigma$-algebra $\{E \in \ats\colon
\varPhi^{-1}(E) \in \ats\}$ contains $\ascr \boxtimes
\varSigma$, we deduce that the transformation
$\varPhi$ is $\ats$-measurable.

The assumptions we gather below will be used in
further parts of this section.
   \begin{align} \label{standa1}
   \begin{minipage}{72ex}
\hspace{2ex} The triplet $(X,\ascr,\mu)$ is a
$\sigma$-finite measure space, $\phi$ is an
$\ascr$-measurable transformation of $X$,
$(T,\varSigma)$ is a measurable space and $P\colon X
\times \varSigma \to [0,1]$ is an $\ascr$-measurable
family of probability measures. The measure $\rho\colon
\ats \to \rbop$ and the transformation $\varPhi$ of $X
\times T$ are determined by \eqref{rhoabs} and
\eqref{vaPhi}, respectively.
   \end{minipage}
   \end{align}

We begin by establishing the basic formula that links
$\hsf_{\phi}$ and $\hsf_{\varPhi}$.
   \begin{lem} \label{hpx}
Suppose \eqref{standa1} holds. Then the following
assertions are valid.
   \begin{enumerate}
   \item[(i)]
If $\phi$ is nonsingular and $P(x,\cdot) \ll
P(\phi(x),\cdot)$ for $\mu$-a.e.\ $x\in X$, then
$\varPhi$ is nonsingular.
   \item[(ii)] If $\varPhi$ is nonsingular, then so is
$\phi$.
   \item[(iii)] If $\varPhi$ is nonsingular and
$\hsf_\phi < \infty$ a.e.\ $[\mu]$, then
$\hsf_{\varPhi} < \infty$ a.e.\ $[\rho]$ and
   \begin{align}  \label{lewa}
\hsf_\phi(x) \big(\esf(P(\cdot, \sigma)) \circ
\phi^{-1}\big) (x) = \int_{\sigma} \hsf_{\varPhi}(x,t)
P(x,\D t) \text{ for $\mu$-a.e.\ $x \in X$}, \quad
\sigma \in \varSigma.
   \end{align}
   \end{enumerate}
   \end{lem}
   \begin{proof}
(i) Take $E \in \ats$ such that $\rho(E)=0$. Then, by
\eqref{rhoabs1}, we have
   \begin{align*}
\text{$\int_T \chi_E(x,t) P(x,\D t) = 0$ \quad for
$\mu$-a.e.\ $x\in X$.}
   \end{align*}
Hence $\chi_E(x,t) = 0$ for $P(x,\cdot)$-a.e.\ $t \in
T$ and for $\mu$-a.e.\ $x\in X$. Since $\phi$ is
nonsingular, we see that $\chi_E(\phi(x),t) = 0$ for
$P(\phi(x),\cdot)$-a.e.\ $t \in T$ and for $\mu$-a.e.\
$x\in X$. By our assumption, this implies that
$\chi_E(\phi(x),t) = 0$ for $P(x,\cdot)$-a.e.\ $t \in
T$ and for $\mu$-a.e.\ $x\in X$. This combined with
\eqref{rhoabs1} implies that
$\rho(\varPhi^{-1}(E))=0$.

   (ii) If $\varDelta \in \ascr$ is such that
$\mu(\varDelta)=0$, then by \eqref{rhoabs} we have
$\rho(\varDelta \times T)=\mu(\varDelta)=0$ and thus
$\mu(\phi^{-1}(\varDelta)) =
\rho(\varPhi^{-1}(\varDelta \times T))=0$.

   (iii) Applying the measure transport theorem, we
obtain
   \allowdisplaybreaks
   \begin{multline}    \label{prawa1}
\rho(\varPhi^{-1}(\varDelta \times \sigma)) =
\rho(\phi^{-1}(\varDelta) \times \sigma)
\overset{\eqref{rhoabs}} = \int_{\phi^{-1}(\varDelta)}
P(x, \sigma) \mu(\D x)
   \\
\overset{\eqref{CE-3}}= \int_{\phi^{-1}(\varDelta)}
\esf(P(\cdot, \sigma)) \D \mu \overset{\eqref{fifi}} =
\int_{\varDelta} \hsf_\phi \esf(P(\cdot,\sigma)) \circ
\phi^{-1} \D \mu, \quad \varDelta \in \ascr, \sigma \in
\varSigma.
   \end{multline}
Since $\varPhi$ is nonsingular, we infer from
\eqref{rhoabs1} that
   \begin{align}  \label{prawa2}
\rho(\varPhi^{-1}(\varDelta \times \sigma)) =
\int_{\varDelta} \int_{\sigma} \hsf_{\varPhi} (x,t)
P(x,\D t) \mu(\D x), \quad \varDelta \in \ascr, \sigma
\in \varSigma.
   \end{align}
Combining \eqref{prawa1} with \eqref{prawa2} and using
the $\sigma$-finiteness of $\mu$, we get \eqref{lewa}.

Since $\hsf_\phi < \infty$ a.e.\ $[\mu]$, there exists
$\{\varDelta_n\}_{n=1}^\infty \subseteq \ascr$ such
that $\varDelta_n \nearrow X$ as $n\to \infty$,
$\mu(\varDelta_k) < \infty$ and $\hsf_\phi \Le k$
a.e.\ $[\mu]$ on $\varDelta_k$ for every $k\in \nbb$.
Then
   \begin{multline} \label{abc}
\int_{\varDelta_n \times T} \hsf_{\varPhi} \D \rho
\overset{\eqref{rhoabs1}}= \int_{\varDelta_n} \int_T
\hsf_{\varPhi}(x,t) P(x,\D t) \mu(\D x)
   \\
 \overset{\eqref{lewa}} = \int_{\varDelta_n} \hsf_\phi
\esf(P(\cdot,T)) \circ \phi^{-1} \D \mu
\overset{\eqref{jedynka}}= \int_{\varDelta_n} \hsf_\phi
\D \mu \Le n \mu(\varDelta_n), \quad n\in \nbb,
   \end{multline}
which implies that $\hsf_{\varPhi} < \infty$ a.e.\
$[\rho]$ on $\varDelta_n \times T$. Since $\varDelta_n
\times T \nearrow X\times T$ as $n \to \infty$, we
conclude that $\hsf_{\varPhi} < \infty$ a.e.\
$[\rho]$. This completes the proof.
   \end{proof}
Below we introduce the conditions \eqref{ccz} and
\eqref{ccz-1} (cf.\ Lemma \ref{hfi0m} and Theorem
\ref{przygot}) which play a fundamental role in this
paper. We begin by proving that the first moments
$\int_T \zeta(t) P(\cdot, \D t)$ of an
$\ascr$-measurable family $P\colon X \times \varSigma
\to [0,1]$ of probability measures satisfying
\eqref{ccz} cannot vanish on a set of positive measure
$\mu$. We also calculate $\hsf_{\varPhi}$.
   \begin{lem} \label{hfi0m}
Suppose \eqref{standa1} holds, $\phi$ is
nonsingular, $\hsf_{\phi} < \infty$ a.e.\ $[\mu]$
and $\zeta\colon T \to \rbop$ is a
$\varSigma$-measurable function such
that\,\footnote{\;By \eqref{hfi0} and
\eqref{rhoabs0} the right-hand side of the
equality in \eqref{ccz} is $\ascr$-measurable
a.e.\ $[\mu]$.}
   \begin{align} \tag{CC$_{\zeta}$} \label{ccz}
\esf(P(\cdot, \sigma))(x) = \frac{\int_{\sigma}
\zeta(t) P(\phi(x),\D t)}{\hsf_\phi(\phi(x))} \text{
for $\mu$-a.e.\ $x \in X$}, \quad \sigma \in \varSigma.
   \end{align}
Then the following three assertions hold{\em :}
   \begin{enumerate}
   \item[(i)] $P(x,\{\zeta=0\})=0$ for $\mu$-a.e.\ $x \in
X$, and $\zeta > 0$ a.e.\ $[\rho]$,
   \item[(ii)] if $\varDelta \in \ascr$ is such that
$\int_T \zeta(t) P(x, \D t)= 0$ for $\mu$-a.e.\ $x \in
\varDelta$, then $\mu(\varDelta)=0$,
   \item[(iii)] $\varPhi$ is nonsingular and
   \begin{align}  \label{tensor0}
\hsf_{\varPhi}(x,t) = \chi_{\{\hsf_{\phi} >
0\}}(x)\zeta(t) \text{ for $\rho$-a.e.\ $(x,t)\in
X\times T$.}
   \end{align}
   \end{enumerate}
   \end{lem}
   \begin{proof}
(i) It follows from \eqref{ccz} that
$\esf(P(\cdot,\{\zeta=0\}))=0$ a.e.\ $[\mu]$.
Hence $\int_{\phi^{-1}(X)} P(x,\{\zeta=0\})
\mu(\D x)=0$, and thus $P(x,\{\zeta=0\})=0$ for
$\mu$-a.e.\ $x \in X$. This in turn implies that
   \begin{align*}
\rho(\{(x,t) \in X \times T \colon \zeta(t) = 0\})
\overset{\eqref{rhoabs}} = \int_X P(x, \{\zeta = 0\})
\mu(\D x) = 0,
   \end{align*}
which means that $\zeta > 0$ a.e.\ $[\rho]$.

(ii) If $x\in X$ is such that $\int_T \zeta(t) P(x, \D
t) = 0$, then $P(x,\{\zeta >0\})=0$. This combined with
(i) implies that $P(x,T)=0$ for $\mu$-a.e.\ $x \in
\varDelta$. Since $P(x,T)=1$ for every $x\in X$, we get
$\mu(\varDelta)=0$.

(iii) Arguing as in \eqref{prawa1} and using
Proposition \ref{dzielenie}, we get
   \allowdisplaybreaks
   \begin{align}    \notag
\rho(\varPhi^{-1}(\varDelta \times \sigma)) &=
\int_{\phi^{-1}(\varDelta)} \esf(P(\cdot, \sigma)) \D
\mu
   \\ \notag
& \hspace{-1.6ex}\overset{\eqref{ccz}}=
\int_{\phi^{-1}(\varDelta)} \frac{\int_{\sigma}
\zeta(t) P(\phi(x),\D t)}{\hsf_\phi(\phi(x))} \mu(\D
x),
   \\   \notag
& \hspace{-.2ex}\overset{\eqref{klucz}}=
\int_{\varDelta} \chi_{\{\hsf_{\phi}>0\}}(x)
\int_{\sigma} \zeta(t)P(x,\D t) \mu(\D x)
   \\   \label{tensor}
& \hspace{-.7ex}\overset{\eqref{rhoabs1}}=
\int_{\varDelta \times \sigma}
\chi_{\{\hsf_{\phi}>0\}}(x) \zeta(t) \D \rho(x,t),
\quad \varDelta \in \ascr, \sigma \in \varSigma.
   \end{align}
It is clear that $\pscr:=\ascr \boxtimes \varSigma$ is
a semi-algebra such that $\sigma(\pscr)= \ats$. Since
$\hsf_{\phi}<\infty$ a.e.\ $[\mu]$, there exists a
sequence $\{\varDelta_n\}_{n=1}^\infty \subseteq \ascr$
such that $\varDelta_n \nearrow X$ as $n\to \infty$,
$\mu(\varDelta_k) < \infty$ and $\hsf_\phi \Le k$ a.e.\
$[\mu]$ on $\varDelta_k$ for every $k\in \nbb$. Then
   \begin{align} \label{fine}
\rho(\varPhi^{-1}(\varDelta_n \times T)) &
\overset{\eqref{rhoabs}} = \mu(\phi^{-1}(\varDelta_n))
= \int_{\varDelta_n} \hsf_{\phi} \D \mu \Le n
\mu(\varDelta_n) < \infty, \quad n\in \nbb.
   \end{align}
By \eqref{tensor}, \eqref{fine} and Lemma
\ref{2miary}, the measures $\ats \ni E \to
\rho(\varPhi^{-1}(E)) \in \rbop$ and $\ats \ni E \to
\int_{E} \chi_{\{\hsf_{\phi}>0\}}(x) \zeta(t) \D
\rho(x,t) \in \rbop$ coincide. Consequently, $\varPhi$
is nonsingular and, by the $\sigma$-finiteness of
$\rho$, the equality \eqref{tensor0} holds.
   \end{proof}
Now we identify circumstances under which the
Radon-Nikodym derivative $\hsf_{\varPhi}$ depends only
on the second variable.
   \begin{thm} \label{przygot}
Suppose \eqref{standa1} holds, $\zeta\colon T \to
\rbop$ is a $\varSigma$-measurable function, $\phi$ is
nonsingular and $\hsf_{\phi} < \infty$ a.e.\ $[\mu]$.
Then the following assertions are equivalent{\em :}
   \begin{enumerate}
   \item[(i)] {\em \eqref{ccz}} holds and
$\hsf_{\phi} > 0$ a.e.\ $[\mu]$,
   \item[(ii)]  {\em \eqref{ccz}}
holds and $\int_T \zeta(t) P(\cdot, \D t)= 0$ a.e.\
$[\mu]$ on $\{\hsf_{\phi}=0\}$,
   \item[(iii)] {\em \eqref{ccz}}
holds, $\varPhi$ is nonsingular and $C_{\varPhi}$ is
quasinormal,
   \item[(iv)] the condition below holds
   \begin{align} \tag{CC$_{\zeta}^{-1}$} \label{ccz-1}
\hsf_\phi(x) \big(\esf(P(\cdot, \sigma)) \circ
\phi^{-1}\big) (x) = \int_{\sigma} \zeta(t) P(x,\D t)
\text{ for $\mu$-a.e.\ $x \in X$}, \quad \sigma \in
\varSigma,
   \end{align}
   \item[(v)] $\varPhi$ is nonsingular and
$\hsf_{\varPhi}(x,t) = \zeta(t)$ for $\rho$-a.e.\
$(x,t)\in X\times T$,
   \item[(vi)] $\varPhi$ is nonsingular, $\hsf_{\phi} > 0$
a.e.\ $[\mu]$ and
   \begin{align} \label{haphiphi}
\int_\sigma \hsf_{\varPhi}(\phi(x),t) P(\phi(x),\D t)
= \int_\sigma \zeta(t) P(\phi(x),\D t) \text{ for
$\mu$-a.e.\ $x\in X$}, \quad \sigma \in \varSigma.
   \end{align}
   \end{enumerate}
Moreover, each of the conditions {\em (i)} to {\em
(vi)} uniquely determines $\zeta$ $($up to sets of
measure $\rho$ zero$)$ and guarantees that $0 < \zeta
< \infty$ a.e.\ $[\rho]$.
   \end{thm}
   \begin{proof}
(i)$\Rightarrow$(iv) Set $H_\sigma(x)=\int_\sigma
\zeta(t) P(x, \D t)$ for $x\in X$ and $\sigma \in
\varSigma$. By \eqref{rhoabs0}, $H_\sigma$ is
$\ascr$-measurable. It follows from \eqref{ccz} and
\eqref{fifi} that
   \begin{align*}
\Big[\hsf_{\phi} \cdot \big(\esf(P(\cdot,
\sigma))\circ \phi^{-1}\big)\Big]\circ \phi = H_\sigma
\circ \phi \text{ a.e.\ $[\mu]$,} \quad \sigma \in
\varSigma.
   \end{align*}
This and the assumption that $\hsf_{\phi}
> 0$ a.e.\ $[\mu]$ imply  \eqref{ccz-1}
(see \eqref{hf0-0}).

(iv)$\Rightarrow$(ii) Substituting $\sigma=T$
into \eqref{ccz-1}, we deduce that $\int_T
\zeta(t) P(\cdot, \D t)= 0$ a.e.\ $[\mu]$ on
$\{\hsf_{\phi}=0\}$. Composing both sides of the
equality in \eqref{ccz-1} with $\phi$ and using
\eqref{compos} and \eqref{fifi}, we obtain
\eqref{ccz}.

(ii)$\Rightarrow$(i) Apply Lemma \ref{hfi0m}(ii) with
$\varDelta:=\{\hsf_{\phi}=0\}$.

(i)$\Rightarrow$(v) Note that if $f, g\colon X \to
\rbop$ are $\ascr$-measurable functions such that
$f=g$ a.e.\ $[\mu]$, then $f (x) \zeta(t) = g (x)
\zeta(t)$ for $\rho$-a.e.\ $(x,t)\in X\times T$.
Indeed, by \eqref{rhoabs1}, we get
   \begin{align*}
\int_{E} f(x) \zeta(t) \D \rho(x,t) = \int_X f(x)
\int_T \chi_E(x,t) \zeta(t) P(x,\D t) \mu(\D x) =
\int_{E} g(x) \zeta(t) \D \rho(x,t)
   \end{align*}
for every $E \in \ats$, which together with the
$\sigma$-finiteness of $\rho$ proves our claim. This
property combined with Lemma \ref{hfi0m}(iii) implies
(v).

(v)$\Rightarrow$(iv) Employing \eqref{rhoabs1} and the
$\sigma$-finiteness of $\mu$, we deduce that for every
$\sigma \in \varSigma$ and for $\mu$-a.e.\ $x \in X$,
$\int_{\sigma} \hsf_{\varPhi}(x,t) P(x,\D t) =
\int_{\sigma} \zeta(t) P(x,\D t)$. This and Lemma
\ref{hpx}(iii) yield \eqref{ccz-1}.

(v)$\Rightarrow$(iii) It follows from Lemma
\ref{hpx}(iii) that $\hsf_{\varPhi} < \infty$ a.e.\
$[\rho]$, and thus, by \eqref{gokr}, $C_{\varPhi}$ is
densely defined. Using \eqref{compos}, we see that
$\hsf_{\varPhi} = \hsf_{\varPhi} \circ \varPhi$ a.e.\
$[\rho]$. Hence, by \cite[Proposition 8.1]{b-j-j-sC},
$C_{\varPhi}$ is quasinormal. Since (v) implies (i),
\eqref{ccz} holds.

(iii)$\Rightarrow$(i) By \cite[Proposition 8.1 and
Corollary 6.6]{b-j-j-sC}, $C_{\varPhi}$ is injective.
Define the mapping $U\colon L^2(\mu) \to L^2(\rho)$ by
$(Uf)(x,t) = f(x)$ for $(x,t)\in X\times T$. Then, in
view of \eqref{rhoabs1}, $U$ is a well-defined
isometric embedding such that $U C_{\phi} =
C_{\varPhi} U$. Hence $C_{\phi}$ is injective. It
follows from \cite[Proposition 6.2]{b-j-j-sC} that
$\hsf_{\phi} > 0$ a.e.\ $[\mu]$.

(v)$\Rightarrow$(vi) As (v) implies (i), we get
$\hsf_{\phi} > 0$ a.e.\ $[\mu]$. Applying the measure
transport theorem, we see that
   \allowdisplaybreaks
   \begin{multline}  \label{many}
\int_{\phi^{-1}(\varDelta)} \int_\sigma
\hsf_{\varPhi}(\phi(x),t) P(\phi(x),\D t) \mu(\D x) =
\int_{\varDelta} \int_\sigma \hsf_{\phi}(x)
\hsf_{\varPhi}(x,t) P(x,\D t) \mu(\D x)
   \\
\overset{\eqref{rhoabs1}} =
\int_{\varDelta\times\sigma} \hsf_{\phi}(x)
\hsf_{\varPhi}(x,t) \D \rho(x,t)
\overset{\mathrm{(v)}}= \int_{\varDelta\times\sigma}
\hsf_{\phi}(x) \zeta(t) \D \rho(x,t)
   \\
= \int_{\phi^{-1}(\varDelta)} \int_\sigma \zeta(t)
P(\phi(x),\D t) \mu(\D x), \quad \varDelta \in \ascr,
\sigma \in \varSigma.
   \end{multline}
This, together with the $\sigma$-finiteness of
$\mu|_{\phi^{-1}(\ascr)}$, yields (vi).

(vi)$\Rightarrow$(i) By Lemma \ref{hpx}(iii), the
condition \eqref{lewa} holds. Composing both
sides of the equality in \eqref{lewa} with $\phi$
and using \eqref{compos} and \eqref{fifi}, we
obtain
   \begin{align*}
\esf(P(\cdot, \sigma)) (x) = \frac{\int_{\sigma}
\hsf_{\varPhi}(\phi(x),t) P(\phi(x),\D
t)}{\hsf_\phi(\phi(x))} \text{ for $\mu$-a.e.\ $x \in
X$}, \quad \sigma \in \varSigma.
   \end{align*}
This, together with \eqref{haphiphi}, gives
\eqref{ccz}.

Now we justify the ``moreover'' part. The uniqueness of
$\zeta$ follows from the equivalence of the conditions
(i) to (vi) and the equality in (v). In turn, by Lemma
\ref{hpx}(iii) and Lemma \ref{hfi0m}(i), we see that $0
< \zeta < \infty$ a.e.\ $[\rho]$. This completes the
proof.
   \end{proof}
Let us make two comments concerning Theorem
\ref{przygot}.
   \begin{rem}
a) First note that instead of proving the
implication (vi)$\Rightarrow$(i), one can prove
the implication (vi)$\Rightarrow$(v). The latter
can be justified as follows. Since $\hsf_{\phi} >
0$ a.e.\ $[\mu]$ and $\rho(\{(x,t)\in X \times
T\colon \hsf_{\phi}(x) = 0\}) = \mu(\{x \in X
\colon \hsf_{\phi}(x) = 0\}$ (cf.\
\eqref{rhoabs}), we get $\hsf_{\phi}
> 0$ a.e.\ $[\rho]$. Arguing as in
\eqref{many}, we see that for every $E \in \ascr
\boxtimes \varSigma$,
   \begin{align} \label{atimesb}
\int_{E} \hsf_{\phi}(x) \hsf_{\varPhi}(x,t) \D
\rho(x,t) = \int_{E} \hsf_{\phi}(x) \zeta(t) \D
\rho(x,t).
   \end{align}
It follows from \eqref{abc} that
$\int_{\varDelta_n\times T} \hsf_{\phi}(x)
\hsf_{\varPhi}(x,t) \D \rho(x,t) \Le n^2
\mu(\varDelta_n) < \infty$ for every $n \in \nbb$.
Hence, by Lemma \ref{2miary}, the equality
\eqref{atimesb} is valid for every $E \in \ats$. Since
$\rho$ is $\sigma$-finite, we deduce that
$\hsf_{\phi}(x) \hsf_{\varPhi}(x,t) = \hsf_{\phi}(x)
\zeta(t)$ for $\rho$-a.e.\ $(x,t)\in X\times T$. This
and the fact that $\hsf_{\phi} > 0$ a.e.\ $[\rho]$
imply (v).

b) Under the assumptions of Theorem \ref{przygot}, if
$\varPhi$ is nonsingular and there exists a countable
family $\varSigma_0$ of subsets of $T$ such that
$\varSigma = \sigma(\varSigma_0)$ (in particular, this
is the case for $T=\rbb_+$ and $\varSigma =
\borel{\rbb_+}$), then \eqref{haphiphi} holds if and
only if
   \begin{align} \label{req}
\text{$\hsf_{\varPhi}(\phi(x),t) = \zeta(t)$ for
$P(\phi(x),\cdot)$-a.e.\ $t\in T$ and for $\mu$-a.e.\
$x\in X.$}
   \end{align}
For this, note that without loss of generality we
may assume that $\varSigma_0$ is a countable
algebra of sets. Suppose \eqref{haphiphi} holds.
It follows from \eqref{abc} that $\int_T
\hsf_{\varPhi}(x,t) P(x,\D t) < \infty$ for
$\mu$-a.e.\ $x\in X$, and thus $\int_T
\hsf_{\varPhi}(\phi(x),t) P(\phi(x),\D t) <
\infty$ for $\mu$-a.e.\ $x\in X$. Hence, there
exists $X_0 \in \ascr$ such that $\mu(X \setminus
X_0)=0$, the equality in \eqref{haphiphi} holds
for all $\sigma \in \varSigma_0$ and $x\in X_0$,
and $\int_T \hsf_{\varPhi}(\phi(x),t)
P(\phi(x),\D t) < \infty$ for every $x\in X_0$.
Applying Lemma \ref{2miary}, we conclude that the
equality in \eqref{haphiphi} holds for all
$\sigma \in \varSigma$ and $x\in X_0$, which
implies \eqref{req}. The reverse implication is
obvious.
   \end{rem}
Now we state the main criterion for subnormality of
unbounded densely defined composition operators written
in terms of the conditions \eqref{ccz} and
\eqref{ccz-1}. Note that the injectivity assumption in
the hypothesis (ii) of Theorem \ref{glowne} is not
restrictive because each subnormal composition operator
being hyponormal is injective (see \cite[Corollary
6.3]{b-j-j-sC}; see also \cite[Theorem 9d]{ha-wh} for
the bounded case).
   \begin{thm} \label{glowne}
Let $(X,\ascr,\mu)$ be a $\sigma$-finite measure space
and $\phi$ be a nonsingular transformation of $X$ such
that $C_{\phi}$ is densely defined. Suppose there
exist an $\ascr$-measurable family $P\colon X \times
\varSigma \to [0,1]$ of probability measures on a
measurable space $(T,\varSigma)$ and a
$\varSigma$-measurable function $\zeta\colon T \to
\rbop$ satisfying one of the following two equivalent
conditions{\em :}
   \begin{enumerate}
   \item[(i)]  \eqref{ccz-1} holds,
   \item[(ii)]  \eqref{ccz} holds and $C_{\phi}$ is
injective.
   \end{enumerate}
Then $C_{\phi}$ is subnormal. Moreover, under the
notation of \eqref{standa1}, $\varPhi$ is nonsingular
and $C_{\varPhi}$ is a quasinormal extension of
$C_\phi$.
   \end{thm}
   \begin{proof}
Since $C_\phi$ is densely defined, we infer from
\eqref{gokr} that $\hsf_{\phi} < \infty$ a.e.\
$[\mu]$. It follows from \cite[Proposition
6.2]{b-j-j-sC} and Theorem \ref{przygot} that the
conditions (i) and (ii) are equivalent. Thus, we may
assume that \eqref{ccz-1} holds. By Theorem
\ref{przygot}, $\varPhi$ is nonsingular and
$C_{\varPhi}$ is quasinormal. Let $U$ be as in the
proof of the implication (iii)$\Rightarrow$(i) of
Theorem \ref{przygot}. Then $U$ is an isometric
embedding such that $U C_{\phi} = C_{\varPhi} U$.
This, combined with the fact that quasinormal
operators are subnormal (cf.\ \cite[Theorem
2]{StSz2}), completes the proof.
   \end{proof}
From now on we will concentrate on the particular
cases of \eqref{ccz} and \eqref{ccz-1} in which
$T=\rbb_+$, $\varSigma = \borel{\rbb_+}$ and
$\zeta(t)=t$ for $t \in \rbb_+$, i.e.,
   \begin{gather}  \tag{CC} \label{cc}
\esf(P(\cdot, \sigma)) (x) = \frac{\int_{\sigma} t
P(\phi(x),\D t)}{\hsf_\phi(\phi(x))} \text{ for
$\mu$-a.e.\ $x \in X$}, \quad \sigma \in
\borel{\rbb_+},
   \\ \tag{CC$^{-1}$} \label{cc-1} \hsf_\phi(x)
\big(\esf(P(\cdot, \sigma)) \circ \phi^{-1}\big) (x) =
\int_{\sigma} t P(x,\D t) \text{ for $\mu$-a.e.\ $x \in
X$}, \; \; \, \sigma \in \borel{\rbb_+}.
   \end{gather}
We refer to \eqref{cc} as the {\em consistency
condition} (it has been inspired by \cite{Bu-St2}). It
is worth pointing out that if an $\ascr$-measurable
family $P\colon X \times \varSigma \to [0,1]$ of
probability measures satisfies \eqref{ccz}
(respectively, \eqref{ccz-1}), $\zeta$ is injective
and $\zeta(\sigma) \in \borel{\rbb_+}$ for every
$\sigma \in \varSigma$, then, by the measure transport
theorem, the mapping $\widetilde P\colon X \times
\borel{\rbb_+} \to [0,1]$ given by
   \begin{align*}
\widetilde P(x,\sigma) = P(x,\zeta^{-1}(\sigma)),
\quad x \in X, \, \sigma \in \borel{\rbb_+},
   \end{align*}
is an $\ascr$-measurable family of Borel probability
measures which satisfies \eqref{cc} (respectively,
\eqref{cc-1}).

Below we show that the consistency condition, which
together with injectivity is sufficient for
subnormality, turns out to be necessary in the case of
quasinormal composition operators.
   \begin{pro} \label{quasi-auto}
Let $(X,\ascr,\mu)$ be a $\sigma$-finite measure space
and $\phi$ be a nonsingular transformation of $X$ such
that $C_{\phi}$ is quasinormal. Then there exists a
$\phi^{-1}(\ascr)$-measurable family $P\colon X \times
\borel{\rbb_+} \to [0,1]$ of probability measures
which satisfies \eqref{cc}. Moreover, if $\widetilde
P\colon X \times \borel{\rbb_+} \to [0,1]$ is any
$\ascr$-measurable family of probability measures
satisfying \eqref{cc}, then $\widetilde P(x, \cdot) =
P(x, \cdot)$ for $\mu$-a.e.\ $x\in X$.
   \end{pro}
   \begin{proof}
We can assume that $0 < \hsf_{\phi} < \infty$
(cf.\ \cite[Section 6]{b-j-j-sC} and
\eqref{gokr}). It follows from \cite[Proposition
8.1]{b-j-j-sC} that $\hsf_{\phi} =
\hsf_{\phi}\circ \phi$ a.e.\ $[\mu]$. Define the
$\phi^{-1}(\ascr)$-measurable family $P\colon X
\times \borel{\rbb_+} \to [0,1]$ of probability
measures by
   \begin{align}  \label{quas1}
P(x,\sigma) = \chi_{\sigma}(\hsf_{\phi}(\phi(x))),
\quad x\in X, \sigma \in \borel{\rbb_+}.
   \end{align}
Since $\hsf_{\phi} = \hsf_{\phi}\circ \phi$ a.e.\
$[\mu]$, we deduce that $P(\phi(x),\sigma) =
\chi_{\sigma}(\hsf_{\phi}(\phi(x)))$ for $\mu$-a.e.\
$x \in X$ and $\sigma \in \borel{\rbb_+}$. This yields
   \begin{align} \label{quas2}
\frac{\int_{\sigma} t P(\phi(x),\D
t)}{\hsf_\phi(\phi(x))} =
\chi_{\sigma}(\hsf_{\phi}(\phi(x))) \text{ for
$\mu$-a.e.\ $x\in X$}, \quad \sigma \in
\borel{\rbb_+}.
   \end{align}
Combining \eqref{quas1} and \eqref{quas2} shows that
$P$ satisfies \eqref{cc}.

The ``moreover'' part follows from \eqref{quas1} and
Corollary \ref{uniq}.
   \end{proof}
   \subsection{\label{tbc}The bounded case}
We begin by proving a ``moment measurability'' lemma
which is a variant of \cite[Lemma 1.3]{lam2}. The
proof of the latter contains an error which comes from
using an untrue statement that characteristic
functions of Borel sets on the real line are of the
first Baire category. The proof of Lemma \ref{mommea}
is extracted from that of \cite[Theorem 4.5]{Bu-St1}.
   \begin{lem} \label{mommea}
Let $(X,\ascr)$ be a measurable space and $K$ be a
compact subset of the complex plane $\cbb$. Suppose
that $\{\vartheta_x\}_{x\in X}$ is a family of Borel
probability measures on $K$ such that
   \begin{align} \label{guido}
   \begin{minipage}{74ex} the map $X\ni x \mapsto \int_{K}
z^m\bar z^n \vartheta_x(\D z) \in \cbb$ is
$\ascr$-measurable for all $m, n\in \zbb_+$.
   \end{minipage}
   \end{align}
Then the function $P\colon X \times \borel{K} \ni
(x,\sigma) \mapsto \vartheta_x(\sigma) \in [0,1]$ is
an $\ascr$-measurable family of probability measures.
   \end{lem}
   \begin{proof}
Without loss of generality we may assume that $K$ is a
rectangle of the form $K=[-r,r] \times [-r,r]$, where
$r$ is a positive real number. It follows from
\eqref{guido} that for every $p\in \cbb[z,\bar z] $
the function $X\ni x \mapsto \int_K p \D \vartheta_x
\in \cbb$ is $\ascr$-measurable, where $\cbb[z,\bar
z]$ stands for the ring of all complex polynomials in
variables $z$ and $\bar z$. If $f\colon K \to \cbb$ is
a continuous function, then by the Stone-Weierstrass
theorem, there exists a sequence $\{p_n\}_{n=1}^\infty
\subseteq \cbb[z,\bar z]$ which converges uniformly on
$K$ to $f$. Hence, due to the fact that each measure
$\vartheta_x$ is finite, the sequence $\{\int_K p_n \D
\vartheta_x\}_{n=1}^\infty$ converges to $\int_K f \D
\vartheta_x$ for every $x\in X$, which implies that
the function $X\ni x \mapsto \int_K f \D \vartheta_x
\in \cbb$ is $\ascr$-measurable. Take an arbitrary
rectangle $L=[a_1,b_1) \times [a_2,b_2)$ with
$a_1,a_2, b_1, b_2 \in \rbb$. Then there exists a
sequence $\{f_n\}_{n=1}^\infty$ of continuous
functions $f_n\colon K\to [0,1]$ which converges
pointwise to $\chi_{L\cap K}$. We infer from the
Lebesgue dominated convergence theorem that the
function $X \ni x \mapsto \vartheta_x (L\cap K) \in
[0,1]$ is $\ascr$-measurable. Set
   \begin{align*}
\mathfrak M = \big\{ \sigma \in \borel{K} \colon
\text{the function } X \ni x \mapsto \vartheta_x
(\sigma) \in [0,1] \text{ is
$\ascr$-measurable}\big\}.
   \end{align*}
It is easily seen that $\mathfrak M$ is a monotone
class which contains $\varnothing$ and $K$, and which
is closed under the operation of taking finite
disjoint union of sets. Hence, the algebra
$\varSigma_0$ generated by the semi-algebra of all
rectangles of the form $L\cap K$ with $L$ as above, is
contained in $\mathfrak M$. By the monotone class
theorem (cf.\ \cite[Theorem 1.3.9]{Ash}), $\mathfrak M
= \sigma(\varSigma_0)=\borel{K}$, which completes the
proof.
   \end{proof}
   \begin{rem}
Lemma \ref{mommea} can be easily adapted to the
$N$-dimensional case by allowing exponents $m,n$ in
\eqref{guido} to vary over the multiindex set
$\zbb_+^N$. The proof is essentially the same.
   \end{rem}
Note that a bounded subnormal operator $S$ always has
a bounded normal extension. Indeed, by \cite[Theorem
1]{StSz3}, the spectrum of a minimal normal extension
$N$ of spectral type of $S$ is contained in the
spectrum of $S$ which is compact; hence, by the
spectral theorem, $N$ is bounded. This means that our
definition of subnormality extends that for bounded
operators.
   \begin{thm}   \label{bsubn}
Suppose $(X,\ascr,\mu)$ is a $\sigma$-finite measure
space and $\phi$ is a nonsingular transformation of
$X$ such that $C_{\phi} \in \ogr{L^2(\mu)}$. Then the
following conditions are equivalent{\em :}
   \begin{enumerate}
   \item[(i)] $C_{\phi}$ is subnormal,
   \item[(ii)] $C_{\phi}$ is
injective and there exists an $\ascr$-measurable
family $P\colon X \times \borel{\rbb_+} \to [0,1]$ of
probability measures which satisfies \eqref{cc},
   \item[(ii$^\prime$)] there exists
an $\ascr$-measurable family $P\colon X \times
\borel{\rbb_+} \to [0,1]$ of probability measures
which satisfies \eqref{cc-1},
   \item[(iii)]
$C_{\phi}$ is injective and there exists an
$\ascr$-measurable family $P\colon X \times
\borel{\rbb_+} \to [0,1]$ of probability measures such
that {\em \eqref{cc}} holds and the closed support of
$P(x,\cdot)$ is contained in $[0,\|C_{\phi}\|^2]$ for
$\mu$-a.e.\ $x\in X$,
   \item[(iii$^\prime$)] there exists
an $\ascr$-measurable family $P\colon X \times
\borel{\rbb_+} \to [0,1]$ of probability measures such
that {\em \eqref{cc-1}} holds and the closed support
of $P(x,\cdot)$ is contained in $[0,\|C_{\phi}\|^2]$
for $\mu$-a.e.\ $x\in X$.
   \end{enumerate}
The conditions above remain equivalent if the
expression ``for $\mu$-a.e.\ $x\in X$'' is replaced by
``for every $x\in X$''. Moreover, if $C_{\phi}$ is
subnormal and $P_1,P_2\colon X \times \borel{\rbb_+}
\to [0,1]$ are $\ascr$-measurable families of
probability measures satisfying \eqref{cc}, then
$P_1(x, \cdot) = P_2(x, \cdot)$ for $\mu$-a.e.\ $x\in
X$.
   \end{thm}
   \begin{proof}
By Theorem \ref{glowne}, (ii) is equivalent to
(ii$^\prime$) and (iii) is equivalent to
(iii$^\prime$).

(i)$\Rightarrow$(iii) Since subnormal operators are
hyponormal, we deduce that $\hsf_{\phi} > 0$ a.e.\
$[\mu]$ (cf.\ \cite[Theorem 9d]{ha-wh}), and thus
$C_{\phi}$ is injective. Set $K=[0,\|C_{\phi}\|^2]$.
By \cite[Corollary 4]{lam1} (or rather by
\cite[Theorem 3.4]{Bu-St1} where $\phi(X)=X$ is not
assumed), there exist a set $\varDelta_0 \in \ascr$
and a family $\{\vartheta_x\colon x\in \varDelta_0\}$
of Borel probability measures on $K$ such that
$\mu(X\setminus \varDelta_0)=0$ and for every $x\in
\varDelta_0$,
   \begin{align} \label{lam1}
\hsf_{\phi^n}(x) = \int_K t^n \vartheta_x(\D t), \quad
n\in \zbb_+.
   \end{align}
Setting $\hsf_{\phi^n}(x)=\chi_{\{0\}}(n)$ and
$\vartheta_x(\sigma)=\chi_{\sigma}(0)$ for $n \in
\zbb_+$, $x \in X \setminus \varDelta_0$ and $\sigma
\in \borel{K}$, we may assume that each
$\hsf_{\phi^n}$ is $\ascr$-measurable and \eqref{lam1}
holds for all $x \in X$. By Lemma \ref{mommea}, the
function $\widetilde P\colon X \times \borel{K} \to
[0,1]$ given by
   \begin{align*}
\widetilde P(x,\sigma)=\vartheta_x(\sigma), \quad x
\in X, \sigma \in \borel{K},
   \end{align*}
is an $\ascr$-measurable family of probability
measures. Set $T=K$ and $\varSigma=\borel{K}$. Let
$\rho$ and $\varPhi$ be as in Section \ref{acrit}
(with $\widetilde P$ in place of $P$). To proceed
further we need \cite[Lemma 2.4]{lam2}. Since its
proof contains an error of the same type as that
mentioned in the first paragraph of Section \ref{tbc},
we provide a correction. Applying the polynomial
approximation procedure given in Lambert's original
proof, we get
   \begin{align}  \label{lamb}
\rho(\varPhi^{-1}(E)) = \int_{E} t \D \rho(x, t)
   \end{align}
for every set $E$ of the form $E=\varDelta \times
(J\cap K)$, where $\varDelta \in \ascr$ and $J=[a,b)$
with $a,b\in \rbb_+$. We shall prove that \eqref{lamb}
holds for all $E \in \ascr \otimes \borel{K}$. For
this, denote by $\mathscr F$ the algebra generated by
the semi-algebra $\{[a,b)\cap K\colon a,b \in
\rbb_+\}$. It is clear that $\pscr:=\{\varDelta \times
\sigma\colon \varDelta \in \ascr, \, \sigma \in
\mathscr F\}$ is a semi-algebra such that
$\sigma(\pscr)=\ascr \otimes \borel{K}$ (because
$\sigma(\mathscr F)=\borel{K}$). By \cite[Proposition
I-6-1]{Nev}, the equality \eqref{lamb} holds for all
$E\in \pscr$. Note that $\rho(\varPhi^{-1}(\varDelta
\times K)) = \int_{\varDelta} \hsf_{\phi} \D \mu <
\infty$ whenever $\mu(\varDelta) < \infty$. As $\mu$
is $\sigma$-finite, an application of Lemma
\ref{2miary} shows that \eqref{lamb} holds for all $E
\in \ascr \otimes \borel{K}$. This means that
$\varPhi$ is nonsingular and $\hsf_{\varPhi}(x,t)=t$
for $\rho$-a.e.\ $(x,t) \in X\times K$. Applying
Theorem \ref{przygot} with $\zeta(t):=t$ for $t \in K$
yields
   \begin{align*}
\esf(\widetilde P(\cdot, \sigma)) (x) =
\frac{\int_{\sigma} t \widetilde P(\phi(x),\D
t)}{\hsf_\phi(\phi(x))} \text{ for $\mu$-a.e.\ $x \in
X$}, \quad \sigma \in \borel{K},
   \end{align*}
Setting $P(x,\sigma)=\widetilde P(x,\sigma \cap K)$
for $x\in X$ and $\sigma \in \borel{\rbb_+}$ shows
that (iii) is satisfied. Note that the closed support
of $P(x,\cdot)$ is contained in $[0,\|C_{\phi}\|^2]$
for every $x\in X$.

(iii)$\Rightarrow$(ii) Obvious.

(ii)$\Rightarrow$(i) Apply Theorem \ref{glowne}.

The ``moreover'' part follows from (iii) and Corollary
\ref{uniq}.
   \end{proof}
   \subsection{\label{tcc}The consistency condition}
The consistency condition is the subject of our
investigations in this section. The following
assumptions will be often used.
   \begin{align} \label{standa2}
   \begin{minipage}{68ex}
\hspace{2ex} The triplet $(X,\ascr,\mu)$ is a
$\sigma$-finite measure space, $\phi$ is a nonsingular
transformation of $X$ such that $\hsf_{\phi} < \infty$
a.e.\ $[\mu]$ and $P\colon X \times \borel{\rbb_+} \to
[0,1]$ is an $\ascr$-measurable family of probability
measures.
   \end{minipage}
   \end{align}
   \begin{lem} \label{jeden}
Suppose \eqref{standa2} holds. Then \eqref{cc} is
equivalent to each of the following three
conditions{\em :}
   \begin{enumerate}
   \item[(i)] $\esf\big(\int_0^\infty f(t)
P(\cdot, \D t)\big) (x)=\frac{\int_0^\infty t f(t)
P(\phi(x),\D t)} {\hsf_\phi(\phi(x))}$ for $\mu$-a.e.\
$x \in X$ and for every Borel function $f \colon
\rbb_+ \to \rbop$,
   \item[(ii)] $P(x,\{0\})=0$
and $\esf\big(\int_{\sigma} \frac 1t P(\cdot, \D
t)\big) (x) = \frac{P(\phi(x),\sigma)}
{\hsf_\phi(\phi(x))}$ for $\mu$-a.e.\ $x \in X$ and
for every $\sigma \in \borel{\rbb_+},$
   \item[(iii)] $P(x,\{0\})=0$
and $\esf\big(\int_0^\infty \frac {g(t)}t P(\cdot, \D
t)\big) (x)=\frac{\int_0^\infty g(t) P(\phi(x),\D
t)}{\hsf_\phi(\phi(x)) }$ for $\mu$-a.e.\ $x \in X$
and for every Borel function $g \colon \rbb_+ \to
\rbop$,
   \end{enumerate}
where $\int_0^\infty h(t) P(\cdot, \D t)$ is
understood as a function $X\ni x \to \int_0^\infty
h(t) P(x, \D t) \in \rbop$ whenever $h\colon \rbb_+
\to \rbop$ is a Borel function. Moreover, if
\eqref{cc} holds, then
   \begin{align*}
\text{$\esf\bigg(\int_0^\infty \frac 1 t P(\cdot,
\D t)\bigg)(x) = \frac{1}{\hsf_\phi(\phi(x))} <
\infty$ for $\mu$-a.e.\ $x\in X$.}
   \end{align*}
   \end{lem}
   \begin{proof}
Since each Borel function $f\colon \rbb_+ \to \rbop$
is a pointwise limit of an increasing sequence of
nonnegative Borel simple functions, one can show that
\eqref{cc} implies (i) by applying the Lebesgue
monotone convergence theorem as well as the additivity
and the monotone continuity of the conditional
expectation (see \eqref{CE-2}). The same argument can
be used to prove that (ii) implies (iii). It is
obvious that (iii) implies (ii) and that (i) implies
\eqref{cc}.

(i)$\Rightarrow$(iii) By Lemma \ref{hfi0m}(i),
$P(x,\{0\})=0$ for $\mu$-a.e.\ $x \in X$. Thus, if $g
\colon \rbb_+ \to \rbop$ is a Borel function, then, by
applying (i) to the Borel function $f(t) = g(t)/t$, we
obtain (iii).

(iii)$\Rightarrow$(i) Apply (iii) to $g(t)=tf(t)$.

The ``moreover'' part follows from \eqref{hfi0}
and (iii) applied to $g(t)\equiv 1$.
   \end{proof}
The equality \eqref{hn+1f} below appeared in
\cite[Lemma 1.2]{lam2} under the assumption that
$\phi$ is surjective and $C_\phi$ is bounded. For
self-containedness, we include its proof (essentially
the same as that of Lambert's one).
   \begin{lem} \label{enplus}
If $(X,\ascr,\mu)$ is a $\sigma$-finite measure space
and $\phi$ is a nonsingular transformation of $X$ such
that $\hsf_\phi < \infty$ a.e.\ $[\mu]$, then
   \begin{align} \label{hn+1f}
\hsf_{\phi^{n+1}} &= \hsf_{\phi} \cdot
\esf(\hsf_{\phi^n}) \circ \phi^{-1} \text{ a.e.\
$[\mu]$ for all $n\in \zbb_+$,}
   \\ \label{hn+1fpre}
\hsf_{\phi^{n+1}} \circ \phi & = \hsf_{\phi} \circ
\phi \cdot \esf(\hsf_{\phi^n}) \text{ a.e.\ $[\mu]$
for all $n\in \zbb_+$.}
   \end{align}
   \end{lem}
   \begin{proof} In view of the measure
transport theorem, we have
   \begin{multline*}
\mu(\phi^{-(n+1)}(\varDelta)) =
\mu(\phi^{-n}(\phi^{-1}(\varDelta))) =
\int_{\phi^{-1}(\varDelta)} \hsf_{\phi^n} \D \mu
   \\
\overset{\eqref{CE-3}}= \int_{\phi^{-1}(\varDelta)}
\esf(\hsf_{\phi^n}) \D \mu \overset{\eqref{fifi}}=
\int_{\varDelta} \hsf_{\phi} \cdot \esf(\hsf_{\phi^n})
\circ \phi^{-1} \D \mu, \quad \varDelta \in \ascr,
   \end{multline*}
which yields \eqref{hn+1f}. By \eqref{fifi} and
\eqref{compos}, the condition \eqref{hn+1fpre} follows
from \eqref{hn+1f}.
   \end{proof}
   \begin{rem}
Using \eqref{hn+1f}, we can express $\hsf_{\phi^n}$ in
terms of $\hsf_{\phi}$ by iterating the
multiplication, the conditional expectation and the
operation $\esf(g) \circ \phi^{-1}$. Unfortunately,
the so-obtained formulas are rather complicated (e.g.,
$\hsf_{\phi^2} = \hsf_{\phi} \cdot \esf(\hsf_{\phi})
\circ \phi ^{-1}$ a.e.\ $[\mu]$, $\hsf_{\phi^3} =
\hsf_{\phi} \cdot \esf\big(\hsf_{\phi} \cdot
\esf(\hsf_{\phi}) \circ \phi^{-1}\big) \circ \phi
^{-1}$ a.e.\ $[\mu]$ and so on).
   \end{rem}
As shown below, under the assumption that
$\hsf_\phi>0$ a.e.\ $[\mu]$, an $\ascr$-measurable
family $P$ of probability measures satisfying
\eqref{cc} has the property that the ``moments'' of
$P(x,\cdot)$ coincide with
$\{\hsf_{\phi^n}(x)\}_{n=0}^\infty$ for $\mu$-a.e.\
$x\in X$. This fact plays an essential role in the
present paper as well as in the proof of the new
characterization of quasinormal composition operators
given in \cite{b-j-j-sE}.
   \begin{thm} \label{sms}
Suppose {\em \eqref{standa2}} and \eqref{cc} hold, and
$\hsf_\phi>0$ a.e.\ $[\mu]$. Then
   \begin{align} \label{funt}
\hsf_{\phi^n}(x)=\int_0^\infty t^n P(x,\D t) \text{
for $\mu$-a.e.\ $x \in X$}, \quad n \in \zbb_+.
   \end{align}
Moreover, if $C_{\phi}^n$ is densely defined for every
$n\in \zbb_+$, then
$\{\hsf_{\phi^n}(x)\}_{n=0}^\infty$ is a Stieltjes
moment sequence with a representing measure
$P(x,\cdot)$ for $\mu$-a.e.\ $x\in X$.
   \end{thm}
   \begin{proof}
To prove \eqref{funt}, we use an induction on $n$. Set
$H_n(x)=\int_0^\infty t^n P(x,\D t)$ for $x \in X$ and
$n\in \zbb_+$. By \eqref{rhoabs0}, the function
$H_n\colon X \to \rbop$ is $\ascr$-measurable for
every $n\in \zbb_+$. Since $P(x,\cdot)$, $x \in X$,
are probability measures, we deduce that $H_0(x)=1$
for all $x\in X$, and thus $H_0=h_{\phi^0}$ a.e.\
$[\mu]$. Suppose that $H_n=h_{\phi^n}$ a.e.\ $[\mu]$
for a fixed $n\in \zbb_+$. Then, by Lemma
\ref{jeden}(i), applied to $f(t)=t^n$, we have
   \begin{multline*}
H_{n+1}(\phi(x)) = \int_0^\infty t^n t
P(\phi(x),\D t) = \hsf_{\phi}(\phi(x))
\esf\big(H_n\big) (x)
   \\
= \hsf_{\phi}(\phi(x))
\esf\big(h_{\phi^n}\big) (x)
\overset{\eqref{hn+1fpre}}= \hsf_{\phi^{n+1}}
(\phi(x)) \quad \text{for $\mu$-a.e.\ $x \in X$.}
   \end{multline*}
Applying \eqref{hf0-0}, we get
$H_{n+1}=\hsf_{\phi^{n+1}}$ a.e.\ $[\mu]$, which
yields \eqref{funt}.

The ``moreover'' part follows from \eqref{funt}
and the fact that under our density assumption,
$\hsf_{\phi^n}(x) < \infty$ for $\mu$-a.e.\ $x\in
X$ and for every $n\in \zbb_+$ (cf.\
\cite[Corollary 4.5]{b-j-j-sC}).
   \end{proof}
Regarding Theorem \ref{sms}, it is worth mentioning
that $C_{\phi}^n$ is densely defined for every $n\in
\zbb_+$ if and only if $\dzn{C_{\phi}}$ is dense in
$L^2(\mu)$ (cf.\ \cite[Theorem 4.7]{b-j-j-sC}).
   \begin{cor} \label{uniq}
Suppose \eqref{standa2} and \eqref{cc} hold,
$\hsf_{\phi} > 0$ a.e.\ $[\mu]$ and the measure
$P(x,\cdot)$ is determinate for $\mu$-a.e.\ $x \in X$.
If $\widetilde P\colon X \times \borel{\rbb_+} \to
[0,1]$ is any $\ascr$-measurable family of probability
measures which satisfies \eqref{cc}, then $\widetilde
P (x, \cdot) = P (x, \cdot)$ for $\mu$-a.e.\ $x\in X$.
   \end{cor}
The proof of the following corollary is patterned on
that of the assertion (b) of \cite[Lemma
10.1]{b-j-j-sC}.
   \begin{cor} \label{cc2phinp}
Assume that {\em \eqref{standa2}} and \eqref{cc}
hold, and $\hsf_\phi>0$ a.e.\ $[\mu]$. Then
$C_{\phi}^n = C_{\phi^n}$ for every $n\in
\zbb_+$.
   \end{cor}
   \begin{proof}
By (3.5) and (3.6) in \cite{b-j-j-sC}, we have
$\dz{C_{\phi^n}} = L^2((1+\hsf_{\phi^n})\D \mu)$
and $\dz{C_{\phi}^n} = L^2((\sum_{j=0}^n
\hsf_{\phi^j}) \D \mu)$, and thus $C_{\phi}^n
\subseteq C_{\phi^n}$. Since $P(x,\cdot)$, $x\in
X$, are probability measures, we deduce from
\eqref{funt} that for $\mu$-a.e.\ $x \in X$,
   \begin{align*}
\sum_{j=0}^n \hsf_{\phi^j}(x) & = \int_0^\infty
\bigg(\sum_{j=0}^n t^j\bigg) P(x,\D t)
   \\
&= \int_{[0,1]} \bigg(\sum_{j=0}^n t^j\bigg)
P(x,\D t) + \int_{(1,\infty)} \bigg(\sum_{j=0}^n
t^j\bigg) P(x,\D t)
   \\[.5ex]
& \Le (n+1) (1 + \hsf_{\phi^n}(x)),
   \end{align*}
which implies that $\dz{C_{\phi^n}} \subseteq
\dz{C_{\phi}^n}$. This completes the proof.
   \end{proof}
   \begin{rem} \label{pseudo}
If $C_{\phi}^n$ is not densely defined for some
integer $n\Ge 1$, then $\hsf_{\phi^n}$ takes the value
$\infty$ on a set of positive measure (cf.\
\eqref{gokr}), which in view of \eqref{funt} may lead
to infinite moments. We say that a sequence
$\gammab=\{\gamma_n\}_{n=0}^\infty \subseteq \rbop$ is
a {\em pseudo-Stieltjes moment sequence} if there
exists a finite Borel measure $\nu$ on $\rbb_+$,
called a {\em representing measure} of $\gammab$, such
that $\gamma_n = \int_0^\infty s^n \nu(\D x)$ for all
$n\in\zbb_+$. If $\gamma_k=\infty$ for some $k \in
\nbb$, then there exists a unique
$k_\infty(\gammab)\in \nbb$ such that
$\gamma_k=\infty$ for every integer $k \Ge
k_\infty(\gammab)$, and $\gamma_k < \infty$ for every
nonnegative integer $k < k_\infty(\gammab)$. It is
easily seen that for every $k\in \nbb$, there exists a
pseudo-Stieltjes moment sequence $\gammab$ such that
$k_\infty(\gammab)=k$ (e.g., the one represented by
the measure $\nu=\sum_{j=1}^\infty \frac 1 {j^{k+1}}
\delta_j$). Note that if $\gammab$ is a
pseudo-Stieltjes moment sequence which is not a
Stieltjes moment sequence, then it has infinitely many
representing measures (i.e., $\gammab$ is
indeterminate). Indeed, let $\nu$ be a representing
measure of $\gammab$. Since the truncated Stieltjes
moment problem (with the unknown Borel measure
$\vartheta$ on $\rbb_+$)
   \begin{align} \label{cu-fi-bar}
\gamma_n = \int_0^\infty s^n \vartheta(\D s),
\quad n=0, \ldots, k_\infty(\gammab) - 1,
   \end{align}
has a solution $\vartheta=\nu$, we infer from
\cite[Theorem 3.6]{cu-fi} that there exists a Borel
measure $\tau$ on $\rbb_+$ with finite support such
that \eqref{cu-fi-bar} holds for $\vartheta=\tau$.
Given $\alpha \in (0,1)$, we set $\nu_\alpha = \alpha
\tau + (1-\alpha)\nu$. It is clear that the measure
$\nu_\alpha$ satisfies \eqref{cu-fi-bar} and that
$\int_0^\infty s^n \D \nu_\alpha = \infty$ for all
integers $n \Ge k_\infty(\gammab)$. Hence $\nu_\alpha$
represents $\gammab$ and, as easily seen, the mapping
$\alpha \mapsto \nu_\alpha$ is injective.
   \end{rem}
   \begin{rem}
Theorem \ref{sms} suggests the method of looking for
an $\ascr$-measurable family $P$ of Borel probability
measures on $\rbb_+$ which satisfies \eqref{cc}.
First, we verify whether
$\{\hsf_{\phi^n}(x)\}_{n=0}^\infty$ is a
pseudo-Stieltjes moment sequence for $\mu$-a.e.\ $x\in
X$ (cf.\ Remark \ref{pseudo}). If this is the case,
then we select a family $\{\vartheta_x\}_{x\in X}$ of
Borel probability measures on $\rbb_+$ such that
$\vartheta_x$ is a representing measure of
$\{\hsf_{\phi^n}(x)\}_{n=0}^\infty$ for $\mu$-a.e.\
$x\in X$, and then verify whether the family $P\colon
X \times \borel{\rbb_+} \to [0,1]$ of probability
measures given by $P(x,\sigma)=\vartheta_x(\sigma)$
for $x\in X$ and $\sigma \in \borel{\rbb_+}$ is
$\ascr$-measurable and satisfies \eqref{cc}. This
method works perfectly well in some cases (see e.g.,
Theorem \ref{matrical} and Example \ref{busube}).
Unfortunately, it may break down even if
$\{\hsf_{\phi^n}(x)\}_{n=0}^\infty$ is a Stieltjes
moment sequence for $\mu$-a.e.\ $x\in X$. Indeed,
there exists a non-subnormal injective composition
operator $C_{\phi}$ in $L^2(\mu)$ such that
$\overline{\dzn{C_{\phi}}}=L^2(\mu)$ and
$\{\hsf_{\phi^n}(x)\}_{n=0}^\infty$ is a Stieltjes
moment sequence for $\mu$-a.e.\ $x\in X$ (cf.\
\cite[Theorem 4.3.3]{j-j-s0} and \cite[Section
11]{b-j-j-sC}). In view of Theorem \ref{glowne}, for
such $C_{\phi}$ there is no possibility to select $P$
with the desired properties.
   \end{rem}
Our next aim is to show that the condition \eqref{cc}
behaves well with respect to the operation of taking
powers of composition operators. We begin by proving an
auxiliary result on conditional expectation which is of
some independent interest in itself. Given a
$\sigma$-finite measure space $(X,\ascr,\mu)$, a
nonsingular transformation $\phi$ of $X$ and a positive
integer $n$ such that $\hsf_{\phi^n} < \infty$ a.e.\
$[\mu]$, we write $\esf_n(f)$ for the conditional
expectation of an $\ascr$-measure function $f\colon X
\to \rbop$ with respect to the sub $\sigma$-algebra
$(\phi^n)^{-1}(\ascr)$ of $\ascr$. In view of the
discussion in the last paragraph of Preliminaries, the
expression $\esf_n(f) \circ \phi^{-n} := \esf_n(f)
\circ (\phi^n)^{-1}$ makes sense.
   \begin{lem} \label{phi-n}
Let $(X,\ascr,\mu)$ be a $\sigma$-finite measure space,
$\phi$ be a nonsingular transformation of $X$ and $n$
be a positive integer such that $\hsf_{\phi},
\hsf_{\phi^n},\hsf_{\phi^{n+1}} < \infty$ a.e.\
$[\mu]$. Then, for every $\ascr$-measurable function
$f\colon X \to \rbop$, the following holds{\em :}
   \begin{enumerate}
   \item[(i)] $\hsf_{\phi^{n+1}} \cdot \esf_{n+1}(f) \circ
\phi^{-(n+1)} = \hsf_{\phi^n} \cdot \esf_n(\hsf_{\phi}
\cdot \esf(f)\circ \phi^{-1}) \circ \phi^{-n}$ a.e.\
$[\mu]$,
   \item[(ii)] $\hsf_{\phi^{n+1}} \cdot \esf_{n+1}(f) \circ
\phi^{-(n+1)} = \hsf_{\phi} \cdot \esf(\hsf_{\phi^n}
\cdot \esf_n(f)\circ \phi^{-n}) \circ \phi^{-1}$ a.e.\
$[\mu]$.
   \end{enumerate}
   \end{lem}
   \begin{proof}
(i) Note that
   \allowdisplaybreaks
   \begin{align} \notag
\int_{\phi^{-(n+1)}(\varDelta)} f \D \mu & = \int_X
\chi_{\phi^{-n}(\varDelta)} \circ \phi \cdot \esf(f)
\D \mu
   \\      \notag
&\hspace{-0.3ex}\overset{\eqref{fifi}} =
\int_{\phi^{-n}(\varDelta)} \hsf_{\phi} \cdot \esf(f)
\circ \phi^{-1} \D \mu
   \\      \notag
&= \int_X \chi_{\varDelta} \circ \phi^n \cdot
\esf_n(\hsf_{\phi} \cdot \esf(f) \circ \phi^{-1}) \D
\mu
   \\        \label{jeden2}
&\hspace{-0.3ex}\overset{\eqref{fifi}}=
\int_{\varDelta} \hsf_{\phi^n} \cdot
\esf_n(\hsf_{\phi} \cdot \esf(f) \circ \phi^{-1})
\circ \phi^{-n} \D \mu, \quad \varDelta \in \ascr,
   \end{align}
and
   \allowdisplaybreaks
   \begin{align} \notag
\int_{\phi^{-(n+1)}(\varDelta)} f \D \mu &= \int_X
\chi_{\varDelta} \circ \phi^{n+1} \cdot \esf_{n+1}(f)
\D \mu
   \\  \label{dwa}
& \hspace{-.8ex}\overset{\eqref{fifi}} =
\int_{\varDelta} \hsf_{\phi^{n+1}} \cdot
\esf_{n+1}(f) \circ \phi^{-(n+1)} \D \mu, \quad
\varDelta \in \ascr.
   \end{align}
Hence (i) follows from \eqref{jeden2}, \eqref{dwa} and
the $\sigma$-finiteness of $\mu$.

(ii) Similarly, the equalities
   \begin{align*}
\int_{\phi^{-(n+1)}(\varDelta)} f \D \mu & = \int_X
\chi_{\phi^{-1}(\varDelta)} \circ \phi^n \cdot
\esf_n(f) \D\mu
   \\
&\hspace{-.8ex}\overset{\eqref{fifi}} = \int_X
\chi_{\phi^{-1}(\varDelta)} \cdot \hsf_{\phi^n}
\cdot \esf_n(f) \circ \phi^{-n} \D\mu
   \\
& = \int_X \chi_{\varDelta} \circ \phi \cdot
\esf(\hsf_{\phi^n} \cdot \esf_n(f) \circ \phi^{-n})
\D\mu
   \\
& \hspace{-.8ex}\overset{\eqref{fifi}} =
\int_{\varDelta} \hsf_{\phi} \cdot
\esf(\hsf_{\phi^n} \cdot \esf_n(f) \circ
\phi^{-n}) \circ \phi^{-1}\D\mu, \quad \varDelta
\in \ascr,
   \end{align*}
combined with \eqref{dwa}, give (ii). This completes
the proof.
   \end{proof}
If $f\equiv 1$, then, in view of \eqref{jedynka}, the
formulas (i) and (ii) of Lemma \ref{phi-n} take the
following forms (see \eqref{hn+1f} where
$\hsf_{\phi^n}$ and $\hsf_{\phi^{n+1}}$ are not
assumed to be finite a.e.\ $[\mu]$)
   \begin{align} \label{em-lam-f}
\hsf_{\phi^{n+1}} = \hsf_{\phi^n} \cdot
\esf_n(\hsf_{\phi}) \circ \phi^{-n} = \hsf_{\phi} \cdot
\esf(\hsf_{\phi^n}) \circ \phi^{-1} \text{ a.e. }
[\mu].
   \end{align}
Under more restrictive assumptions on $\phi$,
equalities \eqref{em-lam-f} appeared in \cite[p.\
166]{Em-Lam-center}.
   \begin{pro}  \label{niema}
Let $(X,\ascr,\mu)$ be a $\sigma$-finite measure
space, $\phi$ be a nonsingular transformation of
$X$ such that $0 < \hsf_{\phi} < \infty$ a.e.\
$[\mu]$. Suppose $P\colon X \times \borel{\rbb_+}
\to [0,1]$ is an $\ascr$-measurable family of
probability measures which satisfies \eqref{cc}.
Let $n\in \nbb$ be such that $\hsf_{\phi^n} <
\infty$ a.e.\ $[\mu]$. Then for every
$j=1,\ldots,n$, $0 < \hsf_{\phi^j} < \infty$
a.e.\ $[\mu]$ and \eqref{cc} holds with
$(\phi^j,E_j,P_j)$ in place of $(\phi,E,P)$,
where $P_j\colon X \times \borel{\rbb_+} \to
[0,1]$ is an $\ascr$-measurable family of
probability measures defined by
   \begin{align*}
P_j(x,\sigma) = P(x,\eta_j^{-1}(\sigma)), \quad x \in
X, \sigma \in \borel{\rbb_+},
   \end{align*}
with $\eta_j\colon \rbb_+ \ni t \mapsto t^j \in
\rbb_+$.
   \end{pro}
   \begin{proof}
It follows from Corollary \ref{cc2phinp} that
$C_{\phi}^j=C_{\phi^j}$ for $j=1,\ldots,n$. This
together with \cite[Section 6]{b-j-j-sC} and
\eqref{gokr} implies that $0 < \hsf_{\phi^j} <
\infty$ a.e.\ $[\mu]$ for $j=1,\ldots,n$. Note
that if $j\in \{1,\ldots,n\}$, then \eqref{cc}
holds with $(\phi^j,E_j,P_j)$ in place of
$(\phi,E,P)$ if and only if
   \begin{align} \label{ccj}
\esf_j(P(\cdot,\sigma))(x) = \frac{\int_{\sigma} t^j
P(\phi^j(x),\D t)}{\hsf_{\phi^j}(\phi^j(x))} \text{
for $\mu$-a.e.\ $x\in X$}, \quad \sigma \in
\borel{\rbb_+}.
   \end{align}
We use induction to prove that \eqref{ccj} holds
for every $j\in \{1,\ldots,n\}$. The case of
$j=1$ is obvious. Assume that $n\Ge 2$ and
\eqref{ccj} holds for a fixed $j\in\{1,
\ldots,n-1\}$. Then, by \eqref{fifi} and
\eqref{hf0-0} applied to $\phi^j$ in place of
$\phi$, we deduce from \eqref{ccj} that
   \begin{align}  \label{DDD}
\hsf_{\phi^j}(x) \cdot
(\esf_j(P(\cdot,\sigma))\circ \phi^{-j})(x) =
\int_{\sigma} t^j P(x,\D t) \text{ for
$\mu$-a.e.\ $x\in X$}.
   \end{align}
Applying Lemma \ref{phi-n}(ii) with $j$ in place
of $n$ and using \eqref{compos} and \eqref{fifi},
we see that
   \allowdisplaybreaks
   \begin{align*}
\hsf_{\phi^{j+1}}(\phi^{j+1}(x)) & \cdot
\esf_{j+1}(P(\cdot, \sigma))(x)
   \\
&\hspace{.8ex} = \hsf_{\phi}(\phi^{j+1}(x)) \cdot
\big(\esf\big(\hsf_{\phi^j} \cdot
\esf_j(P(\cdot,\sigma))\circ \phi^{-j}\big) \circ
\phi^j\big)(x)
   \\
&\overset{\eqref{DDD}} =
\hsf_{\phi}(\phi^{j+1}(x)) \cdot
\Big(\esf\Big(\int_{\sigma} t^j P(\cdot, \D
t)\Big) \circ \phi^j\Big)(x)
   \\
&\hspace{.5ex} \overset{(\dag)}=
\hsf_{\phi}(\phi^{j+1}(x)) \cdot
\frac{\int_\sigma t^{j+1} P(\phi^{j+1}(x),\D
t)}{\hsf_{\phi}(\phi^{j+1}(x))}
   \\
&\hspace{.8ex} = \int_{\sigma} t^{j+1}
P(\phi^{j+1}(x), \D t) \text{ for $\mu$-a.e.\ $x
\in X$,}
   \end{align*}
where the equality $(\dag)$ follows from Lemma
\ref{jeden}(i) and \eqref{compos}. Hence, \eqref{ccj}
holds for $j+1$ in place of $j$. This completes the
proof.
   \end{proof}
  \subsection{\label{scc}The strong consistency condition}
Under the assumptions of \eqref{standa2}, we say
that $P$ satisfies the {\em strong consistency
condition} if
   \begin{align} \tag{SCC} \label{consistB}
P(x, \sigma) = \frac{\int_{\sigma} t P(\phi(x),\D
t)}{\hsf_\phi(\phi(x))} \text{ for $\mu$-a.e.\ $x \in
X$}, \quad \sigma \in \borel{\rbb_+}.
   \end{align}
Some characterizations of \eqref{consistB} can be
easily obtained by adapting Lemma \ref{jeden} and its
proof to the present context. It is clear that $P$
satisfies \eqref{consistB} if and only if it satisfies
\eqref{cc} and the following equality
   \begin{align} \label{esfppx}
\esf(P(\cdot,\sigma))(x) = P(x,\sigma) \text{ for
$\mu$-a.e.\ $x \in X$}, \quad \sigma \in
\borel{\rbb_+}.
   \end{align}
Of course, \eqref{esfppx} is valid if $\phi^{-1}(\ascr)
= \ascr$. The latter holds if $\phi$ is injective and
$\ascr$-bimeasurable (i.e., $\phi$ is
$\ascr$-measurable and $\phi(\varDelta) \in \ascr$ for
every $\varDelta\in \ascr$). In particular, this is the
case for matrix symbols (cf.\ Section \ref{tmc}). Note
also that each quasinormal composition operator
satisfies \eqref{consistB} with some $P$ (cf.\
Proposition \ref{quasi-auto}).

Now we show that if a measurable family $P\colon
X \times \borel{\rbb_+} \to [0,1]$ satisfies
\eqref{consistB}, then all negative moments of
the measure $P(x,\cdot)$ are finite for
$\mu$-a.e.\ $x \in X$.
   \begin{pro} \label{trzy}
Suppose \eqref{standa2} and \eqref{consistB}
hold. Then $P(x,\{ 0\})=0$ for $\mu$-a.e.\ $x \in
X$ and the following equalities are valid for
every Borel function $f \colon \rbb_+ \to
\rbop${\em :}
   \begin{align} \label{wzjw}
\int_0^\infty f(t) P(x, \D t) = \frac{\int_0^\infty
f(t) t^n P(\phi^n(x),\D t)}{\prod_{j=1}^n
\hsf_\phi(\phi^j(x))} \text{ for $\mu$-a.e.\ $x\in
X$}, \quad n\in \nbb.
   \end{align}
In particular, the following equalities hold for
all $n\in \nbb$ and $m\in \zbb${\em :}
   \begin{enumerate}
   \item[(i)] $\int_{\sigma}
t^m P(x,\D t) = \frac{\int_{\sigma} t^{m+n}
P(\phi^n(x),\D t)} {\prod_{j=1}^n
\hsf_{\phi}(\phi^j(x))}$ for $\mu$-a.e.\ $x \in X$ and
every $\sigma \in \borel{\rbb_+}$,
   \item[(ii)] $\int_{\sigma}
t^m P(x,\D t) = \frac{\int_{\sigma} t^{m+n}
P(\phi^n(x),\D t)} {\int_0^\infty t^n P(\phi^n(x),\D
t)}$ for $\mu$-a.e.\ $x \in X$ and every $\sigma \in
\borel{\rbb_+}$,
   \item[(iii)] $\int_{\sigma} \frac{1}{t^n} P(x,\D t)
= \frac{P(\phi^n(x),\sigma)}{\prod_{j=1}^n
\hsf_{\phi}(\phi^j(x))}$ for $\mu$-a.e.\ $x \in X$ and
every $\sigma \in \borel{\rbb_+}$,
   \item[(iv)] $\int_0^\infty t^n P(\phi^n(x),\D t)
= \prod_{j=1}^n \hsf_{\phi}(\phi^j(x))$ for
$\mu$-a.e.\ $x \in X$,
   \item[(v)] $\int_0^\infty \frac{1}{t^n} P(x,\D t)
= \frac{1}{\prod_{j=1}^n \hsf_{\phi}(\phi^j(x))}$ for
$\mu$-a.e.\ $x \in X$.
   \end{enumerate}
Moreover, if $\hsf_{\phi} > 0$ a.e.\ $[\mu]$, then
$E(\hsf_{\phi^n}) = \hsf_{\phi^n}$ a.e.\ $[\mu]$ for
every $n\in \zbb_+$.
   \end{pro}
   \begin{proof}
That $P(x,\{0\})=0$ for $\mu$-a.e.\ $x \in X$ follows
directly from \eqref{consistB}. Using repeatedly
\eqref{consistB} with appropriate substitutions (cf.\
\eqref{compos}), we get
   \begin{align*}
P(x, \sigma) = \frac{\int_{\sigma} t P(\phi(x),\D
t)}{\hsf_\phi(\phi(x))} = \frac{\int_{\sigma} t^2
P(\phi^2(x),\D t)}{\hsf_\phi(\phi(x)) \hsf_\phi(\phi^2
(x))} = \ldots = \frac{\int_{\sigma} t^n
P(\phi^n(x),\D t)}{\prod_{j=1}^n \hsf_\phi(\phi^j(x))}
   \end{align*}
for $\mu$-a.e.\ $x \in X$ whenever $n\in\nbb$ and
$\sigma \in \borel{\rbb_+}$. Hence, by applying
\cite[Theorem 1.29]{Rud}, we get \eqref{wzjw}.

Substituting $f(t)=t^m \chi_{\sigma}(t)$ into
\eqref{wzjw}, we get (i). Applying (i) to $m=-n$ we
obtain (iii). In turn, applying (i) to $m=0$ and
$\sigma=\rbb_+$, we get (iv). Combining (i) and (iv)
gives (ii). Finally, (v) follows from (iii), applied
to $\sigma=\rbb_+$.

To show the ``moreover'' part, assume that $\hsf_{\phi}
> 0$ a.e.\ $[\mu]$. Arguing as in the proof of Lemma \ref{jeden},
we infer from \eqref{esfppx} that for every Borel
function $f \colon \rbb_+ \to \rbop$,
   \begin{align*}
\esf\bigg(\int_0^\infty f(t) P(\cdot,\D t)\bigg)(x) =
\int_0^\infty f(t) P(x,\D t) \text{ for $\mu$-a.e.\ $x
\in X$.}
   \end{align*}
Substituting $f(t)=t^n$ and using Theorem \ref{sms} we
complete the proof.
   \end{proof}
   \begin{cor}
If \eqref{standa2} and \eqref{consistB} hold,
then for every $n \in \nbb$,
   \begin{align} \label{a+b}
\prod_{j=n+1}^{2n} \hsf_{\phi}(\phi^j(x)) \Le
\prod_{j=1}^n \hsf_{\phi}(\phi^j(x)) \Le
\int_0^\infty t^n P(x,\D t) \text{ for
$\mu$-a.e.\ $x \in X$.}
   \end{align}
   \end{cor}
   \begin{proof}
By Proposition \ref{trzy}(v) and the
Cauchy-Schwarz inequality, we have
   \begin{multline*}
\prod_{j=1}^n \hsf_{\phi}(\phi^j(x))
\int_0^\infty \frac{1}{t^n} P(x,\D t) = 1 =
\bigg(\int_0^\infty \sqrt{t^n} \,
\frac{1}{\sqrt{t^n}} P(x,\D t)\bigg)^2
   \\
\Le \int_0^\infty t^n P(x,\D t) \int_0^\infty
\frac{1}{t^n} P(x,\D t) \text{ for $\mu$-a.e.\ $x
\in X$}, \quad n\in \nbb.
   \end{multline*}
Hence, the right-hand inequality in \eqref{a+b}
holds. This, together with Proposition
\ref{trzy}(iv), implies the left-hand inequality
 in \eqref{a+b}.
   \end{proof}
In Proposition \ref{ehfn} we characterize the
circumstances under which the equalities
$E(\hsf_{\phi^n}) = \hsf_{\phi^n}$ a.e.\ $[\mu]$, $n
\in \zbb_+$, hold. It is worth mentioning that the
condition (iv) below resembles the formula (6.4) in
\cite[Lemma 6.2]{Bu-St2} which was proved for
$C_0$-semigroups of bounded composition operators with
bimeasurable symbols.
   \begin{pro} \label{ehfn}
Let $(X,\ascr,\mu)$ be a $\sigma$-finite measure space
and $\phi$ be a nonsingular transformation of $X$ such
that $\hsf_\phi < \infty$ a.e.\ $[\mu]$. Then the
following two conditions are equivalent{\em :}
   \begin{enumerate}
   \item[(i)] $\esf(\hsf_{\phi^n}) = \hsf_{\phi^n}$
a.e.\ $[\mu]$ for all $n \in \nbb$,
   \item[(ii)] $\hsf_{\phi^{n+1}} \circ \phi = \hsf_{\phi}
\circ \phi \cdot \hsf_{\phi^n}$ a.e.\ $[\mu]$ for all
$n \in \nbb$.
   \end{enumerate}
Moreover, if {\em (i)} holds, then the following
equalities are valid{\em :}
   \begin{enumerate}
   \item[(iii)] $\hsf_{\phi^{m+n}} \circ \phi^{n} =
\hsf_{\phi} \circ \phi \cdots \hsf_{\phi} \circ \phi^n
\cdot \hsf_{\phi^m}$ a.e.\ $[\mu]$ for all $m \in
\zbb_+$ and $n \in \nbb$,
   \item[(iv)] $\hsf_{\phi^{m+n}} \circ \phi^{n} =
\hsf_{\phi^n} \circ \phi^{n} \cdot \hsf_{\phi^m}$
a.e.\ $[\mu]$ for all $m \in \zbb_+$ and $n \in \nbb$,
   \item[(v)] $\hsf_{\phi^n} \circ \phi^{n} = \hsf_{\phi}
\circ \phi \cdots \hsf_{\phi} \circ \phi^n$ a.e.\
$[\mu]$ for all $n \in \nbb$,
   \item[(vi)] $\hsf_{\phi^{n+1}} \circ \phi^{n} = \hsf_{\phi}
\circ \phi^0 \cdots \hsf_{\phi} \circ \phi^n$ a.e.\
$[\mu]$ for all $n \in \zbb_+$.
   \end{enumerate}
   \end{pro}
   \begin{proof}
(i)$\Rightarrow$(ii) This is a direct consequence of
\eqref{hn+1fpre}.

(ii)$\Rightarrow$(i) Applying the operator of
conditional expectation to both sides of the
equality in (ii) and using \eqref{CE-3}, we get
$\hsf_{\phi^{n+1}} \circ \phi = \hsf_{\phi} \circ
\phi \cdot \esf(\hsf_{\phi^n})$ a.e.\ $[\mu]$ for
all $n \in \nbb$. This together with (ii) implies
that $\hsf_{\phi} \circ \phi \cdot \hsf_{\phi^n}
= \hsf_{\phi} \circ \phi \cdot
\esf(\hsf_{\phi^n})$ a.e.\ $[\mu]$ for all $n \in
\nbb$. Since $\hsf_{\phi} \circ \phi > 0$ a.e.\
$[\mu]$, we get (i).

Now assume that (i) is satisfied. By (ii), the
equality in (iii) is valid for $n=1$ and for all $m\in
\zbb_+$. Suppose that this equality holds for a fixed
$n\in \nbb$ and for all $m\in \zbb_+$. Since the
equality in (ii) is valid for $n=0$, we see that for
every $m\in \zbb_+$,
   \begin{multline*}
\hsf_{\phi^{m+(n+1)}} \circ \phi^{n+1} =
\hsf_{\phi^{(m+1)+n}} \circ \phi^n \circ \phi =
\hsf_{\phi} \circ \phi^2 \cdots \hsf_{\phi} \circ
\phi^{n+1} \cdot \hsf_{\phi^{m+1}} \circ \phi
   \\
\overset{\mathrm{(ii)}}= \hsf_{\phi} \circ \phi^2
\cdots \hsf_{\phi} \circ \phi^{n+1} \cdot \hsf_{\phi}
\circ \phi \cdot \hsf_{\phi^m} = \hsf_{\phi} \circ
\phi \cdots \hsf_{\phi} \circ \phi^{n+1} \cdot
\hsf_{\phi^m} \text{ a.e.\ $[\mu]$.}
   \end{multline*}
By induction, this implies (iii).

Substituting $m=0$ and $m=1$ into (iii) we get (v) and
(vi), respectively. Combining (iii) with (v) gives
(iv). This completes the proof.
   \end{proof}
The following is a direct consequence of Propositions
\ref{trzy} and \ref{ehfn}.
   \begin{cor}
If \eqref{standa2} and \eqref{consistB} are satisfied
and $\hsf_{\phi} > 0$ a.e.\ $[\mu]$, then
$\hsf_{\phi^{n+1}} \circ \phi = \hsf_{\phi} \circ \phi
\cdot \hsf_{\phi^n}$ a.e.\ $[\mu]$ for all $n \in
\nbb$.
   \end{cor}
   \begin{rem}
Under the assumptions of Proposition \ref{ehfn}, if
additionally $\phi$ is a bijection whose inverse
$\phi^{-1}$ is nonsingular (see \cite[Lemma
3.1(ii)]{Bu-St2} for the possibility of weakening this
assumption), then $\phi^{-1}(\ascr) = \ascr$ and thus,
by Proposition \ref{ehfn}(v),
   \begin{align*}
\hsf_{\phi^n} = \hsf_{\phi} \circ \phi^0 \cdots
\hsf_{\phi} \circ \phi^{-(n-1)} \text{ a.e.\ $[\mu]$},
\quad n\in\nbb.
   \end{align*}
This happens for composition operators with matrix
symbols (cf.\ Section \ref{tmc}).
   \end{rem}
The next observation is inspired by \cite[Remark
6.4]{Bu-St2}.
   \begin{rem}
Note that if \eqref{standa2} holds, the measure
$tP(x,\D t)$ is determinate for $\mu$-a.e.\ $x \in X$
and $H_{n+1} \circ \phi = \hsf_{\phi} \circ \phi \cdot
H_n$ a.e.\ $[\mu]$ for every $n\in \zbb_+$, where
$H_n(x)=\int_0^\infty t^n P(x,\D t)$, then
\eqref{consistB} is valid. Moreover, if $\hsf_{\phi}>0$
a.e.\ $[\mu]$, then $H_n = \hsf_{\phi^n}$ a.e.\ $[\mu]$
for every $n\in \zbb_+$. Indeed, take a set $X_0\in
\ascr$ of $\mu$-full measure such that for every $x\in
X_0$, the measure $tP(x,\D t)$ is determinate and
$H_{n+1}(\phi(x)) = \hsf_{\phi}(\phi(x)) H_n(x)$ for
every $n\in \zbb_+$. Then the measures $tP(\phi(x),\D
t)$ and $P(x,\D t)$ are determinate for every $x\in X_0
\cap \phi^{-1}(X_0)$ (cf.\ \cite[Lemma 2.1.1]{j-j-s0}).
Since, by our assumption, the $n$th moments of the
measures $tP(\phi(x),\D t)$ and $\hsf_{\phi}(\phi(x))
P(x,\D t)$ coincide for all $n\in \zbb_+$ and $x\in
X_0$, and $\mu(X \setminus (X_0 \cap
\phi^{-1}(X_0)))=0$, we see that \eqref{consistB} is
satisfied. The ``moreover''part follows from Theorem
\ref{sms}.
   \end{rem}
We conclude this section by showing that for bounded
subnormal composition operators condition (i) of
Proposition \ref{ehfn} holds if and only if the
representing measures of
$\{\hsf_{\phi^n}(x)\}_{n=0}^\infty$, $x\in X$, form a
measurable family which satisfies \eqref{consistB}.
   \begin{pro} \label{eqscc}
Suppose \eqref{standa2} holds, $C_{\phi} \in
\ogr{L^2(\mu)}$ and
   \begin{align} \label{numerek-1}
   \begin{minipage}{65ex}
for $\mu$-a.e.\ $x\in X$,
$\{\hsf_{\phi^n}(x)\}_{n=0}^\infty$ is a Stieltjes
moment sequence with a representing measure
$P(x,\cdot)$.
   \end{minipage}
   \end{align}
Then $P$ satisfies \eqref{cc} and the following three
conditions are equivalent{\em :}
   \begin{enumerate}
   \item[(i)] $P$ satisfies \eqref{consistB},
   \item[(ii)] $\esf(P(\cdot,\sigma))(x) = P(x,\sigma)$ for
$\mu$-a.e.\ $x\in X$ and for every $\sigma \in
\borel{\rbb_+}$,
   \item[(iii)] $\esf(\hsf_{\phi^n}) =
\hsf_{\phi^n}$ a.e.\ $[\mu]$ for every $n\in \zbb_+$.
   \end{enumerate}
   \end{pro}
   \begin{proof}
First we show that $P$ satisfies \eqref{cc}. In view
of \eqref{numerek-1} and Lambert's criterion (see
\cite{lam1}; see also \cite[Theorem 3.4]{Bu-St1}),
$C_{\phi}$ is subnormal. By \cite[Theorem 9d]{ha-wh}
and Theorem \ref{bsubn}, $\hsf_{\phi} > 0$ a.e.\
$[\mu]$ and there exists an $\ascr$-measurable family
$\widetilde P\colon X \times \borel{\rbb_+} \to [0,1]$
of probability measures which satisfies {\em
\eqref{cc}} (with $\widetilde P$ in place of $P$), and
which has the property that the closed support of
$\widetilde P(x,\cdot)$ is contained in
$[0,\|C_{\phi}\|^2]$ for $\mu$-a.e.\ $x\in X$. It
follows from \eqref{numerek-1} and Theorem \ref{sms}
that the $n$th moments of the measures $P(x,\cdot)$
and $\widetilde P(x,\cdot)$ coincide for every
$n\in\zbb_+$ and for $\mu$-a.e.\ $x\in X$. Since any
Borel measure on $\rbb_+$ with compact support is
determinate, we conclude that $\widetilde
P(x,\cdot)=P(x,\cdot)$ for $\mu$-a.e. $x\in X$. Hence
$P$ satisfies \eqref{cc}.

(i)$\Leftrightarrow$(ii) This is clear, because $P$
satisfies \eqref{cc}.

(ii)$\Rightarrow$(iii) Apply the ``moreover'' part of
Proposition \ref{trzy}.

(iii)$\Rightarrow$(ii) We partially follow the proof
of \cite[Theorem 3.4]{Bu-St1}. Without loss of
generality we may assume that $\hsf_{\phi^0}=1$,
$\hsf_{\phi^n}$ is $\phi^{-1}(\ascr)$-measurable and
$0 \Le \hsf_{\phi^n} < \infty$ for all $n\in \zbb_+$.
Set $Y=\bigcap_{n=0}^\infty \big\{x\in X\colon
\hsf_{\phi^{2(n+1)}}(x) \Le \|C_{\phi}\|^4
\hsf_{\phi^{2n}}(x)\big\}$. It is clear that $Y \in
\phi^{-1}(\ascr)$. Since for every $f\in L^2(\mu)$ and
for all $n\in \zbb_+$,
   \begin{align*}
\int_X |f|^2 \hsf_{\phi^{2(n+1)}} \D \mu =
\|C_{\phi}^2 C_{\phi}^{2n} f\|^2 \Le \|C_{\phi}\|^4
\|C_{\phi}^{2n} f\|^2 = \|C_{\phi}\|^4 \int_X |f|^2
\hsf_{\phi^{2n}} \D \mu,
   \end{align*}
we deduce that $\mu(X\setminus Y)=0$. Given a nonempty
subset $W$ of $\cbb$, we define the subsets $Z_W$ and
$\widetilde Z_W$ of $X$ by
   \allowdisplaybreaks
   \begin{align*}
Z_W&=\bigcap_{n\in\zbb_+} \; \bigcap_{\lambda_0,
\ldots, \lambda_n \in W} \bigg\{x\in X\colon
\sum_{i,j=0}^n \hsf_{\phi^{i+j}}(x)
\lambda_i\bar\lambda_j \Ge 0\bigg\},
   \\
\widetilde Z_W&=\bigcap_{n\in\zbb_+} \;
\bigcap_{\lambda_0, \ldots, \lambda_n \in W}
\bigg\{x\in X\colon \sum_{i,j=0}^n
\hsf_{\phi^{i+j+1}}(x) \lambda_i\bar\lambda_j \Ge
0\bigg\}.
   \end{align*}
Let $S$ be a countable and dense subset of $\cbb$.
Noting that $Z_{\cbb}=Z_{S}$ and $\widetilde Z_{\cbb}
= \widetilde Z_{S}$, we deduce that $Z_{\cbb},
\widetilde Z_{\cbb} \in \phi^{-1}(\ascr)$. It follows
from \eqref{numerek-1} and \cite[Theorem 6.2.5]{b-c-r}
that $\mu(X\setminus Z_{\cbb}) = \mu(X\setminus
\widetilde Z_{\cbb}) = 0$. Set $\varOmega=Y \cap
Z_{\cbb} \cap \widetilde Z_{\cbb}$. Then $\varOmega\in
\phi^{-1}(\ascr)$ and $\mu(X\setminus \varOmega)=0$.
Applying \cite[Theorem 6.2.5]{b-c-r} and \cite[Theorem
2]{Sza}, we see that for every $x\in \varOmega$ there
exists a Borel probability measure $\vartheta_x$ on
$K:=[0,\|C_{\phi}\|^2]$ such that $\int_K t^n
\vartheta_x(\D t) = \hsf_{\phi^n}(x)$ for all $n\in
\zbb_+$. It follows from Lemma \ref{mommea} that the
function $\varOmega\ni x \mapsto
\vartheta_x(\sigma)\in [0,1]$ is
$\phi^{-1}(\ascr)$-measurable for every $\sigma \in
\borel{K}$. Define $\hat P\colon X \times
\borel{\rbb_+} \to [0,1]$ by
   \begin{align*}
\hat P(x,\sigma) =
   \begin{cases}
   \vartheta_x(\sigma \cap K) & \text{if } x\in
\varOmega,
   \\
   \delta_0(\sigma) & \text{otherwise,}
   \end{cases}
\quad \sigma \in \borel{\rbb_+}.
   \end{align*}
It is clear that $\hat P$ is a
$\phi^{-1}(\ascr)$-measurable family of probability
measures. By \eqref{numerek-1}, the $n$th moments of
the measures $P(x,\cdot)$ and $\hat P(x,\cdot)$
coincide for all $n\in \zbb_+$ and for $\mu$-a.e.\
$x\in X$. Hence $P(x,\cdot) = \hat P(x,\cdot)$ for
$\mu$-a.e.\ $x\in X$. This yields
   \begin{align*}
\esf(P(\cdot,\sigma))(x) = \esf(\hat
P(\cdot,\sigma))(x) = \hat P(x,\sigma) = P(x,\sigma)
   \end{align*}
for $\mu$-a.e.\ $x\in X$ and for all $\sigma\in
\borel{\rbb_+}$. This completes the proof.
   \end{proof}
   \section{APPLICATIONS AND EXAMPLES}
   \subsection{\label{tmc}The matrix case}
   Fix a positive integer $\kappa$. Denote by
$\omega_\kappa$ the $\kappa$-dimen\-sional Lebesgue
measure on the $\kappa$-dimensional Euclidean space
$\rbb^\kappa$. We begin by introducing a class of
densities on $\rbb^\kappa$. Denote by $\hscr$ the set
of all entire functions $\gamma$ on $\cbb$ of the form
   \begin{align} \label{reprez}
\gamma(z) = \sum_{n=0}^\infty a_n z^n, \quad z \in
\cbb,
   \end{align}
where $a_n$ are nonnegative real numbers and $a_k
> 0$ for some $k \Ge 1$. Let $\gamma$ be in
$\hscr$ and $\|\cdot\|$ be a norm on
$\rbb^\kappa$ induced by an inner product. Define
the $\sigma$-finite Borel measure $\mu_\gamma$ on
$\rbb^\kappa$ by $\mu_\gamma(\D x) =
\gamma(\|x\|^2) \omega_\kappa(\D x)$. Given a
linear transformation $A$ of $\rbb^\kappa$, one
can verify that the composition operator $C_A$ in
$L^2(\mu_\gamma)$ is well-defined if and only if
$A$ is invertible. If this is the case, then
(cf.\ \cite[equation (2.1)]{sto})
   \begin{align} \label{pochodna}
\hsf_A (x) = \frac{1}{|\det A|}
\frac{\gamma(\|A^{-1}x\|^2)}{\gamma(\|x\|^2)}, \quad x
\in \rbb^\kappa \setminus \{0\}.
   \end{align}
Hence, each well-defined composition operator $C_A$ is
automatically densely defined and injective (because
$0 < \hsf_A <\infty$ a.e.\ $[\mu_\gamma]$). We refer
the reader to \cite{sto} for more information on this
class of operators (see \cite{ml} for the case of
Gaussian density).

The main result of this section will be preceded by an
auxiliary lemma concerning the measurability of
convolution powers of families of Borel measures on
$\rbb_+$. Given $n\in \nbb$ and a finite Borel measure
$\nu$ on $\rbb_+$, we define the $n$th multiplicative
convolution power $\nu^{*n}$ of $\nu$ by
   \begin{align}  \label{nsplot}
\nu^{*n}(\sigma) = \int_0^\infty \ldots \int_0^\infty
\chi_{\sigma}(t_1 \cdots t_n) \nu(\D t_1) \ldots
\nu(\D t_n), \quad \sigma \in \borel{\rbb_+}.
   \end{align}
We also set $\nu^{*0}(\sigma) = \chi_{\sigma}(1)$ for
$\sigma \in \borel{\rbb_+}$. The standard
measure-theoretic argument shows that for every Borel
function $f \colon \rbb_+ \to \rbop$,
   \begin{align}   \label{intsplot}
\int_0^\infty f(t) \nu^{*n}(\D t) = \int_0^\infty
\ldots \int_0^\infty f(t_1 \cdots t_n) \nu(\D t_1)
\ldots \nu(\D t_n), \quad n \in \nbb.
   \end{align}
   \begin{lem} \label{Borem}
Let $(X,\ascr)$ be a measurable space and
$\{\nu_x\colon x \in X\}$ be a family of finite Borel
measures on $\rbb_+$ such that the function $X\ni x
\mapsto \nu_x(\sigma) \in \rbb_+$ is
$\ascr$-measurable for all $\sigma \in
\borel{\rbb_+}$. Then the function $X \ni x \mapsto
\nu_x^{*n}(\sigma)$ is $\ascr$-measurable for all
$\sigma \in \borel{\rbb_+}$ and $n\in \zbb_+$.
   \end{lem}
   \begin{proof}
We can assume that $n\Ge 2$. Suppose first that there
exists $R\in \rbb_+$ such that the closed support of
each measure $\nu_x$ is contained in $K:=[0,R]$. The
standard measure-theoretic argument shows that the
function $X\ni x \mapsto \int_0^\infty t^m \nu_x(\D t)
\in \rbb_+$ is $\ascr$-measurable for all $m\in\zbb_+$.
It follows from \eqref{intsplot} that
   \begin{align}  \label{dyrdyr}
\int_0^\infty t^m \nu_x^{*n}(\D t) =
\Big(\int_0^\infty t^m \nu_x(\D t)\Big)^n, \quad x \in
X, m \in \zbb_+.
   \end{align}
By \cite[Corollary 3.4]{2xSt2}, the closed
support of $\nu_x^{*n}$ is contained in $[0,R^n]$
for every $x\in X$. Note that $\nu_x^{*n}=0$
whenever $\nu_x(\rbb_+)=0$. Since $X\ni x \mapsto
\nu_x(\rbb_+) \in \rbb_+$ is $\ascr$-measurable,
we deduce from \eqref{dyrdyr} and Lemma
\ref{mommea} that the function $X \ni x \mapsto
\nu_x^{*n}(\sigma) \in \rbb_+$ is
$\ascr$-measurable for all $\sigma \in
\borel{\rbb_+}$.

Coming back to the general case, we set
$\nu_{k,x}(\sigma) = \nu_{x}(\sigma\cap [0,k])$ for
$x\in X$, $\sigma \in \borel{\rbb_+}$ and $k\in \nbb$.
Applying the above to $\{\nu_{k,x}\colon x\in X\}$,
$k\in \nbb$, and using \eqref{nsplot}, we complete the
proof.
   \end{proof}
Now we show that if a linear transformation $A$ of
$\rbb^\kappa$ is normal\footnote{\;Equivalently:\
$VAV^{-1}$ is normal in $(\rbb^\kappa, |\cdot|)$, where
$|\cdot|$ is the Euclidean norm and $V$ is a positive
invertible operator in $(\rbb^\kappa, |\cdot|)$ such
that $\|x\|= |Vx|$ for all $x\in X$ (cf.\ \cite[p.\
310]{sto}).} in $(\rbb^\kappa, \|\cdot\|)$, then the
composition operator $C_A$ is subnormal in
$L^2(\mu_\gamma)$. As shown in \cite[Theorem 2.5]{sto},
the converse implication is true for bounded
composition operators (see also \cite[Theorem
3.6]{Bu-St1} for the case of families of composition
operators). It is an open question whether this is true
for unbounded operators.
   \begin{thm} \label{matrical}
Let $\gamma$ be in $\hscr$, $\|\cdot\|$ be a norm
on $\rbb^\kappa$ induced by an inner product and
$A$ be an invertible linear transformation of
$\rbb^\kappa$. If $A$ is normal in $(\rbb^\kappa,
\|\cdot\|)$, then $C_A$ is subnormal in
$L^2(\mu_\gamma)$.
   \end{thm}
   \begin{proof}
Let $(\cbb^\kappa, \|\cdot\|_{\mathrm c})$ be the
Hilbert space complexification of $(\rbb^\kappa,
\|\cdot\|)$ with the inner product
$\is{\cdot}{-}_{\mathrm c}$ and $A_{\mathrm c}$ be the
corresponding complexification of $A$. Then
$A_{\mathrm c}$ is invertible and normal in
$(\cbb^\kappa, \|\cdot\|_{\mathrm c})$. Denote by $E$
the spectral measure of $|A_{\mathrm c}|^{-2}$. For
$x\in \rbb^{\kappa}$, define the finite Borel measure
$\nu_x$ on $\rbb_+$ by $\nu_x(\sigma) =
\is{E(\sigma)x}{x}_{\mathrm c}$ for $\sigma \in
\borel{\rbb_+}$. Since $A_{\mathrm c}$ is normal, we
see that $A_{\mathrm c}E(\cdot)=E(\cdot)A_{\mathrm
c}$, which yields
   \begin{align}  \label{overt}
\nu_{Ax}(\sigma) = \is{|A_{\mathrm c}|^2
E(\sigma)x}{x}_{\mathrm c} = \is{(|A_{\mathrm
c}|^{-2})^{-1} E(\sigma)x}{x}_{\mathrm c} =
\int_{\sigma} \frac{1}{t} \nu_x (\D t)
   \end{align}
for all $x \in \rbb^\kappa$ and $\sigma \in
\borel{\rbb_+}$. Noting that the function $\rbb^\kappa
\ni x \mapsto \nu_x(\sigma) \in \rbb_+$ is continuous
for every $\sigma \in \borel{\rbb_+}$ and applying
Lemma \ref{Borem}, we deduce that the mapping $P\colon
\rbb^\kappa \times \borel{\rbb_+} \to [0,1]$ given by
   \begin{align} \label{straszny}
P(x,\sigma) =
   \begin{cases}
\displaystyle{\frac{1}{\gamma(\|x\|^2)}
\sum_{n=0}^\infty a_n \nu_x^{*n} (|\det A| \cdot
\sigma)} & \text{ if } x \neq 0,
   \\[2ex]
\displaystyle{\chi_\sigma(1)} & \text{ if } x=0,
   \end{cases}
\quad \sigma\in \borel{\rbb_+},
   \end{align}
is a $\borel{\rbb^\kappa}$-measurable family of
probability measures, where
$\{a_n\}_{n=0}^\infty$ is as in \eqref{reprez}
and $|\det A| \cdot \sigma := \{|\det A| \,
t\colon t \in \sigma\}$.

We claim that $P$ satisfies \eqref{consistB}. For
this, note that \allowdisplaybreaks
   \begin{align}  \notag
\int_{\sigma} t \, \nu_{Ax}^{*n}(\D t) &
\overset{\eqref{intsplot}} = \int_0^\infty \ldots
\int_0^\infty \chi_{\sigma} (t_1 \cdots t_n) \cdot t_1
\cdots t_n \nu_{Ax}(\D t_1) \ldots \nu_{Ax}(\D t_n)
   \\ \notag
& \overset{\eqref{overt}} = \int_0^\infty \ldots
\int_0^\infty \chi_{\sigma} (t_1 \cdots t_n)
\nu_{x}(\D t_1) \ldots \nu_{x}(\D t_n)
   \\  \label{duzolinii}
& \overset{\eqref{nsplot}} = \nu_{x}^{*n}
(\sigma), \quad x \in \rbb^\kappa, \, \sigma\in
\borel{\rbb_+}, \, n \in \nbb.
   \end{align}
Now, by applying the measure transport theorem
and \eqref{pochodna}, we get\footnote{\;The
notation $\nu_{Ax}^{*n}(|\det A| \cdot \D t)$ is
used when integrating with respect to the measure
   \begin{align*}
\text{$\borel{\rbb_+} \ni \sigma \mapsto
\nu_{Ax}^{*n}(|\det A| \cdot \sigma) \in \rbb_+$.}
   \end{align*}}
   \allowdisplaybreaks
   \begin{align*}
\frac{1}{\hsf_{A}(A(x))} \int_{\sigma} t \, P(Ax,
\D t) & \overset{\eqref{straszny}} =
\frac{1}{\gamma(\|x\|^2)} \sum_{n=0}^\infty a_n
\int_\sigma |\det A| \cdot t \,
\nu_{Ax}^{*n}(|\det A| \cdot \D t)
   \\
&\hspace{.7ex}= \frac{1}{\gamma(\|x\|^2)}
\sum_{n=0}^\infty a_n \int_{|\det A|\cdot \sigma} t \,
\nu_{Ax}^{*n}(\D t)
   \\
&\overset{\eqref{duzolinii}}=
\frac{1}{\gamma(\|x\|^2)} \sum_{n=0}^\infty a_n
\nu_{x}^{*n}(|\det A| \cdot \sigma)
   \\
& \hspace{.7ex}= P(x,\sigma), \quad x\in \rbb^\kappa
\setminus \{0\}, \sigma \in \borel{\rbb_+},
   \end{align*}
which proves our claim. Applying Theorem
\ref{glowne} completes the proof.
   \end{proof}
The part of the proof of Theorem \ref{matrical}
regarding the $\borel{\rbb^\kappa}$-measurability
of the family $P$ is based on Lemma \ref{Borem}.
Although in the matrix case this can be justified
in an elementary way, Lemma \ref{Borem} is much
more general and fits well into the context of
Lemma \ref{mommea}.

   We conclude this section by noticing that
Theorem \ref{matrical} remains true for
composition operators whose symbols are
invertible $\cbb$-linear transformations of
$\cbb^\kappa$. The proof goes along the same
lines with one exception, namely we have to
replace $|\det A|$ by $|\det A|^2$ (cf.\
\cite[Section 3]{sto}).
    \subsection{\label{dc}The discrete case}
In this section we assume that $(X, \ascr,\mu)$
is a {\em discrete measure space}, i.e., $X$ is a
countably infinite set, $\ascr = 2^X$ and $\mu$
is a $\sigma$-finite measure on $\ascr$ (or
equivalently, $\mu(\{x\}) < \infty$ for every
$x\in X$). Let $\phi$ be a transformation of $X$.
Clearly, $\phi$ is $\ascr$-measurable. To
simplify notation, we write $\mu(x)=\mu(\{x\})$
and $\phii{x} =\{y \in \phi^{-1}(\{x\})\colon
\mu(y) > 0\}$ for $x\in X$. The transformation
$\phi$ is nonsingular if and only if
$\mu(\phi^{-1}(\{x\}))=0$ for every $x\in X$ such
that $\mu(x)=0$. Hence, if $\mu(x)> 0$ for every
$x\in X$, then $\phi$ is nonsingular. Assume that
$\phi$ is nonsingular. Setting
$\hsf_{\phi^n}(x)=1$ if $\mu(x)=0$, we see that
   \begin{align} \label{hfin}
\hsf_{\phi^n}(x) =
\frac{\mu(\phi^{-n}{(\{x\})})}{\mu(x)}, \quad x
\in X,\, n \in \zbb_+.
   \end{align}
(Recall that, according to our convention,
$\frac{0}{0}=1$). Thus $\hsf_{\phi} < \infty$ a.e.\
$[\mu]$ if and only if
$\mu(\phi^{-1}{(\{x\})})<\infty$ for every $x\in X$
with $\mu(x)>0$. The positivity of $\hsf_{\phi}$ and
surjectivity of $\phi$ relates to each other as
follows.
   \begin{lem} \label{phix}
If $\mu(x)>0$ for all $x\in X$, then
$\hsf_{\phi}(x)>0$ for all $x\in X$ if and only if
$\phi(X)=X$.
   \end{lem}
   \begin{proof}
Note that for every $x\in X$, $\hsf_{\phi}(x)>0$ if
and only if $\phi^{-1}(\{x\}) \neq \varnothing$.
   \end{proof}

Assume that $\hsf_{\phi} < \infty$ a.e.\ $[\mu]$.
Since $X = \bigsqcup_{x\in \phi(X)} \phi^{-1}(\{x\})$,
we get
   \begin{align*}
\phi^{-1}{(\ascr)} = \bigg\{\bigsqcup_{x\in \varDelta}
\phi^{-1}(\{x\}) \colon \varDelta \subseteq
\phi(X)\bigg\},
   \end{align*}
where the symbol ``$\bigsqcup$'' is used to denote
pairwise disjoint union of sets. Note that a function
$f$ on $X$ taking values in $\rbop$ or in $\cbb$ is
$\phi^{-1}(\ascr)$-measurable if and only if $f$ is
constant on $\phi^{-1}(\{x\})$ for every $x\in
\phi(X)$. Setting $\esf(f)=1$ on $\phi^{-1}(\{x\})$ if
$\mu(\phi^{-1}(\{x\}))=0$, we see that
   \begin{align} \label{condex}
\esf(f) = \sum_{x\in \phi(X)}
\frac{\int_{\phi^{-1}(\{x\})} f \D
\mu}{\mu(\phi^{-1}(\{x\}))} \cdot
\chi_{\phi^{-1}(\{x\})}
   \end{align}
for every function $f\colon X \to \rbop$. By
linearity this equality holds a.e.\ $[\mu]$ for
every $f \in L^2(\mu)$ as well.

Now we investigate the consistency condition
\eqref{cc} in the context of discrete measure spaces.
Since $\ascr = 2^X$, we can abbreviate the expression
``an $\ascr$-measurable family of probability
measures'' to ``a family of probability measures''.
   \begin{lem} \label{chcc}
Let $(X,\ascr,\mu)$ be a discrete measure space,
$P\colon X \times \borel{\rbb_+} \to [0,1]$ be a
family of probability measures and $\phi$ be a
nonsingular transformation of $X$ such that
$\hsf_{\phi} < \infty$ a.e.\ $[\mu]$. Then \eqref{cc}
is equivalent to each of the following conditions{\em
:}
      \begin{enumerate}
   \item[(i)] for every $x\in X$ such that
$\mu(\phi^{-1}(\{x\}))>0$, the following holds{\em :}
   \begin{align*}
\int_\sigma t P(x,\D t) = \sum_{y\in \phii{x}} \;
\frac{\mu(y)}{\mu(x)} \cdot P(y, \sigma), \quad \sigma
\in \borel{\rbb_+},
   \end{align*}
   \item[(ii)] for every $x\in X$ such that
$\mu(\phi^{-1}(\{x\}))>0$, the following holds{\em :}
   \begin{align*}
   &P(y,\{0\})=0 \text{ for every $y \in \phii{x}$,
and}
   \\
   &P(x,\sigma) = \sum_{y \in \phii{x}} \;
\frac{\mu(y)}{\mu(x)} \cdot \int_\sigma \frac{1}{t}
P(y,\D t), \quad \sigma \in \borel{\rbb_+},
   \end{align*}
   \item[(iii)] for every $x\in X$ such that
$\mu(\phi^{-1}(\{x\}))>0$, the following holds{\em :}
   \begin{align} \notag
&P(y,\cdot) \ll P(x,\cdot) \text{ for every $y\in
\phii{x}$, and}
   \\ & t = \sum_{y \in \phii{x}} \;
\frac{\mu(y)}{\mu(x)} \cdot \frac{\D P(y,\cdot)}{\D
P(x,\cdot)}(t) \text{ for $P(x,\cdot)$-a.e.\ $t\in
\rbb_+$,} \notag
   \end{align}
   \item[(iv)]
for every $x\in X$ such that
$\mu(\phi^{-1}(\{x\}))>0$, the following holds{\em :}
   \begin{align*}
&P(x,\{0\})=0, P(y,\cdot) \ll P(x,\cdot) \text{ for
every $y\in \phii{x}$, and}
   \\
&1 = \sum_{y \in \phii{x}} \; \frac{\mu(y)}{\mu(x)}
\cdot \frac{1}{t} \cdot \frac{\D P(y,\cdot)}{\D
P(x,\cdot)}(t) \text{ for $P(x,\cdot)$-a.e.\ $t\in
\rbb_+$.}
   \end{align*}
   \end{enumerate}
   \end{lem}
   \begin{proof}
Applying \eqref{hfin}, \eqref{condex} and the
decomposition $X = \bigsqcup_{x\in \phi(X)}
\phi^{-1}(\{x\})$, we deduce that \eqref{cc} is
equivalent to (i). In turn, employing
\eqref{hfin}, \eqref{condex} and Lemma
\ref{jeden}(ii), we verify that (i) is equivalent
to (ii). By the Radon-Nikodym theorem, (i) is
easily seen to be equivalent to (iii).

(ii)$\Rightarrow$(iv) Since (ii) implies (iii),
it suffices to show that $P(x,\{0\})=0$ whenever
$\mu(\phi^{-1}(\{x\}))>0$. Suppose that, on the
contrary, there exists $x\in X$ such that
$\mu(\phi^{-1}(\{x\}))>0$ and $P(x,\{0\})> 0$.
Since $\phi$ is nonsingular, we see that
$\mu(x)>0$. Hence $x \in \phii{\phi(x)}$, and
thus by (ii) $P(x,\{0\})=0$, a contradiction.

(iv)$\Rightarrow$(iii) Evident.
   \end{proof}
   The above preparation enables us to state a
discrete version of Theorem \ref{glowne}.
   \begin{thm} \label{glownedis}
Let $(X,\ascr,\mu)$ be a discrete measure space and
$\phi$ be a transformation of $X$ such that
   \begin{enumerate}
   \item[(i)] for every $x\in X$, $\mu(x)=0$ if and
only if $\mu(\phi^{-1}(\{x\}))=0$,
   \item[(ii)] $\mu(\phi^{-1}{(\{x\})})<\infty$
for every $x\in X$ such that $\mu(x)>0$.
   \end{enumerate}
Suppose there exists a family $P\colon X \times
\borel{\rbb_+} \to [0,1]$ of probability measures
which satisfies one of the equivalent conditions {\em
(i)} to {\em (iv)} of Lemma {\em \ref{chcc}}. Then
$C_{\phi}$ is subnormal.
   \end{thm}
   \begin{proof}
By \eqref{hfin}, the conjunction of the conditions (i)
and (ii) is equivalent to requiring that $\phi$ is
nonsingular and $0 < \hsf_{\phi} < \infty$ a.e.\
$[\mu]$. Combining \cite[Proposition 6.2]{b-j-j-sC}
and \eqref{gokr}, we see that $0 < \hsf_{\phi} <
\infty$ a.e.\ $[\mu]$ if and only if $C_\phi$ is
injective and densely defined. Hence, by applying
Lemma \ref{chcc} and Theorem \ref{glowne} with
$\zeta(t)=t$, we complete the proof.
   \end{proof}
It is worth mentioning that if $\phi$ is an
injective nonsingular transformation of a
discrete measure space, then, by \eqref{hfin},
$\hsf_{\phi^n} < \infty$ a.e.\ $[\mu]$ for every
$n\in \nbb$, and thus, by \cite[Corollary 4.5 and
Theorem 4.7]{b-j-j-sC}, $\dzn{C_{\phi}}$ is a
core for $C_{\phi}^n$ for every $n\in \zbb_+$.
Moreover, the conditional expectation
$\esf(\cdot)$ acts as the identity map (see the
paragraph just below \eqref{esfppx}). Hence
\eqref{cc} becomes \eqref{consistB}. This
observation enables us to apply the results of
Section \ref{scc}. In particular, combining
Propositions \ref{trzy}(i) and \ref{ehfn}(v), we
get the following.
   \begin{pro}
Let $(X,\ascr,\mu)$ be a discrete measure space and
$\phi$ be an injective nonsingular transformation of
$X$. Assume that $P\colon X \times \borel{\rbb_+} \to
[0,1]$ is a family of probability measures which
satisfies \eqref{cc}. Then
   \begin{enumerate}
   \item[(i)] $\int_\sigma t^n
P(\phi^n(x), \D t) = \hsf_{\phi^n}(\phi^n(x)) \cdot
P(x,\sigma)$ for all $\sigma \in \borel{\rbb_+}$,
$n\in \zbb_+$ and $x \in X$ such that $\mu(x)>0$.
   \end{enumerate}
Moreover, if $\mu(x)>0$ for every $x\in X$, then
   \begin{enumerate}
   \item[(ii)] $\int_\sigma t^n P(x, \D t) =
\hsf_{\phi^n}(x) \cdot
P\big((\phi^{n})^{-1}(x),\sigma\big)$ for all $\sigma
\in \borel{\rbb_+}$, $x \in \phi^n(X)$ and $n\in
\zbb_+$.
   \end{enumerate}
   \end{pro}
Below we will discuss the question of subnormality of
composition operators in $L^2$-spaces over discrete
measure spaces with injective symbols. This is done by
exploiting a model for such operators which is based
on \cite[Proposition 2.4]{StB}.
   \begin{rem}  \label{model}
Suppose $(X, \ascr, \mu)$ is a $\sigma$-finite measure
space such that $X$ is at most countable, $\ascr =
2^X$ and $\mu(x)>0$ for every $x\in X$. Let $\phi$ be
an injective transformation of $X$. We say that
$C_{\phi}$ is of {\em type} I if there exists $u\in X$
such that the mapping $\zbb_+ \ni n \to \phi^{n}(u)
\in X$ is bijective, of {\em type} II if $\phi$ is
bijective and there exists $u\in X$ such that the
mapping $\zbb \ni n \to \phi^{n}(u) \in X$ is
bijective, and of {\em type} III if there exist $u\in
X$ and $m \in \nbb$ such that the mapping $\{0,
\ldots, m-1\} \ni n \mapsto \phi^n (u) \in X$ is
bijective (note that then $\phi^m = \mathrm{id}_X$).
One can show that a composition operator of type I
cannot be subnormal (in fact, it is not hyponormal
because $C_{\phi} \chi_{\{u\}} = 0$ and $C_{\phi}^*
\chi_{\{u\}} \neq 0$), and it is unitarily
equivalent\footnote{\;via the unitary isomorphism
$U\colon \ell^2(\zbb_+) \to L^2(\mu)$ given by
$(Uf)(\phi^{n}(u)) = \frac{f(n)}
{\sqrt{\mu(\phi^{n}(u))}}$ for $n\in \zbb_+$ and $f\in
\ell^2(\zbb_+)$; see also \cite[Remark 3.1.4]{j-j-s};}
to the adjoint of an injective unilateral weighted
shift. A composition operator of type II is unitarily
equivalent to an injective bilateral weighted shift.
Hence, by applying Theorem \ref{glownedis}, we obtain
the Berger-Gellar-Wallen characterization of
subnormality of injective bilateral weighted shifts
(see \cite[Theorem 3.2]{b-j-j-sB} and note that
Theorem \ref{wsi} follows from Theorem
\ref{glownedis}). In turn, a composition operator of
type III is a bounded $m$th root of $I$ (because $\dim
L^2(\mu) < \infty$ and $\phi^m=\mathrm{id}_X$). Hence,
by Proposition \ref{potega}, it is subnormal if and
only if it is unitary. The latter happens if and only
if $\hsf_{\phi}=1$ (again because $\dim L^2(\mu) <
\infty$), or equivalently if and only if $X \ni x
\mapsto \mu(x) \in (0,\infty)$ is a constant function.
It follows from \cite[Proposition 2.4]{StB} and
Proposition \ref{orthsum} that if $\phi$ is an
arbitrary injective transformation of $X$, then there
exist $N \in \nbb \cup \{\infty\}$ and a sequence
$\{Y_n\}_{n=1}^N \subseteq \ascr(\phi)$ of pairwise
disjoint nonempty sets such that $X=\bigcup_{n=1}^N
Y_n$, each $C_{\phi_{Y_n}}$ is of one of the types I,
II or III, and $C_\phi$ is unitarily equivalent to
$\bigoplus_{n=1}^N C_{\phi_{Y_n}}$ (with the notation
as in Appendix \ref{AppC}). In view of the above
discussion, if $C_{\phi}$ is subnormal then there is
no summand of type I in the decomposition
$\bigoplus_{n=1}^N C_{\phi_{Y_n}}$, and thus
$C_{\phi}$ is unitarily equivalent to an orthogonal
sum of at most countably many operators, each of which
is either a subnormal injective bilateral weighted
shift or a unitary $m$th root ($m\Ge 1$) of the
identity operator on a finite dimensional space. On
the other hand, by Corollary \ref{orthsum-c}, an
orthogonal sum of at most countably many composition
operators of one of the types I, II or III is
unitarily equivalent to a composition operator
$C_{\phi}$ in an $L^2$-space over a $\sigma$-finite
measure space $(X,2^X,\mu)$ such that $X$ is at most
countable, $\mu(x)>0$ for every $x\in X$ and $\phi$ is
injective.
   \end{rem}
   \subsection{\label{lcs}Local consistency}
In this section we show that the ``local consistency
technique'' introduced in \cite[Lemma 4.1.3]{b-j-j-sA}
for weighted shifts on directed trees can be
implemented in the context of composition operators in
$L^2$-spaces over discrete measure spaces. The
non-discrete case does not seem to make sense. In what
follows we preserve the notation from Section
\ref{dc}.
   \begin{lem} \label{dcc}
Let $(X,\ascr,\mu)$ be a discrete measure space and
$\phi$ be a nonsingular transformation of $X$ such
that $\hsf_{\phi} < \infty$ a.e.\ $[\mu]$. Let $x \in
X$ be such that $\mu(\phi^{-1}(\{x\}))> 0$ and for
every $y \in \phii{x}$,
$\{\hsf_{\phi^n}(y)\}_{n=0}^\infty$ is a Stieltjes
moment sequence with a representing measure
$\vartheta_y$. Then the following assertions are
valid.
   \begin{enumerate}
   \item[(i)] If
   \begin{align} \label{ajajZenon}
\sum_{y \in \phii{x}} \frac{\mu(y)}{\mu(x)}
\int_0^\infty \frac {1}{t} \vartheta_y(\D t) \Le 1,
   \end{align}
then $\{\hsf_{\phi^n}(x)\}_{n=0}^\infty$ is a
Stieltjes moment sequence with a representing
measure $\widetilde\vartheta_x$ given by
   \begin{align} \label{war1}
\widetilde\vartheta_x(\sigma) = \sum_{y \in
\phii{x}} \frac{\mu(y)}{\mu(x)} \int_{\sigma}
\frac {1}{t} \vartheta_y(\D t) + \varepsilon_x
\cdot \delta_0(\sigma), \quad \sigma \in
\borel{\rbb_+},
   \end{align}
where
   \begin{align} \label{war2}
\varepsilon_x = 1 - \sum_{y \in \phii{x}}
\frac{\mu(y)}{\mu(x)} \int_0^\infty \frac {1}{t}
\vartheta_y(\D t).
   \end{align}
   \item[(ii)] If $\{\hsf_{\phi^n}(x)\}_{n=0}^\infty$
is a Stieltjes moment sequence, and
$\{\hsf_{\phi^{n+1}}(x)\}_{n=0}^\infty$ is a
determinate Stieltjes moment sequence, then
\eqref{ajajZenon} holds, the Stieltjes moment
sequence $\{\hsf_{\phi^n}(x)\}_{n=0}^\infty$ is
determinate and its unique representing measure
$\widetilde\vartheta_x$ is given by \eqref{war1}
and \eqref{war2}.
   \end{enumerate}
   \end{lem}
   \begin{proof}
It follows from Lemma \ref{enplus} that
   \begin{align}\label{hfn+1x}
\hsf_{\phi^{n+1}}(x) = \hsf_{\phi^{n+1}}(\phi(y))
\overset{\eqref{hn+1fpre}}= \hsf_{\phi}(x) \cdot
\esf(\hsf_{\phi^{n}})(y), \quad y\in \phii{x}, \,
n \in \zbb_+.
   \end{align}
Using \eqref{condex}, we see that for every function
$f\colon X \to \rbop$,
   \begin{align*}
(\esf(f))(z) = \sum_{y \in \phii{x}}
\frac{\mu(y)}{\mu(\phi^{-1}(\{x\}))} f(y), \quad z \in
\phi^{-1}(\{x\}).
   \end{align*}
This and \eqref{hfn+1x} yield
   \begin{align*}
\hsf_{\phi^{n+1}}(x) =
\frac{\hsf_{\phi}(x)}{\mu(\phi^{-1}(\{x\}))}
\cdot \sum_{y \in \phii{x}} \mu(y) \int_0^\infty
t^n \vartheta_y(\D t) = \int_0^\infty t^n
\nu_x(\D t), \quad n \in \zbb_+,
   \end{align*}
where $\nu_x$ is the Borel measure on $\rbb_+$ given
by
   \begin{align*}
\nu_x = \frac{\hsf_{\phi}(x)}{\mu(\phi^{-1}(\{x\}))}
\sum_{y \in \phii{x}} \mu(y) \cdot \vartheta_y.
   \end{align*}
Hence, $\{\hsf_{\phi^{n+1}}(x)\}_{n=0}^\infty$ is
a Stieltjes moment sequence with the representing
measure $\nu_x$. Noticing that
$\hsf_{\phi^0}(x)=1$ and
   \begin{align*}
\int_\sigma \frac{1}{t} \nu_x(\D t) = \sum_{y \in
\phii{x}} \frac{\mu(y)}{\mu(x)} \int_\sigma \frac
{1}{t} \vartheta_y(\D t), \quad \sigma \in
\borel{\rbb_+},
   \end{align*}
we can apply \cite[Lemma 2.4.1]{b-j-j-sA} with
$\vartheta=1$ to obtain (i) and (ii). This
completes the proof.
   \end{proof}
   \begin{rem}
Regarding Lemma \ref{dcc}, it is worth pointing out
that if $\esf(\hsf_{\phi^n}) = \hsf_{\phi^n}$ a.e.\
$[\mu]$ for every $n\in \zbb_+$, then assertions (i)
and (ii) are still valid if \eqref{ajajZenon} is
replaced by
   \begin{align*}
\hsf_{\phi}(x) \cdot \int_0^\infty \frac{1}{t}
\vartheta_y(\D t) \Le 1 \text{ for some $y \in
\phii{x}$,}
   \end{align*}
and \eqref{war1} and \eqref{war2} are replaced by
(with the above $y$)
   \begin{align*}
\widetilde\vartheta_x(\sigma) = \hsf_{\phi}(x)
\int_{\sigma} \frac {1}{t} \vartheta_y(\D t) +
\varepsilon_x \cdot \delta_0(\sigma) \text{ with }
\varepsilon_x = 1 - \hsf_{\phi}(x) \int_0^\infty \frac
{1}{t} \vartheta_y(\D t).
   \end{align*}
Indeed, in view of \eqref{hfn+1x}, the Stieltjes
moment sequence $\{\hsf_{ \phi^{n+1}}
(x)\}_{n=0}^\infty$ is represented by the measure
$\hsf_{\phi}(x) \cdot \vartheta_y$ and thus we can
apply \cite[Lemma 2.4.1]{b-j-j-sA}. Note that under
the circumstances of (ii) the measure $\vartheta_y$
does not depend on $y \in \phii{x}$.
   \end{rem}
It is worth mentioning that Lemma \ref{dcc} does
not exclude the possibility that the
transformation $\phi$ has an essential fixed
point $x$, i.e., $x \in \phii{x}$ (under the
assumption $\mu(\phi^{-1}(\{x\}))>0$, this is
equivalent to $\phi(x)=x$). We will show that if
this is the case (cf.\ Example \ref{busube}),
then, under the determinacy assumption, all the
representing measures $\vartheta_y$ are
concentrated on the interval $(1, \infty)$ except
for $\vartheta_x$ which is concentrated on
$[1,\infty)$.
   \begin{lem} \label{1inf}
Under the assumptions of Lemma {\em \ref{dcc}}, if
$\{\hsf_{\phi^n}(x)\}_{n=0}^\infty$ is a Stieltjes
moment sequence,
$\{\hsf_{\phi^{n+1}}(x)\}_{n=0}^\infty$ is a
determinate Stieltjes moment sequence and $x \in
\phii{x}$, then $\vartheta_x([0,1))=0$ and
$\vartheta_y([0,1])=0$ for every $y \in
\phii{x}\setminus \{x\}$.
   \end{lem}
   \begin{proof}
Since, by Lemma \ref{dcc}(ii), the sequence
$\{\hsf_{\phi^n}(x)\}_{n=0}^\infty$ is
determinate, we deduce that $\widetilde
\vartheta_x = \vartheta_x$ (with $\widetilde
\vartheta_x$ as in Lemma \ref{dcc}). By
\eqref{ajajZenon}, $\vartheta_y(\{0\})=0$ for all
$y \in \phii{x}$. In view of \eqref{war1}, we see
that for every $\sigma \in \borel{\rbb_+}$,
   \begin{align}   \label{mumia}
\int_{\sigma} \Big(1-\frac{1}{t}\Big) \vartheta_x (\D
t) = \sum_{y \in \phii{x} \setminus \{x\}}
\frac{\mu(y)}{\mu(x)} \int_{\sigma} \frac{1}{t}
\vartheta_y (\D t) + \varepsilon_x \cdot
\delta_0(\sigma),
   \end{align}
with the convention that $\sum_{y\in \varnothing} v_y=
0$. Since the right-hand side of the equality in
\eqref{mumia} is nonnegative, we conclude that the
measure $\vartheta_x$ is concentrated on $[1,\infty)$.
Hence, $\varepsilon_x=0$ and each measure
$\vartheta_y$, $y \in \phii{x}$, is concentrated on
$[1,\infty)$. Substituting $\sigma=\{1\}$ into
\eqref{mumia} completes the proof.
   \end{proof}
The local consistency technique enables us to prove
the subnormality of injective composition operators
$C_{\phi}$ under certain determinacy assumption.
Theorem \ref{detimsub1} below can be regarded as a
counterpart of \cite[Theorem 5.1.3]{b-j-j-sA}.
   \begin{thm} \label{detimsub1}
Let $(X,\ascr,\mu)$ be a discrete measure space
and $\phi$ be a nonsingular transformation of $X$
such that $\{\hsf_{\phi^n}(x)\}_{n=0}^\infty$ is
a Stieltjes moment sequence and
$\{\hsf_{\phi^{n+1}}(x)\}_{n=0}^\infty$ is a
determinate Stieltjes moment sequence for
$\mu$-a.e.\ $x\in X$. Then $C_{\phi}$ is
subnormal if and only if $\hsf_{\phi} > 0$ a.e.\
$[\mu]$. In particular, $C_{\phi}$ is subnormal
if $\mu(x) > 0$ for every $x\in X$.
   \end{thm}
   \begin{proof}
Suppose $\hsf_{\phi} > 0$ a.e.\ $[\mu]$. Set
$X_{\bullet} = \{x \in X \colon \mu(x)>0\}$. We
infer from \eqref{hfin} that $X_{\bullet}=\{x \in
X\colon \mu(\phi^{-1}(\{x\}))> 0\}$. By Lemma
\ref{dcc}(ii), for every $x\in X_{\bullet}$, the
Stieltjes moment sequence
$\{\hsf_{\phi^n}(x)\}_{n=0}^\infty$ is
determinate; denote its unique representing
measure by $P(x, \cdot)$. Set
$P(x,\cdot)=\delta_0$ for $x \in X \setminus
X_{\bullet}$. Since $\hsf_{\phi^0}\equiv 1$, we
see that $P\colon X\times \borel{\rbb_+} \to
[0,1]$ is a family of probability measures. By
Lemma \ref{dcc}(ii), we have
   \begin{align} \label{krak1}
P(x,\sigma) = \sum_{y \in \phii{x}}
\frac{\mu(y)}{\mu(x)} \int_{\sigma} \frac {1}{t}
P(y,\D t) + \varepsilon_x \cdot \delta_0(\sigma),
\quad \sigma\in \borel{\rbb_+}, \, x \in
X_{\bullet}.
   \end{align}
It follows from \eqref{krak1} that $P(y,
\{0\})=0$ for all $y \in \phii{x}$ and $x\in
X_{\bullet}$. Since $x\in \phi^{-1}(\{\phi(x)\})$
for every $x\in X$, we deduce that $\phi(x) \in
X_{\bullet}$ and $x\in \phii{\phi(x)}$ for every
$x\in X_{\bullet}$. Hence $P(x,\{0\})=0$ for
every $x\in X_{\bullet}$. Substituting
$\sigma=\{0\}$ into \eqref{krak1}, we deduce that
$\varepsilon_x=0$ for every $x \in X_{\bullet}$.
This means that condition (ii) of Lemma
\ref{chcc} is satisfied. By Theorem
\ref{glownedis}, $C_{\phi}$ is subnormal. The
reverse implication follows from
\cite[Proposition 6.2 and Corollary
6.3]{b-j-j-sC}.

Suppose now that $\mu(x) > 0$ for every $x\in X$. Note
that for every $x\in X$, the Stieltjes moment sequence
$\{\hsf_{\phi^n}(x)\}_{n=0}^\infty$ is determinate
(see e.g., \cite[Lemma 2.4.1]{b-j-j-sA}); denote its
representing measure by $\vartheta_x$. In view of the
previous paragraph and Lemma \ref{phix}, it suffices
to show that $\phi(X)=X$. Suppose that, contrary to
our claim, there exists $x_0 \in X\setminus \phi(X)$.
Then $\phi^{-n}(\{x_0\}) = \varnothing$ for all $n\Ge
1$, which implies that $\vartheta_{x_0}=\delta_0$.
Observe that $x_0\in \phii{\phi(x_0)}$. Applying Lemma
\ref{dcc}(ii) to $x=\phi(x_0)$ and using
\eqref{ajajZenon}, we deduce that
$\vartheta_{x_0}(\{0\})=0$, which contradicts
$\vartheta_{x_0}=\delta_0$. This completes the proof.
   \end{proof}
   \subsection{\label{asefp}A single essential fixed point}
Now we address the question of subnormality of
composition operators induced by a transformation
which has a single essential fixed point $x$,
i.e., $\phi^{-1}(\{x\})$ is a two-point set and
$\phi^{-1}(\{y\})$ is a one-point set for every
$y\neq x$. The situation seems to be simple, but
it is not. It leads to nontrivial questions in
the theory of moment problems. This enables us to
construct unbounded subnormal composition
operators $C_{\phi}$ with the sequence
$\{\hsf_{\phi^{n+1}}(0)\}_{n=0}^\infty$ being
either determinate or indeterminate according to
our needs. For them the equalities
$\esf(\hsf_{\phi^{n}}) = \hsf_{\phi^{n}}$ a.e.\
$[\mu]$, $n\in \zbb_+$, cannot hold. This is also
rare in the bounded case.
   \begin{exa} \label{busube}
Let $(X,\ascr,\mu)$ be a discrete measure space with
$X=\zbb_+$ such that $\mu(n)>0$ for every $n\in
\zbb_+$. Assume that $\mu(0)=1$. Define the
(nonsingular) transformation $\phi$ of $\zbb_+$ by
$\phi(0)=0$ and $\phi(n)=n-1$ for $n\Ge 1$. By
\eqref{hfin}, we have
   \begin{align} \label{pink}
\hsf_{\phi^n}(k) =
   \begin{cases} \displaystyle{
   \frac{\mu(n+k)}{\mu(k)}} & \text{if $k\Ge 1$,}
\\[3ex]
   \displaystyle{\sum_{j=0}^n \mu(j)} & \text{if
$k=0$,}
   \end{cases}
\quad n\in\zbb_+.
   \end{align}
Since $\{\chi_{\{x\}}\colon x \in X\} \subseteq
\dzn{C_{\phi}}$, we see that $\dzn{C_{\phi}}$ is
dense in $L^2(\mu)$.

Suppose $\{\hsf_{\phi^n}(0)\}_{n=0}^\infty$ is a
Stieltjes moment sequence with a representing measure
$\vartheta_0$, $\{\hsf_{\phi^{n+1}}(0)\}_{n=0}^\infty$
is a determinate Stieltjes moment sequence and
$\{\hsf_{\phi^n}(1)\}_{n=0}^\infty$ is a Stieltjes
moment sequence with a representing measure
$\vartheta_1$. It follows from Lemma \ref{1inf},
applied to $x=0$, that
$\vartheta_0([0,1))=\vartheta_1([0,1])=0$. We claim
that the Stieltjes moment sequence
$\{\hsf_{\phi^n}(1)\}_{n=0}^\infty$ is determinate,
   \begin{gather*}
\int_0^\infty \frac{\mu(1)}{t-1} \vartheta_1(\D t) \Le
1
   \end{gather*}
and
   \begin{gather*}
\vartheta_0 (\sigma) = \int_{\sigma}
\frac{\mu(1)}{t-1} \vartheta_1(\D t) + \varepsilon
\delta_1(\sigma), \quad \sigma \in \borel{\rbb_+},
   \end{gather*}
with
   \begin{align*}
\varepsilon = 1 - \int_0^\infty \frac{\mu(1)}{t-1}
\vartheta_1(\D t).
   \end{align*}
Indeed, by \eqref{pink}, we have
   \begin{align*}
\hsf_{\phi^n}(0) & = 1 + \mu(1) \int_0^\infty (1
+ \ldots + t^{n-1}) \vartheta_1 (\D t), \quad n
\in \nbb.
   \end{align*}
This yields
   \begin{align} \label{fi-n1}
\int_0^\infty t^n(t-1) \vartheta_0(\D t) =
\hsf_{\phi^{n+1}}(0) - \hsf_{\phi^n}(0) = \mu(1)
\int_0^\infty t^n \vartheta_1(\D t), \quad n \in
\zbb_+.
   \end{align}
Note that the measure $(t-1) \vartheta_0(\D t)$
is determinate. Indeed, since the measure $t
\vartheta_0(\D t)$, being a representing measure
of $\{\hsf_{\phi^{n+1}}(0)\}_{n=0}^\infty$, is
determinate, we infer from \eqref{mriesz} that
$\cbb[t]$ is dense in $L^2((1+t^2) t
\vartheta_0(\D t))$. Hence, if $\sigma\in
\borel{\rbb_+}$, then there exists a sequence
$\{p_n\}_{n=1}^\infty \subseteq \cbb[t]$ such
that
   \begin{align*}
\lim_{n\to \infty} \int_0^\infty |\chi_{\sigma}(t) -
p_n(t)|^2 (1+t^2) t\vartheta_0(\D t) = 0.
   \end{align*}
Therefore
   \begin{align*}
\lim_{n\to \infty} \int_0^\infty |\chi_{\sigma}(t) -
p_n(t)|^2 (1+t^2) (t-1) \vartheta_0(\D t) = 0.
   \end{align*}
This implies that $\cbb[t]$ is dense in
$L^2((1+t^2) (t-1) \vartheta_0(\D t))$. Applying
\eqref{mriesz} again completes the proof of the
determinacy of $(t-1) \vartheta_0(\D t)$ (because
$\vartheta_0([0,1))=0$). This combined with
\eqref{fi-n1} implies that
$\{\hsf_{\phi^n}(1)\}_{n=0}^\infty$ is
determinate and $\mu(1)\vartheta_1(\sigma) =
\int_{\sigma} (t-1) \vartheta_0(\D t)$ for every
$\sigma \in \borel{\rbb_+}$. Hence, for every
$\sigma \in \borel{\rbb_+}$,
   \begin{align*}
\vartheta_0(\sigma) = \vartheta_0(\sigma \cap
(1,\infty)) + \vartheta_0(\sigma\cap \{1\}) =
\int_{\sigma \cap (1,\infty)} \frac{\mu(1)}{t-1}
\vartheta_1(\D t) +
\vartheta_0(\{1\})\delta_1(\sigma),
   \end{align*}
and $\vartheta_0(\{1\}) = \varepsilon$, which
proves our claim.

The above reasoning can be reversed in a sense.
Namely, we will provide a general procedure of
constructing the measure $\mu$ that guarantees
the subnormality of $C_{\phi}$ (with $X$, $\ascr$
and $\phi$ as at the beginning of this example
and $\mu(0)=1$). Take a Borel probability measure
$\vartheta$ on $\rbb_+$ such that
   \begin{align} \label{tetamir}
\text{$\vartheta([0,1])=0$, $\alpha:=\int_0^\infty
\frac{1}{t-1} \vartheta(\D t)<\infty$, $\int_0^\infty
t^n \vartheta(\D t) < \infty$ for all $n\in \zbb_+$. }
   \end{align}
Note that $\alpha>0$. Take $\mu(1) \in (0,1/\alpha]$
and set
   \begin{align} \label{pink1}
\mu(n) = \mu(1) \int_0^\infty t^{n-1} \vartheta(\D t),
\quad n \Ge 2.
   \end{align}
Clearly, $\mu(k)>0$ for all $k\in \zbb_+$. Define the
family $P\colon \zbb_+ \times \borel{\rbb_+} \to
[0,1]$ of probability measures by
   \begin{align} \label{bosy}
P(k,\sigma) =
   \begin{cases}
   \displaystyle{\frac{\mu(1)}{\mu(k)} \int_{\sigma}
t^{k-1} \vartheta(\D t)} & \text{if } k\Ge 1,
   \\[3ex]
   \displaystyle{\int_{\sigma} \frac{\mu(1)}{t-1}
\vartheta(\D t) + \varepsilon \delta_1(\sigma)} &
\text{if } k=0,
   \end{cases}
\quad \sigma\in \borel{\rbb_+}.
   \end{align}
with $\varepsilon = 1 - \int_0^\infty
\frac{\mu(1)}{t-1} \vartheta(\D t)$. Observe that
$P$ satisfies condition (i) of Lemma \ref{chcc}.
Indeed, if $k\Ge 1$, then
   \begin{align*}
\int_{\sigma} t P(k,\D t) = \frac{\mu(1)}{\mu(k)}
\int_{\sigma} t^k \vartheta(\D t) =
\frac{\mu(k+1)}{\mu(k)} P(k+1,\sigma), \quad \sigma
\in \borel{\rbb_+},
   \end{align*}
while for $k=0$, we have
   \begin{align*}
\int_{\sigma} t P(0,\D t) &= \mu(1) \int_{\sigma}
\frac{t}{t-1} \vartheta(\D t) + \varepsilon
\delta_1(\sigma)
   \\
&= \mu(1) \vartheta(\sigma) + \mu(1) \int_{\sigma}
\frac{1}{t-1} \vartheta(\D t) + \varepsilon
\delta_1(\sigma)
   \\
&= \mu(1) P(1,\sigma) + P(0,\sigma), \quad \sigma\in
\borel{\rbb_+}.
   \end{align*}
Hence, by Theorem \ref{glownedis}, $C_{\phi}$ is
subnormal. In view of Lemma \ref{chcc} and
Theorem \ref{sms}, $P(k,\cdot)$ is a representing
measure of $\{\hsf_{\phi^n}(k)\}_{n=0}^\infty$
for every $k\in \zbb_+$. Note also that
   \begin{align} \label{ehfn2}
\text{$\esf(\hsf_{\phi^{n}}) = \hsf_{\phi^{n}}$ a.e.\
$[\mu]$ for all $n\in \zbb_+$ if and only if
$\vartheta = \delta_{1+\mu(1)}$.}
   \end{align}
(Of course, if $\vartheta = \delta_{1+\mu(1)}$, then
$\varepsilon=0$.) Indeed, it is clear that
$\esf(\hsf_{\phi^{n}}) = \hsf_{\phi^{n}}$ a.e.\
$[\mu]$ for all $n\in \zbb_+$ if and only if
$\hsf_{\phi^{n}}(0) = \hsf_{\phi^{n}}(1)$ for all
$n\in\zbb_+$ (cf.\ \eqref{condex}), or equivalently if
and only if $\sum_{j=0}^n \mu(j) = \mu(n+1)/\mu(1)$
for all $n\in \zbb_+$ (cf.\ \eqref{pink}). By
induction on $n$, the latter holds if and only if
$\mu(n+1) = \mu(1) (1+\mu(1))^n$ for all $n\in
\zbb_+$. This and \eqref{pink1} (consult also
\eqref{mriesz}) completes the proof of \eqref{ehfn2}.
We point out that the situation described in
\eqref{ehfn2} may happen only when $C_{\phi} \in
\ogr{L^2(\mu)}$, and if this is the case, then
$\|C_{\phi}\|^2 = 1+\mu(1)$ (cf.\ \eqref{norm}).

Note that if $\vartheta$ and $\mu$ are as in
\eqref{tetamir} and \eqref{pink1} with
$\vartheta(\rbb_+)=1$, $\mu(0) = 1$ and $\mu(1)
\in (0,1/\alpha]$, then $C_{\phi} \in
\ogr{L^2(\mu)}$ if and only if $\sup(\supp
\vartheta) < \infty$. Indeed, by \eqref{pink},
$C_{\phi} \in \ogr{L^2(\mu)}$ if and only if
$\beta < \infty$, where $\beta:=\sup_{k\Ge 1}
\frac{\mu(k+1)}{\mu(k)}$. Since
$\vartheta(\rbb_+)=1$, we infer from
\eqref{pink1} that $\{\mu(k+1)\}_{k=0}^\infty$ is
a Stieltjes moment sequence with a repre\-senting
measure $\mu(1) \vartheta(\D t)$. Hence, by Lemma
\ref{granica}, we see that $\beta < \infty$ if
and only if $\sup(\supp \vartheta) < \infty$.
Moreover, if $C_{\phi} \in \ogr{L^2(\mu)}$, then
by Lemma \ref{granica}, \cite[Theorem 1]{nor} and
\eqref{pink} we have
   \begin{align} \label{norm}
\|C_{\phi}\|^2 = \max\big\{1+\mu(1), \sup(\supp
\vartheta)\big\}.
   \end{align}

Now we provide explicit examples of measures
$\vartheta$ leading to unbounded subnormal
$C_{\phi}$'s for which the sequence
$\{\hsf_{\phi^{n+1}}(0)\}_{n=0}^\infty$ is either
determinate or indeterminate according to our
needs. We begin with the determinate case. Set
   \begin{align*}
\text{$\vartheta = c^{-1} \sum_{j=2}^\infty j^{-1}
\E^{-j^2} \delta_j$ and $\gamma_n = \int_0^\infty t^n
\vartheta(\D t)$ for $n \in \zbb_+$,}
   \end{align*}
where $c= \sum_{j=2}^\infty j^{-1} \E^{-j^2}$. It
is easily seen that $\vartheta$ is a probability
measure which satisfies \eqref{tetamir}. Let
$\alpha$, $\mu$ and $P$ be as in \eqref{tetamir},
\eqref{pink1} and \eqref{bosy} with $\mu(0)=1$
and $\mu(1) \in (0,1/\alpha]$. Note that there
exists a positive real number $b$ such that
$\gamma_{n} \Le b n^n$ for all $n\Ge 1$ (see
\cite[Example 4.2.2]{j-j-s0} and \cite[Example
7.1]{2xSt2}). This implies that there exists a
positive real number $b^\prime$ such that
$\hsf_{\phi^n}(0)=\int_0^\infty t^{n} P(0,\D t)
\Le b^\prime n^n$ for all $n\Ge 1$. By the
Carleman criterion (see e.g., \cite[Corollary
4.5]{sim}), the Stieltjes moment sequences
$\{\hsf_{\phi^n}(0)\}_{n=0}^\infty$ and
$\{\hsf_{\phi^{n+1}}(0)\}_{n=0}^\infty$ are
determinate.

The indeterminate case can be done as follows. Let
$\vartheta$ be an indeterminate probability measure
such that\footnote{\;Consider e.g., the measure
$\vartheta$ given by $\vartheta(\sigma) = \widetilde
\vartheta (\frac{1}2 \cdot \sigma)$ for $\sigma \in
\borel{\rbb_+}$, where $\widetilde\vartheta$ is the
$q$-orthogonality probability measure for the
Al-Salam-Carlitz polynomials ($0 < q < 1$), which is
indeterminate and supported in
$\{q^{-n}\}_{n=0}^\infty$ (cf.\ \cite{chi}).}
$\vartheta([0,2))=0$. Clearly $\vartheta$ satisfies
\eqref{tetamir}. Set $\mu(1)=\frac{1}{\alpha}$. Then
$\varepsilon=0$ and for every Borel function $f\colon
\rbb_+ \to \rbop$,
   \begin{align*}
\mu(1) \int_0^\infty f(t) (1+t^2) \vartheta(\D t) \Le
\int_0^\infty f(t) (1+t^2)t P(0,\D t).
   \end{align*}
By \eqref{mriesz} and the indeterminacy of
$\vartheta$, this implies that the measure $tP(0,\D
t)$ is indeterminate, and thus the corresponding
sequence of moments
$\{\hsf_{\phi^{n+1}}(0)\}_{n=0}^\infty$ is
indeterminate.
   \end{exa}
   \subsection{\label{fullk}Finite constant valences
on generations} In this section we investigate
composition operators built on a directed tree with
finite constant valences on generations. Let
$\tcal=(V,E)$ be a rootless and leafless directed
tree, where $V$ and $E$ stand for the sets of vertices
and edges of $\tcal$, respectively. Denote by $\pa v$
the parent of $v \in V$. Assume that $V$ is countably
infinite. Let $\mu$ be a $\sigma$-finite measure on
$2^V$ such that $\mu(x) > 0$ for every $x\in V$; call
$\mu(x)$ the {\em mass} of the vertex $x$. Set
$\phi=\paa$. By \cite[Proposition 2.1.12]{j-j-s},
there exists a partition $\{G_m\}_{m\in \zbb}$ of $V$
such that $G_{m+1} = \bigsqcup_{x\in G_m}
\phi^{-1}(\{x\})$ for every $m\in \zbb$; call $G_m$
the $m$th {\em generation} of $\tcal$. Assume that
$\{\kappa_m\}_{m\in \zbb}$ is a two-sided sequence of
positive integers and $\{\alpha_m\}_{m\in \zbb}$ is a
two-sided sequence of positive real numbers such that
   \begin{align} \label{gener}
&\text{$\phi^{-1}(\{x\})$ has $\kappa_m$ elements for
all $x\in G_m$ and $m\in \zbb$,}
   \\  \label{meser}
&\text{$\mu(x)=\alpha_{m}$ for all $x\in G_m$ and
$m\in \zbb$.}
   \end{align}
We call $\{\kappa_m\}_{m\in \zbb}$ the {\em valence
sequence} of $\tcal$. Define $\{\hat \kappa_m\}_{m\in
\zbb} \subseteq (0,\infty)$ by
   \begin{align}  \label{kmzd}
\hat\kappa_{m} =
   \begin{cases}
\prod_{j=0}^{m-1} \kappa_j & \text{if } m\Ge 1,
   \\
1 & \text{if } m=0,
   \\
\big(\prod_{j=1}^{-m} \kappa_{-j}\big)^{-1} & \text{if
} m\Le -1.
   \end{cases}
   \end{align}
It is a matter of routine to show that
   \begin{align} \label{rutyna}
\kappa_{m} \hat \kappa_{m} = \hat \kappa_{m+1}, \quad
m\in \zbb.
   \end{align}
   \begin{lem} \label{kappa1}
Under the assumptions above we have
   \begin{enumerate}
   \item[(i)] $\hsf_{\phi^n}(x) =
\displaystyle{\frac{\alpha_{m+n}}{\alpha_{m}}}
\prod_{j=0}^{n-1} \kappa_{m+j}$ for all $x\in G_{m}$,
$m\in \zbb$ and $n\Ge 1$,
   \item[(ii)] $\esf(\hsf_{\phi^n}) = \hsf_{\phi^n}$
for all $n\in \zbb_+$,
   \item[(iii)] $\dzn{C_{\phi}}$ is dense in
$L^2(\mu)$.
   \end{enumerate}
   \end{lem}
   \begin{proof}
   (i) We use induction on $n$. If $n=1$, then by
\eqref{gener} and \eqref{meser}, we have
   \begin{align*}
\hsf_{\phi}(x) =
\frac{\mu(\phi^{-1}(\{x\}))}{\alpha_{m}} =
\frac{\alpha_{m+1} \kappa_{m}}{\alpha_{m}}, \quad
x\in G_{m}, \, m\in \zbb.
   \end{align*}
Now, assume that the induction hypothesis holds for a
fixed $n\Ge 1$. Then
   \allowdisplaybreaks
   \begin{align*}
\hsf_{\phi^{n+1}}(x) &=
\frac{\mu\big(\phi^{-n}(\phi^{-1}(\{x\}))\big)}{\alpha_{m}}
\overset{\eqref{meser}}= \sum_{y\in
\phi^{-1}(\{x\})} \frac{\alpha_{m+1}
\mu(\phi^{-n}(\{y\}))}{\alpha_{m} \mu(y)}
   \\
&= \sum_{y\in \phi^{-1}(\{x\})}
\frac{\alpha_{m+1}}{\alpha_{m}} \hsf_{\phi^n}(y)
= \sum_{y\in \phi^{-1}(\{x\})}
\frac{\alpha_{m+1}}{\alpha_{m}}
\frac{\alpha_{m+n+1}}{\alpha_{m+1}}
\prod_{j=0}^{n-1} \kappa_{m+j+1}
   \\
&\hspace{-.8ex}\overset{\eqref{gener}}=
\frac{\alpha_{m+n+1}}{\alpha_{m}} \kappa_m
\prod_{j=1}^{n} \kappa_{m+j} =
\frac{\alpha_{m+n+1}}{\alpha_{m}} \prod_{j=0}^{n}
\kappa_{m+j}, \quad x\in G_{m}, \, m\in \zbb.
   \end{align*}
This completes the proof of (i).

(ii) By (i), the function $\hsf_{\phi^n}$ is constant
on $\phi^{-1}(\{x\})$ for all $x\in V$ and $n\Ge 1$.
Since $\hsf_{\phi^0}\equiv 1$, we get (ii).

(iii) By (i), $\{\chi_{\{x\}}\colon x \in V\}
\subseteq \dzn{C_{\phi}}$, which yields (iii).
   \end{proof}
A two-sided sequence $\{a_n\}_{n\in \zbb}
\subseteq \rbb_+$ is called a {\em two-sided
Stieltjes moment sequence} if there exists a
Borel measure $\nu$ on $(0,\infty)$ such that
$a_{n}=\int_{(0,\infty)} s^n \nu(\D s)$ for every
$n \in \zbb$; the measure $\nu$ is called a {\em
representing measure} of $\{a_n\}_{n\in \zbb}$.
By \cite[page 202]{b-c-r}, we have
   \begin{align} \label{char2sid}
   \begin{minipage}{29em}
$\{a_n\}_{n\in \zbb} \subseteq \rbb_+$ is a two-sided
Stieltjes moment sequence if and only if
$\{a_{n-k}\}_{n=0}^\infty$ is a Stieltjes moment
sequence for every $k \in \zbb_+$.
   \end{minipage}
   \end{align}
Using our main criterion, we provide necessary and
sufficient conditions for subnormality of composition
operators considered above. To the best of our
knowledge, this class of operators is the third one,
besides unilateral and bilateral injective weighted
shifts (cf.\ \cite{StSz2,b-j-j-sB}), for which
condition (ii) of Theorem \ref{ehphin}, known as
Lambert's condition (see \cite{lam}), characterizes
the subnormality in the unbounded~ case.
   \begin{thm} \label{ehphin}
Under the assumptions of the first paragraph of this
section, $\dzn{C_{\phi}}$ is dense in $L^2(\mu)$ and
the following four conditions are equivalent{\em :}
   \begin{enumerate}
   \item[(i)] $C_{\phi}$ is subnormal,
   \item[(ii)] $\{\|C_{\phi}^n f\|^2\}_{n=0}^\infty$
is a Stieltjes moment sequence for every $f\in
\dzn{C_{\phi}}$,
   \item[(iii)] $\{\hsf_{\phi^n}(x)\}_{n=0}^\infty$ is
a Stieltjes moment sequence for every $x\in V$,
   \item[(iv)] $\{\alpha_m \hat\kappa_{m}\}_{m\in
\zbb}$ is a two-sided Stieltjes moment sequence
$($cf.\ \eqref{kmzd}$)$.
   \end{enumerate}
   \end{thm}
   \begin{proof}
By Lemma \ref{kappa1}(iii),
$\overline{\dzn{C_{\phi}}}=L^2(\mu)$. The implications
(i)$\Rightarrow$(ii) and (ii)$\Rightarrow$(iii) follow
from \cite[Proposition 3.2.1]{b-j-j-sA} and
\cite[Theorem 10.4]{b-j-j-sC} respectively.

(iii)$\Rightarrow$(iv) Set $\gamma_m = \alpha_m
\hat\kappa_{m}$ for $m\in \zbb$. An induction argument
based on \eqref{rutyna} shows that $\hat \kappa_{n-m}
= \hat \kappa_{-m} \prod_{j=0}^{n-1} \kappa_{j-m}$ for
all $m \in \zbb_+$ and $n \in \nbb$. Applying Lemma
\ref{kappa1}(i) implies that $\gamma_{n-m} =
\alpha_{-m} \hat \kappa_{-m} \hsf_{\phi^n}(x)$ for
every $n \in \zbb_+$ and for all $x\in G_{-m}$ and $m
\in \zbb_+$. This, together with \eqref{char2sid},
yields (iv).

(iv)$\Rightarrow$(i) Let $\nu$ be a representing
measure of the two-sided Stieltjes moment sequence
$\{\alpha_0^{-1}\alpha_m \hat\kappa_{m}\}_{m\in
\zbb}$. Define the mapping $P\colon V\times
\borel{\rbb_+} \to [0,1]$ by
    \begin{align} \label{kuku2}
P(x,\sigma) = \frac{\alpha_0}{\alpha_m \hat
\kappa_{m}} \int_{\sigma} t^m \D \nu(t), \quad x \in
G_m, \, \sigma \in \borel{\rbb_+}, \, m \in \zbb.
    \end{align}
Since $\nu$ is a representing measure of
$\{\alpha_0^{-1}\alpha_m \hat\kappa_{m}\}_{m\in
\zbb}$, we see that $P$ is a family of probability
measures. Applying \eqref{rutyna}, \eqref{kuku2} and
Lemma \ref{kappa1}(i), we deduce that
   \allowdisplaybreaks
    \begin{align*}
\frac{\int_\sigma t P(\phi(x),\D
t)}{\hsf_{\phi}(\phi(x))} & =
\frac{\alpha_0}{\alpha_{m} \kappa_{m-1} \hat
\kappa_{m-1}} \int_{\sigma} t^{m} \D \nu(t)
   \\
& = P(x,\sigma), \quad \sigma\in\borel{\rbb_+}, \, x
\in G_m, \, m\in \zbb.
    \end{align*}
This means that the family $P$ satisfies
\eqref{consistB}. Since $0 < \hsf_{\phi} < \infty$, we
infer from Theorem \ref{glowne} that $C_{\phi}$ is
subnormal. This completes the proof.
   \end{proof}
   \begin{rem}
In view of Theorem \ref{ehphin}, $C_{\phi}$ is
subnormal if and only if there exists a two-sided
Stieltjes moment sequence $\{\gamma_m\}_{m\in \zbb}$
such that $\alpha_m = \hat \kappa_m^{-1}\gamma_m$ for
all $m\in \zbb$. Hence, if $\tcal$ is a full
$\kappa$-ary directed tree, i.e., $\kappa_m = \kappa$
for all $m \in \zbb$, then $\hat \kappa_m = \kappa^m$
for all $m\in \zbb$, and consequently $C_{\phi}$ is
subnormal if and only if $\{\alpha_m\}_{m\in \zbb}$ is
a two-sided Stieltjes moment sequence. This
characterization of subnormality of $C_{\phi}$ does
not depend on $\kappa$. For $\kappa=1$, it covers the
case of injective bilateral weighted shifts (cf.\
\cite{hal2} and \cite{StSz2}). Therefore, a question
arises as to whether the composition operator
$C_{\phi}$ built on a directed tree with the valence
sequence $\{\kappa_m\}_{m\in \zbb}$ is unitarily
equivalent to an orthogonal sum of injective bilateral
weighted shifts. The answer is in the negative if
$\kappa_m > 1$ for some $m\in \zbb$. This is because
the adjoint of an injective bilateral weighted shift
is injective and $\jd{C_{\phi}^*} \neq \{0\}$. To see
that $\jd{C_{\phi}^*} \neq \{0\}$, observe that the
linear span of the set $\{\chi_{\{x\}}\colon x\in V\}$
is a core for $C_{\phi}$ (use \cite[(3.5)]{b-j-j-sC}
and $\hsf_{\phi} < \infty$). Hence $f \in L^2(\mu)$
belongs to $\jd{C_{\phi}^*}$ if and only if
$\is{f}{\chi_{\phi^{-1}(\{x\})}} = 0$ for every $x\in
V$, which implies that for every $x \in \varGamma :=
\bigcup_{m :\kappa_m> 1} G_m$ there exists normalized
$h_{x} \in \chi_{\phi^{-1}(\{x\})} L^2(\mu)$
orthogonal to $\chi_{\phi^{-1}(\{x\})}$ and vanishing
on $V\setminus \phi^{-1}(\{x\})$. Then $\{h_{x}\colon
x \in \varGamma\}$ is an orthonormal system in
$\jd{C_{\phi}^*}$ and thus $\jd{C_{\phi}^*} \neq
\{0\}$. Clearly, if $\varGamma$ is infinite, then
$\dim \jd{C_{\phi}^*} = \aleph_0$.

Now we discuss the case of unilateral weighted shifts.
By Lemma \ref{kappa1}(i), for $\{\kappa_m\}_{m\in
\zbb} \subseteq \nbb$ there exists $\{\alpha_m\}_{m\in
\zbb} \subseteq (0,\infty)$ such that $C_{\phi}$ is an
isometry. Clearly
   \begin{align*}
\obn{C_{\phi}}:=\bigcap_{n=1}^\infty \ob{C_{\phi}^n} =
\bigcap_{n=1}^\infty \bigcap_{x\in V} \Big\{f\in
L^2(\mu)\colon f \text{ is constant on }
\phi^{-n}(\{x\})\Big\}.
   \end{align*}
Hence, $f\in L^2(\mu)$ belongs to $\obn{C_{\phi}}$ if
and only if $f$ is constant on $G_m$ for every $m \in
\zbb$. Thus, by \eqref{meser}, $\obn{C_{\phi}} =
\{0\}$ if and only if $G_m$ is infinite for every
$m\in \zbb$ (by \cite[(6.1.3)]{j-j-s}, the latter is
equivalent to $\limsup_{m\to -\infty} \kappa_m \Ge
2$). If this is the case, then by Wold's decomposition
theorem (cf.\ \cite[Theorem 23.7]{Con2}) $C_{\phi}$ is
unitarily equivalent to an orthogonal sum of
unilateral isometric shifts of multiplicity $1$.
Otherwise, the unitary part of $C_{\phi}$ is
nontrivial and so, by Wold's decomposition, $C_{\phi}$
is not unitarily equivalent to an orthogonal sum of
unilateral~ weighted shifts.

Regarding Theorem \ref{ehphin}, if the masses of
vertices of the same generation are not assumed to be
constant, then there is no hope to get a
characterization of subnormality of $C_{\phi}$. Only a
sufficient condition written in terms of consistent
systems of probability measures can be provided (cf.\
\cite{b-j-j-sA,b-j-j-sB}). The implications
(ii)$\Rightarrow$(i) and (iii)$\Rightarrow$(i) are not
true in general (cf.\ \cite{j-j-s0,b-j-j-sC}).
   \end{rem}
   Now we characterize the boundedness and left
semi-Fredholmness of subnormal composition operators
considered in Theorem \ref{ehphin}. For the theory of
Fredholmness of general and particular operators, we
refer the reader to \cite{gol} and \cite{j-j-s}
respectively.
   \begin{pro}\label{ogr-a}
Under the assumptions of the first paragraph of this
section, if $C_{\phi}$ is subnormal and $\nu$ is a
representing measure of $\{\alpha_m
\hat\kappa_{m}\}_{m\in \zbb}$ $($cf.\ \eqref{kmzd}$)$,
then $\supp \nu\neq \varnothing$ and the following
assertions hold{\em :}
   \begin{enumerate}
   \item[(i)] $C_{\phi}$ is in $\ogr{L^2(\mu)}$ if and
only if $\sup(\supp \nu) < \infty$; moreover, if this
is the case, then $\|C_{\phi}\|^2 = \sup(\supp \nu)$,
   \item[(ii)] if  $c$ is a positive
real number, then $\|C_{\phi}f\| \Ge c \|f\|$ for
every $f\in \dz{C_{\phi}}$ if and only if $\inf(\supp
\nu) \Ge c^2$,
   \item[(iii)]  $C_{\phi}$ is left semi-Fredholm if and
only if $\inf(\supp\nu) > 0$.
   \end{enumerate}
   \end{pro}
   \begin{proof}
Set $\gamma_m = \alpha_m \hat\kappa_{m}$ for $m\in
\zbb$. Since $\gamma_0 > 0$, we see that $\supp \nu
\neq \varnothing$.

   (i) Applying Lemma \ref{granica} to the sequences
$\{\gamma_{m-k}\}_{m=0}^\infty$, $k\in \zbb_+$, we
deduce that the two-sided sequence
$\big\{\frac{\gamma_{m+1}}{\gamma_m}\big\}_{m\in
\zbb}$ is monotonically increasing and
   \begin{align*}
\sup_{x\in V} \hsf_{\phi}(x)
\overset{(\dag)}=\sup_{m\in\zbb}
\frac{\gamma_{m+1}}{\gamma_m} = \sup_{m\in\zbb_+}
\frac{\gamma_{m+1}}{\gamma_m} = \sup(\supp \nu),
   \end{align*}
where $(\dag)$ follows from Lemma \ref{kappa1}(i) and
\eqref{rutyna}. This and \cite[Theorem 1]{nor} yields
(i).

   (ii) We first note that $\{\gamma_{-m}\}_{m\in
\zbb}$ is a two-sided Stieltjes moment sequence with
the representing measure $\nu \circ \tau^{-1}$, where
$\tau$ is the transformation of $\rbb_+$ given by
$\tau(t)=\frac{1}{t}$ for $t\in (0,\infty)$ and
$\tau(0)=0$. Using the fact that the two-sided
sequence $\big\{\frac{\gamma_{m+1}} {\gamma_{m}}
\big\}_{m\in \zbb}$ is monotonically increasing (see
the previous paragraph) and applying Lemma
\ref{granica} to the Stieltjes moment sequence
$\{\gamma_{-m}\}_{m=0}^\infty$, we get
   \allowdisplaybreaks
   \begin{align} \notag
\inf \Big\{\frac{\gamma_{m+1}}{\gamma_{m}} \colon m\in
\zbb\Big\} &= \inf
\Big\{\frac{\gamma_{-m}}{\gamma_{-m-1}} \colon m\in
\zbb_+\Big\} = \frac{1}{\sup
\Big\{\displaystyle{\frac{\gamma_{-(m+1)}}{\gamma_{-m}}}
\colon m\in \zbb_+\Big\}}
   \\        \label{xxx}
& = \frac{1}{\sup(\supp \nu \circ \tau^{-1})} =
\inf(\supp \nu).
   \end{align}
By Proposition \ref{ogrodd}, Lemma \ref{kappa1}(i) and
\eqref{rutyna}, $\|C_{\phi}f\| \Ge c \|f\|$ for every
$f\in \dz{C_{\phi}}$ if and only if $\inf
\big\{\frac{\gamma_{m+1}}{\gamma_{m}} \colon m\in
\zbb\big\} \Ge c^2$. This and \eqref{xxx} imply (ii).

(iii) Since $C_{\phi}$ is injective closed and densely
defined, we infer from the closed graph theorem that
$C_{\phi}$ is left semi-Fredholm if and only if it is
bounded from below. This and (ii) complete the proof.
   \end{proof}
Note that under the assumptions of Proposition
\ref{ogr-a}, it may happen that $C_{\phi}$ is bounded
from below and the measure $\nu$ is indeterminate. A
sample of such measure appears in the last paragraph
of Example \ref{busube}. In fact, any $N$-extremal
measure on $\rbb_+$ different from the Krein one meets
our requirements (see \cite[Section 2.1]{j-j-s0} for
an overview of the theory of indeterminate moment
problems).
   \subsection{\label{wsrdc}Weighted shifts on rootless
directed trees} Using Theorem \ref{glownedis}, we
will show that Theorem 5.1.1 of \cite{b-j-j-sA}
remains true for weighted shifts on rootless and
leafless directed trees with nonzero weights
without assuming the density of
$C^\infty$-vectors in the underlying
$\ell^2$-space. Recall that by a {\em weighted
shift} on a rootless directed tree $\tcal=(V,E)$
with weights $\lambdab=\{\lambda_v\}_{v \in V}
\subseteq \cbb$ we mean the operator $\slam$ in
$\ell^2(V)$ given by
   \begin{align*}
   \begin{aligned}
\dz {\slam} & = \{f \in \ell^2(V) \colon
\varLambda_\tcal f \in \ell^2(V)\},
   \\
\slam f & = \varLambda_\tcal f, \quad f \in \dz
{\slam},
   \end{aligned}
   \end{align*}
where $\varLambda_\tcal$ is the mapping defined on
functions $f\colon V \to \cbb$ via
   \begin{align*}
(\varLambda_\tcal f) (v) = \lambda_v \cdot f\big(\pa
v\big), \quad v \in V,
   \end{align*}
and $\pa v$ stands for the parent of $v$. We refer the
reader to \cite{j-j-s} for the foundations of the
theory of weighted shifts on directed trees.
   \begin{thm} \label{wsi}
Let $\slam$ be a densely defined weighted shift on a
rootless and leafless directed tree $\tcal=(V,E)$ with
nonzero weights $\lambdab=\{\lambda_v\}_{v \in V}$.
Suppose there exists a system $\{\mu_v\}_{v \in V}$ of
Borel probability measures on $\rbb_+$ such that
   \begin{align} \label{consist6}
\mu_u(\sigma) = \sum_{v \in \dzi{u}}
|\lambda_v|^2 \int_\sigma \frac{1}{t} \mu_v(\D
t), \quad \sigma \in \borel{\rbb_+}, \, u \in V,
   \end{align}
where $\dzi{u}$ denotes the set of all children of
$u$. Then $\slam$ is subnormal.
   \end{thm}
   \begin{proof}
In view of \cite[Theorem 3.2.1]{j-j-s}, there is no
loss of generality in assuming that all the weights of
$\slam$ are positive. It follows from
\cite[Proposition 3.1.10]{j-j-s} that $V$ is at most
countable. Since $\tcal$ is rootless, we infer from
\cite[Proposition 2.1.6]{j-j-s} that $V$ is countably
infinite. Let $\ascr=2^V$ and $\phi(u) = \pa u$ for $u
\in V$. Since $\tcal$ is rootless and leafless, we see
that $\phi$ is a well-defined surjection. As the
weights of $\slam$ are positive, we deduce from the
proof of \cite[Lemma 4.3.1]{j-j-s0} that there exists
a $\sigma$-finite measure $\mu$ on $\ascr$ which
satisfies the following three conditions:
   \begin{align} \label{(i)}
&\text{$\mu(u)>0$ for all $u \in V$,}
   \\    \label{(ii)}
&\text{$\mu(v) = \lambda_v^2 \, \mu(u)$ for all $v \in
\dzi{u}$ and $u \in V$,}
   \\    \label{(iii)}
&\text{$\slam$ is unitarily equivalent to the
composition operator $C_{\phi}$ in
$L^2(V,\ascr,\mu)$.}
   \end{align}
It follows from \eqref{(iii)} that $C_{\phi}$ is
densely defined, and thus $\hsf_{\phi} < \infty$
a.e.\ $[\mu]$, or equivalently
$\mu(\phi^{-1}{(\{u\})}) < \infty$ for every
$u\in V$ (cf.\ \eqref{hfin}). Since $\tcal$ is
rootless, we infer from \eqref{consist6} that
$\mu_u(\{0\})=0$ for every $u \in V$. Using
\eqref{(i)} and \eqref{(ii)}, we deduce from
\eqref{consist6} that
   \begin{align*}
\mu_u(\sigma) = \sum_{v \in \phi^{-1}(\{u\})}
\frac{\mu(v)}{\mu(u)} \cdot \int_\sigma
\frac{1}{t} \mu_v(\D t), \quad \sigma \in
\borel{\rbb_+}, \, u \in V,
   \end{align*}
which means that the family $P\colon V \times
\borel{\rbb_+} \to [0,1]$ of probability measures
defined by $P(u,\sigma) = \mu_u(\sigma)$ for $u \in V$
and $\sigma \in \borel{\rbb_+}$ satisfies condition
(ii) of Lemma \ref{chcc}. Hence, by applying
\eqref{(i)}, \eqref{(iii)} and Theorem
\ref{glownedis}, we complete the proof.
   \end{proof}
Arguing as in the proof of Theorem \ref{wsi}, one
can deduce from Lemma \ref{dcc} and Theorem
\ref{detimsub1} that \cite[Lemma 4.1.3]{b-j-j-sA}
and \cite[Theorem 5.1.3]{b-j-j-sA} remain true
for weighted shifts on rootless and leafless
directed trees with nonzero weights without
assuming the density of $C^\infty$-vectors in the
underlying $\ell^2$-space.

In the proof of Theorem \ref{wsi} we have used
the fact that a weighted shift on a rootless and
leafless directed tree with nonzero weights is
unitarily equivalent to a composition operator in
an $L^2$-space. Weighted shifts on directed trees
are particular instances of {\em weighted}
composition operators in $L^2$-spaces. Therefore,
one can ask a question whether weighted
composition operators in $L^2$-spaces are
unitarily equivalent to composition operators in
$L^2$-spaces. The answer is in the negative
regardless of whether the underlying measure
space is discrete or not. This can be deduced
from Proposition \ref{wcoc} which in turn can be
inferred from Proposition~ \ref{wco}.
   \begin{pro}\label{wcoc}
Let $(Y,\bscr,\nu)$ be a $\sigma$-finite measure
space, $w\colon Y\to\rbb$ be a $\bscr$-measurable
function and $\psi$ be the identity transformation of
$Y$. If the weighted composition operator $T$ in
$L^2(\nu)$ given by
   \begin{align*}
\dz{T}&=\big\{f\in L^2(\nu)\colon w\cdot (f\circ
\psi)\in L^2(\nu)\big\},
   \\
T f&=w\cdot (f\circ \psi),\quad f\in \dz{T},
   \end{align*}
is unitarily equivalent to a composition operator in
an $L^2$-space over a $\sigma$-finite measure space,
then $|w|=1$ a.e.\ $[\nu]$.
   \end{pro}
Note that Proposition \ref{wcoc} is no longer valid if
we allow $w$ to be complex-valued because normal
operators are unitarily equivalent to the
multiplication operators (cf.\ \cite[Theorem
7.33]{Weid}; see also \cite[Theorem VIII.4]{R-S}) and
there are normal composition operators in $L^2$-spaces
which are not unitary (see e.g., \cite[Example
4.2]{sin-kum}).

\numberwithin{equation}{section}
\numberwithin{thm}{section}
   \appendix
   \section{\label{AppA}Composition operators
induced by roots of the identity} In this
appendix we will show that a subnormal
composition operator induced by an $n$th root of
$\mathrm{id}_X$ must be bounded and unitary. The
proof depends heavily on the fact that all powers
of a composition operator induced by an $n$th
root of $\mathrm{id}_X$ are densely defined. We
begin by showing that the closures of ({\em a
priori} unbounded) subnormal $n$th roots of $I$
are unitary. The case of bounded operators can be
easily derived from Putnam's inequality (cf.\
\cite[Theorem 1]{Put}). Below we present a
considerably more elementary proof.
   \begin{lem} \label{nilsub}
If $S$ is a subnormal operator in a complex Hilbert
space $\hh$ such that $S^n$ is densely defined and
$S^n \subseteq I$ for some integer $n\Ge 2$, then
$\bar{S}$ is unitary.
   \end{lem}
   \begin{proof}
Clearly, $S$ is closable and the closure
$\bar{S}$ of $S$ is subnormal. By
\cite[Proposition 5.3]{sto-b}, $\bar{S}^n$ is
closed. Since $S^n$ is densely defined, we deduce
that $\bar{S}^n=I$. Hence, by the closed graph
theorem, $\bar{S} \in \ogr{\hh}$. Let $N\in
\ogr{\kk}$ be a minimal normal extension of
$\bar{S}$ acting in a complex Hilbert space
$\kk$. By minimality of $N$, $N^n=I_{\kk}$. This
implies that $|N|^{2n}=I_{\kk}$, and so
$|N|=I_{\kk}$. Therefore, $N$ is unitary and
consequently $\bar{S}$ is an isometry which is
onto (because $\bar{S}^n=I$).
   \end{proof}
Lemma \ref{nilsub} is no longer true if we do not
assume $S^n$ to be densely defined. Indeed, for
every integer $n\Ge 2$, there exists an unbounded
closed symmetric operator\footnote{\;Recall that
symmetric operators are always subnormal (cf.\
\cite[Theorem 1 in Appendix I.2]{a-g}).} $S$ such
that $S^{n-1}$ is densely defined and
$\dz{S^n}=\{0\}$ (cf.\ \cite[Remark 4.6.3]{Schm};
see also \cite{Nai,Cher} for $n=2$). Then $S^n
\subseteq I$, but $S$ is not a normal operator.

In the rest of Appendix \ref{AppA} we assume that
$(X,\ascr,\mu)$ is a $\sigma$-finite measure
space. A transformation $\phi$ of $X$ is called
{\em $\ascr$-bimeasurable} if $\phi(\varDelta)
\in \ascr$ and $\phi^{-1}(\varDelta) \in \ascr$
for every $\varDelta\in \ascr$. The following
lemma is inspired by \cite[Proposition
4.1(vi)]{Bu-St2}.
   \begin{lem} \label{iloczyn}
If $\{\phi_j\}_{j=1}^n$ is a finite sequence of
bijective $\ascr$-bimeasurable nonsingular
transformations of $X$ such that $\phi_1 \circ \cdots
\circ \phi_n = \mathrm{id}_X$ and $n\Ge 2$, then
   \begin{align} \label{produkt}
\hsf_{\phi_1} \cdot \hsf_{\phi_2}\circ \phi_1^{-1}
\cdots \hsf_{\phi_n} \circ (\phi_1 \circ \cdots \circ
\phi_{n-1})^{-1} = 1 \text{ a.e.\ $[\mu]$.}
   \end{align}
   \end{lem}
   \begin{proof}
Applying the measure transport theorem repeatedly and
an induction argument, we get
   \allowdisplaybreaks
   \begin{align*}
\mu(\varDelta) & = \mu(\phi_n^{-1}((\phi_1 \circ
\cdots \circ \phi_{n-1})^{-1}(\varDelta)))
   \\
&= \int_X \chi_{\varDelta} \circ \phi_1 \circ \cdots
\circ \phi_{n-1} \cdot \hsf_{\phi_n} \circ
\phi_{n-1}^{-1} \circ \phi_{n-1} \D \mu
   \\
&= \int_X \chi_{\varDelta} \circ \phi_1 \circ \cdots
\circ \phi_{n-2} \cdot \hsf_{\phi_{n-1}} \cdot
\hsf_{\phi_n} \circ \phi_{n-1}^{-1} \D \mu
   \\
&= \int_X \chi_{\varDelta} \circ \phi_1 \circ \cdots
\circ \phi_{n-2} \cdot \hsf_{\phi_{n-1}} \circ
\phi_{n-2}^{-1} \circ \phi_{n-2} \cdot \hsf_{\phi_n}
\circ \phi_{n-1}^{-1} \circ \phi_{n-2}^{-1} \circ
\phi_{n-2} \D \mu
   \\
&= \int_X \chi_{\varDelta} \circ \phi_1 \circ \cdots
\circ \phi_{n-3} \cdot \hsf_{\phi_{n-2}} \cdot
\hsf_{\phi_{n-1}} \circ \phi_{n-2}^{-1} \cdot
\hsf_{\phi_n} \circ \phi_{n-1}^{-1} \circ
\phi_{n-2}^{-1} \D \mu
   \\
& \hspace{35ex}\vdots
   \\
&=\int_X \chi_{\varDelta} \cdot \hsf_{\phi_1} \cdot
\hsf_{\phi_2}\circ \phi_1^{-1} \cdots \hsf_{\phi_n}
\circ (\phi_1 \circ \cdots \circ \phi_{n-1})^{-1}
\D\mu, \quad \varDelta \in \ascr.
   \end{align*}
By the $\sigma$-finiteness of $\mu$, this implies
\eqref{produkt}.
   \end{proof}
We are now ready to prove the main result of
Appendix \ref{AppA}.
   \begin{pro} \label{potega}
If $\phi$ is a nonsingular transformation of $X$ such
that $\phi^n=\mathrm{id}_X$ for some integer $n\Ge 2$,
then the following conditions hold{\em :}
   \begin{enumerate}
   \item[(i)] $\phi^m$ is
a bijective and nonsingular transformation of $X$ for
every $m\in \zbb$,
   \item[(ii)] $\dz{C_{\phi}^{m}} = \dz{C_{\phi}^{n-1}}$
for every integer $m\Ge n$,
   \item[(iii)] $C_{\phi}^m =
C_{\phi}^r|_{\dzn{C_{\phi}}}$ for all $m,r\in \zbb_+$
such that $m \Ge n$ and $r \equiv m$ $($mod $n$$)$,
   \item[(iv)] $\dzn{C_{\phi}}$ is
a core for $C_{\phi}^m$ for every $m\in \zbb_+$,
   \item[(v)] $C_{\phi} \in \ogr{L^2(\mu)}$ if and only if
$C_{\phi}^n$ is closed,
   \item[(vi)] $C_{\phi}$ is subnormal if and only if
$C_{\phi}$ is unitary.
   \end{enumerate}
   \end{pro}
   \begin{proof}
(i) Since $\phi^n=\mathrm{id}_X$, the
transformation $\phi$ is bijective and
$\phi^{-1}=\phi^{n-1}$. This implies that $\phi$
is $\ascr$-bimeasurable and $\phi^{-1}$ is
nonsingular. Hence (i) is satisfied.

(ii) and (iii) follow from \cite[Proposition
14]{StSz4} and the equality
$\phi^n=\mathrm{id}_X$.

(iv) If $j\in \{1, \ldots,n\}$, then by Lemma
\ref{iloczyn}, applied to $n=2$, $\phi_1=\phi^j$ and
$\phi_2=\phi^{n-j}$, we deduce that $\hsf_{\phi^j} <
\infty$ a.e.\ $[\mu]$. In view of \cite[Corollary
4.5]{b-j-j-sC}, this implies that $C_{\phi}^n$ is
densely defined. Hence, by (ii), $\dzn{C_{\phi}}$ is
dense in $L^2(\mu)$. Applying \cite[Theorem
4.7]{b-j-j-sC} completes the proof of (iv).

(v) Suppose $C_{\phi}^n$ is closed. Since $C_{\phi}^n
\subseteq I$, we infer from (iv) that $C_{\phi}^n = I$
and so $\dz{C_{\phi}}=L^2(\mu)$. Hence, by the closed
graph theorem, $C_{\phi} \in \ogr{L^2(\mu)}$. The
reverse implication is obvious.

(vi) This condition follows from (iv) and Lemma
\ref{nilsub}.
   \end{proof}
   \begin{exa}
We will show that for every integer $n\Ge 3$,
there exists a nonsingular transformation $\phi$
of a discrete measure space $(X,\ascr,\mu)$ such
that $\phi^{n}=\mathrm{id}_X$ and
$\dz{C_{\phi}^{n-1}} \varsubsetneq
\dz{C_{\phi}^{n-2}} \varsubsetneq \ldots
\varsubsetneq \dz{C_{\phi}}$. By
\cite[Proposition 14]{StSz4}, it suffices to show
that $\dz{C_{\phi}^{n-1}} \varsubsetneq
\dz{C_{\phi}^{n-2}}$. Set $X=\zbb_+$ and
$\ascr=2^X$, and take a sequence
$\{\gamma_k\}_{k=0}^\infty \subset (0,\infty)$
tending to $\infty$. Define $\mu$ by $\mu(j + k
n) = \gamma_k^j$ for $j\in \{0, \ldots, n-1\}$
and $k\in \zbb_+$. Let $\phi$ be the
transformation of $X$ given by $\phi(j + k n) =
\widehat{j+1} + k n$ for $j\in\{0, \ldots, n-1\}$
and $k\in \zbb_+$, where $\widehat{j+1}=j+1$ if
$j+1 < n$ and $\widehat{j+1}=0$ if $j+1=n$. It is
clear that $\phi^n=\mathrm{id}_X$. Suppose that,
contrary to our claim, $\dz{C_{\phi}^{n-1}} =
\dz{C_{\phi}^{n-2}}$. Then, by \cite[Proposition
4.3]{b-j-j-sC}, there exists $c\in (0,\infty)$
such that $\hsf_{\phi^{n-1}} \Le c(1+
\sum_{l=1}^{n-2} \hsf_{\phi^l})$. Since
$\hsf_{\phi^{n-1}}(n-2+kn) = \gamma_k$ and
$\hsf_{\phi^l}(n-2+kn) = \gamma_k^{-l}$ for all
$l \in \{1,\ldots,n-2\}$ and $k\in \zbb_+$, we
arrive at the contradiction.
   \end{exa}
   \section{Symmetric composition operators}
We will show that symmetric composition operators are
selfadjoint and unitary.
   \begin{pro}\label{wco}
Let $(X,\ascr,\mu)$ be a $\sigma$-finite measure space
and $\phi$ be a nonsingular transformation of $X$. If
$C_{\phi}$ is symmetric, then $C_{\phi}$ is
selfadjoint and unitary, and $C_{\phi}^2=I$. If
$C_{\phi}$ is positive and symmetric, then
$C_{\phi}=I$.
   \end{pro}
   \begin{proof}
Since symmetric operators are formally normal, we
infer from \cite[Theorem 9.4]{b-j-j-sC} that if
$C_{\phi}$ is symmetric, then $C_{\phi}$ is normal and
consequently selfadjoint. For clarity, the rest of the
proof will be divided into two steps.

{\em Step} 1. If $C_{\phi}$ is positive and
selfadjoint, then $C_{\phi}=I$.

Indeed, by \cite[Proposition 6.2]{b-j-j-sC},
$C_{\phi}$ is injective. Since $C_{\phi}=|C_{\phi}|$,
the partial isometry $U$ in the polar decomposition of
$C_{\phi}$ is the identity operator on $L^2(\mu)$.
This together with \cite[Proposition
7.1(iv)]{b-j-j-sC} yields
   \begin{align} \label{zzizzu}
f\circ \phi = f \cdot \sqrt{\hsf_{\phi} \circ \phi}
\text{ a.e.\ $[\mu]$}, \quad f \in L^2(\mu).
   \end{align}
Take $\varDelta \in \ascr$ such that
$\mu(\varDelta) < \infty$. Substituting
$f=\chi_{\varDelta}$ into \eqref{zzizzu} and
using \eqref{hfi0}, we see that $\mu(\varDelta
\setminus \phi^{-1}(\varDelta)) =
\mu(\phi^{-1}(\varDelta) \setminus \varDelta )=0$
and thus $\mu(\varDelta) = (\mu \circ
\phi^{-1})(\varDelta)$. Since $\mu$ is
$\sigma$-finite, we conclude that $\mu = \mu
\circ \phi^{-1}$. Therefore $\hsf_{\phi} = 1$
a.e.\ $[\mu]$. By \cite[Proposition
7.1(i)]{b-j-j-sC}, $C_{\phi}=|C_{\phi}|$ is the
operator of multiplication by $\hsf_{\phi}^{1/2}$
and thus $C_{\phi}=I$.

   {\em Step} 2. If $C_{\phi}$ is selfadjoint,
then $C_{\phi}$ is unitary and $C_{\phi}^2=I$.

Indeed, by \cite[Theorem 7.19]{Weid}, $C_{\phi}^2$ is
selfadjoint. Hence $C_{\phi}^2$ is closed. By
\cite[Corollary 4.2]{b-j-j-sC} (with $n=2$), we have
   \begin{align}   \label{fifi*}
C_{\phi}^*C_{\phi}= C_{\phi}C_{\phi}^* =C_{\phi}^2 =
\overline{C_{\phi}^2} = C_{\phi^2},
   \end{align}
which means that $C_{\phi^2}$ is positive and
selfadjoint. It follows from Step 1 that
$C_{\phi^2}=I$. Therefore, by \eqref{fifi*}, $C_{\phi}$
is unitary (see also Lemma \ref{nilsub}) and
$C_{\phi}^2=I$.

Putting this all together completes the proof.
   \end{proof}
Adapting \cite[Example 3.2]{Bu-St1} to the present
context, one can show that the equality $C_{\phi}=I$
does not imply that $\phi=\mathrm{id}_X$ a.e.\
$[\mu]$. It may even happen that the set $\{x \in
X\colon \phi(x)=x\}$ is not $\ascr$-measurable.
   \begin{exa}
We will show that there exists a selfadjoint
composition operator which is not positive. Set
$X=\zbb_+$ and $\ascr=2^X$. Consider a measure $\mu$ on
$\ascr$ such that $0 < \mu(2k)=\mu(2k+1) < \infty$ for
all $k\in \zbb_+$, and the transformation $\phi$ of $X$
given by $\phi(2k)=2k+1$ and $\phi(2k+1)=2k$ for $k\in
\zbb_+$. Then $\phi^2 = \mathrm{id}_X$ and consequently
$\phi^{-1}=\phi$. It is clear that $\hsf_{\phi}=1$ and
thus $C_{\phi} \in \ogr{L^2(\mu)}$. Since
$\phi^{-1}(\ascr) = \ascr$, we deduce from
\cite[Corollary 7.3 and Remark 7.4]{b-j-j-sC} that
$C_{\phi}^* f = \hsf_{\phi} \cdot f\circ \phi^{-1} = f
\circ \phi = C_{\phi} f$ for all $f\in L^2(\mu)$. Hence
$C_{\phi}$ is selfadjoint. Since $C_{\phi}f = - f$,
where $f(l)=(-1)^l \chi_{\{0,1\}}(l)$ for $l\in\zbb_+$,
the operator $C_{\phi}$ is not positive.
   \end{exa}
   \section{\label{AppC}Orthogonal sums of composition
operators} Let $(X, \ascr, \mu)$ be a $\sigma$-finite
measure space and $\phi$ be a nonsingular
transformation of $X$. Define $\ascr(\phi)=\{Y \in
\ascr\colon \phi(Y) \subseteq Y \text{ and }
\phi(X\setminus Y) \subseteq X\setminus Y\}$. Since
$\ascr(\phi) = \{Y \in \ascr\colon \phi^{-1}(Y)=Y\}$,
$\ascr(\phi)$ is a $\sigma$-algebra. For nonempty $Y
\in \ascr(\phi)$, we set $\ascr_{Y}=\{\varDelta \in
\ascr \colon \varDelta \subseteq Y\}$, $\mu_{Y} =
\mu|_{\ascr_{Y}}$ and $\phi_{Y} = \phi|_{Y}$. Clearly,
$(Y, \ascr_{Y}, \mu_{Y})$ is a $\sigma$-finite measure
space and $\phi_{Y}$ is a nonsingular transformation
of $Y$. Given $N\in \nbb \cup \{\infty\}$, we write
$J_N$ for the set of all integers $n$ such that $1\Le
n \Le N$.
   \begin{pro} \label{orthsum}
Suppose $N\in \nbb \cup \{\infty\}$ and
$\{Y_n\}_{n=1}^N \subseteq \ascr(\phi)$ is a
sequence of pairwise disjoint nonempty sets. Set
$Y=\bigcup_{n=1}^N Y_n$. Then the following
holds{\em :}
   \begin{enumerate}
   \item[(i)] $\chi_{Y_n} L^2(\mu)$
reduces $C_{\phi}$ and $C_{\phi}|_{\chi_{Y_n}
L^2(\mu)}$ is unitarily equivalent to $C_{\phi_{Y_n}}$
for every $n \in J_N$,
   \item[(ii)] $C_{\phi}|_{\chi_{Y} L^2(\mu)} =
\bigoplus_{n=1}^N C_{\phi}|_{\chi_{Y_n} L^2(\mu)}$,
   \item[(iii)] $C_{\phi}|_{\chi_{Y} L^2(\mu)}$ is unitarily
equivalent to $\bigoplus_{n=1}^N C_{\phi_{Y_n}}$.
   \end{enumerate}
   \end{pro}
   \begin{proof}
Since the orthogonal projection $P_{Y_n}$ of
$L^2(\mu)$ onto $\chi_{Y_n} L^2(\mu)$ is given by
$P_{Y_n} (f) = \chi_{Y_n} \cdot f$ for $f \in
L^2(\mu)$, we see that $(P_{Y_n} f) \circ \phi =
P_{Y_n} (f \circ \phi)$ for all $f \in
\dz{C_{\phi}}$. Hence $P_{Y_n} C_{\phi} \subseteq
C_{\phi} P_{Y_n}$. The rest of the proof of (i)
is straightforward. Since $\chi_{Y} L^2(\mu) =
\bigoplus_{n=1}^N \chi_{Y_n} L^2(\mu)$, (ii)
follows from (i) and the fact that $C_{\phi}$ is
closed. Finally, (iii) is a direct consequence of
(i) and (ii).
   \end{proof}
   \begin{cor} \label{orthsum-c}
An orthogonal sum of countably many composition
operators in $L^2$-spaces is unitarily equivalent to a
composition operator in an $L^2$-space.
   \end{cor}
   \begin{proof}
Let $\{(X_n, \ascr_n, \mu_n)\}_{n=1}^N$ be a
sequence of $\sigma$-finite measure spaces and
$\{\phi_n\}_{n=1}^N$ be a sequence of nonsingular
transformations $\phi_n$ of $X_n$, where $N\in
\nbb \cup \{\infty\}$. Set $X=\bigcup_{n=1}^N X_n
\times \{n\}$, $\ascr = \big\{\bigcup_{n=1}^N
\varDelta_n \times \{n\} \colon \varDelta_n \in
\ascr_n \; \forall n\in J_N \big\}$ and
$\mu(\varDelta) = \sum_{n=1}^N \mu_n
(\varDelta_n)$ for $\varDelta=\bigcup_{n=1}^N
\varDelta_n \times \{n\}$ ($\varDelta_n \in
\ascr_n$). Define the transformation $\phi$ of
$X$ by $\phi((x,n))=(\phi_n(x),n)$ for $x\in X_n$
and $n\in J_N$. Then $(X,\ascr,\mu)$ is a
$\sigma$-finite measure space and $\phi$ is
nonsingular. Applying Proposition \ref{orthsum}
to $Y_n:=X_n \times \{n\}$, we deduce that
$\bigoplus_{n=1}^N C_{\phi_n}$ is unitarily
equivalent to $C_{\phi}$.
   \end{proof}
   \subsection*{Acknowledgement} A substantial part
of this paper was written while the first, the second
and the fourth authors visited Kyungpook National
University during the spring and the autumn of 2012
and the spring of 2013. They wish to thank the faculty
and the administration of this unit for their warm
hospitality.

   \bibliographystyle{amsalpha}
   
   \end{document}